\documentclass[11pt,reqno]{amsart}

\usepackage[utf8]{inputenc}
\usepackage{url}
\usepackage{graphicx}
\usepackage{amssymb,amsmath,amsthm,amsfonts}
\usepackage{bbm}
\usepackage{a4wide}
\usepackage{enumerate}
\usepackage{eucal}
\usepackage[usenames,dvipsnames]{color}
\usepackage{bookmark}
\usepackage{hyperref}
\usepackage{mathrsfs}

\usepackage{calligra}
\usepackage{wela}
\usepackage[T1]{fontenc}

\usepackage{bookmark}
\usepackage{hyperref}
\hypersetup{pdfstartview={FitH}}

\newtheorem{theorem}{Theorem}
\newtheorem{lemma}[theorem]{Lemma}
\newtheorem{corollary}[theorem]{Corollary}

\newtheorem{proposition}[theorem]{Proposition}
\newtheorem{preremark}[theorem]{Remark}  \newenvironment{remark}
{\begin{preremark}\rm}{\end{preremark}}

\numberwithin{theorem}{section}
\numberwithin{equation}{section}

\font\gotviii=eufm10 at 8pt
\font\gotxi=eufm10 at 11pt
\font\posebni=msam10

\newcommand{\siC}{\textsf{\textit{C}}}

\newcommand{\C}{\mathbb{C}}
\newcommand{\N}{\mathbb{N}}
\newcommand{\R}{\mathbb{R}}

\newcommand\bA{\mathbf A}
\newcommand\bS{\mathbf{S}}

\newcommand{\cA}{\mathcal A}
\newcommand{\cB}{\mathcal B}
\newcommand{\cC}{\mathcal C}
\newcommand{\cD}{\mathcal D}
\newcommand{\cE}{\mathcal E}
\newcommand{\cF}{\mathcal F}
\newcommand{\cG}{\mathcal G}
\newcommand{\cH}{\mathcal H}
\newcommand{\cI}{\mathcal I}
\newcommand{\cM}{\mathcal M}
\newcommand{\cP}{\mathcal P}
\newcommand{\cQ}{\mathcal Q}
\newcommand{\cR}{\mathcal R}
\newcommand{\cS}{\mathcal S}
\newcommand{\cV}{\mathcal V}
\newcommand{\cW}{\mathcal W}

\newcommand{\su}{\mathsf u}
\newcommand{\sv}{\mathsf v}
\newcommand{\sw}{\mathsf w}
\newcommand{\sA}{\mathsf A}
\newcommand{\sB}{\mathsf B}
\newcommand{\sC}{\mathsf C}

\newcommand{\tf}{\mathtt f}
\newcommand{\tg}{\mathtt g}
\renewcommand{\th}{\mathtt h}

\newcommand{\oU}{{\mathscr U}}
\newcommand{\oV}{{\mathscr V}}
\newcommand{\oW}{{\mathscr W}}

\renewcommand{\a}{\text{\gotxi a}}
\newcommand{\gota}{\text{\gotxi a}}

\newcommand{\gX}{\text{\gotxi X}}
\newcommand{\gmX}{\text{\gotviii X}}
\newcommand{\gZ}{\text{\gotxi Z}}

\newcommand{\gH}{\text{\gotxi H}}

\newcommand{\gm}{\text{\gotxi m}}

\def\Dom{\cD}
\def\Ran{\operatorname{R}}

\newcommand\bena{\mathbf 1}
\newcommand{\e}{\varepsilon}
\newcommand{\f}{\varphi}
\newcommand{\leqsim}{\,\text{\posebni \char46}\,}
\newcommand{\geqsim}{\,\text{\posebni \char38}\,}
\newcommand{\sign}{\mathop{\rm sign}\nolimits}
\newcommand{\wt}{\widetilde}
\newcommand\wrt{\,\textup{d}}
\newcommand{\nor}[1]{\left\| #1 \right\|}
\newcommand\norm[2]{{\left\Vert{#1}\right\Vert_{#2}}}
\newcommand{\sk}[2]{\left\langle #1 , #2\right\rangle}
\newcommand{\mn}[2]{\left\{ #1\, ;\, #2 \right\}}

\renewcommand{\div}{{\rm div}}
\renewcommand{\leq}{\leqslant}
\renewcommand{\geq}{\geqslant}
\renewcommand{\Re}{{\rm Re}\,}
\renewcommand{\Im}{{\rm Im}\,}
\renewcommand\mod[1]{\left\vert{#1}\right\vert}

\begin{document}

\title[Trilinear embedding for divergence-form operators]{Trilinear embedding for divergence-form operators with complex coefficients}

\author[A. Carbonaro]{Andrea Carbonaro}
\address{Andrea Carbonaro, Universit\`{a} degli Studi di Genova, Dipartimento di Matematica, Via Dodecaneso 35, 16146 Genova, Italy}
\email{carbonaro@dima.unige.it}

\author[O. Dragičević]{Oliver Dragičević}
\address{Oliver Dragičević, Department of Mathematics, Faculty of Mathematics and Physics, University of Ljubljana,  Jadranska 19, SI-1000 Ljubljana, Slovenia, and Institute of Mathematics, Physics and Mechanics, Jadranska 19, SI-1000 Ljubljana, Slovenia}
\email{oliver.dragicevic@fmf.uni-lj.si}

\author[V. Kovač]{Vjekoslav Kovač}
\address{Vjekoslav Kova\v{c}, Department of Mathematics, Faculty of Science, University of Zagreb, Bijeni\v{c}ka cesta 30, 10000 Zagreb, Croatia}
\email{vjekovac@math.hr}

\author[K. A. Škreb]{Kristina Ana Škreb}
\address{Kristina Ana \v{S}kreb, Faculty of Civil Engineering, University of Zagreb, Fra Andrije Ka\v{c}i\'{c}a Mio\v{s}i\'{c}a 26, 10000 Zagreb, Croatia}
\email{kskreb@grad.hr}

% \date{\today}

% \thanks{{\it Version}: ``Trilinear\_ArXiv3-2.tex''}

\subjclass[2010]{Primary 35J15; %Second-order elliptic equations
Secondary 42B20, %Singular and oscillatory integrals
47D06% \hfill Trilinear\_ArXiv3-1.tex
} %One-parameter semigroups and linear evolution equations

%\keywords{...}

\begin{abstract}
We prove a dimension-free $L^p(\Omega)\times L^q(\Omega)\times L^r(\Omega)\rightarrow L^1(\Omega\times (0,\infty))$
embedding for triples of elliptic operators in divergence form with complex coefficients and subject to mixed boundary conditions on $\Omega$, and for triples of exponents $p,q,r\in(1,\infty)$ mutually related by the identity $1/p+1/q+1/r=1$. Here $\Omega$ is
allowed to be an arbitrary open subset of $\R^d$.
Our assumptions involving the exponents and coefficient matrices are  
expressed in terms of a condition known as $p$-ellipticity. 
 The proof utilizes the method of Bellman functions and heat flows.
As a corollary, we give applications to (i) paraproducts and (ii) square functions associated with the corresponding operator semigroups, moreover, we prove (iii) inequalities of Kato--Ponce type for elliptic operators with complex coefficients. 
All the above results are the first of their kind for elliptic divergence-form operators with complex coefficients on arbitrary open sets.
Furthermore, the approach to (ii),(iii) through trilinear embeddings seems to be new.
\end{abstract}

\maketitle

%%%%%%%%%%%%%%%%%%%%%%%%%%%%%%%%%%%%%%%%%%%%%%%%%%%%%%%%%%%

\section{Introduction and statement of the main results}
Let $\Omega\subseteq\R^d$ be an arbitrary open set.
Denote by $\cA(\Omega)$ the family of all complex {\it uniformly strictly accretive} (also called {\it elliptic}) $d\times d$ matrix functions on $\Omega$ with $L^{\infty}$ coefficients.
That is, $\cA(\Omega)$ is the set of all measurable $A:\Omega\rightarrow\C^{d\times d}$ for which
there exist $\lambda,\Lambda>0$ such that for almost all $x\in\Omega$ we have
\begin{eqnarray}
\label{eq: ellipticity}
\Re\sk{A(x)\xi}{\xi}
&\hskip -19pt\geq \lambda|\xi|^2\,,
&\hskip 20pt\forall\xi\in\C^{d};
\\
\label{eq: boundedness}
\mod{\sk{A(x)\xi}{\eta}}
&\hskip-6pt\leq \Lambda \mod{\xi}\mod{\eta}\,,
&\hskip 20pt\forall\xi,\eta\in\C^{d}.
\end{eqnarray}
Elements of $\cA(\Omega)$ will also more simply be referred to as {\it accretive}
or {\it elliptic matrices}.
For any $A\in\cA(\Omega)$ denote by
$\lambda(A)$
the largest admissible $\lambda$ in \eqref{eq: ellipticity} and by
$\Lambda(A)$
the smallest $\Lambda$ in \eqref{eq: boundedness}.

\subsection{The $p$-ellipticity condition}
\label{s: Tendinitis Achilaris dex}
The concept of $p$-ellipticity was introduced by the first two authors of the present paper in \cite{CD-DivForm} as follows.

Given $A\in\cA(\Omega)$ and $p\in (1,\infty)$, we say that $A$ is {\it $p$-elliptic} if
$\Delta_{p}(A)>0$, where
\begin{equation}
\label{eq: kabuto}
\Delta_{p}(A):=
\underset{x\in\Omega}{{\rm ess}\inf}
\min_{\substack{\xi\in\C^{d}\\ |\xi|=1}}
\Re\sk{A(x)\xi}{\xi+|1-2/p|\bar\xi}_{\C^{d}}.
\end{equation}
Equivalently, $A$ is $p$-elliptic if
there exists $c=c(A,p)>0$ such that for a.e. $x\in\Omega$,
\begin{equation}
\label{eq: Sparky 21}
\Re\sk{A(x)\xi}{\xi+|1-2/p|\bar\xi}_{\C^{d}}
\geq c |\xi|^2\,,
\hskip30pt
 \forall\xi\in\C^{d}.
\end{equation}
It follows straight from \eqref{eq: kabuto}
that $\Delta_{p}$ is invariant under conjugation of $p$, meaning that $\Delta_{p}(A)=\Delta_{q}(A)$ when $1/p+1/q=1$.
Furthermore, note that $\Delta_2(A)=\lambda(A)$, so $p$-ellipticity generalizes the notion of classical ellipticity.
We will refer to $\Delta_p(A)$ and $\Lambda(A)$ collectively as the {\it $p$-ellipticity constants} of $A$.
In order to unify some computations, we extend the definition \eqref{eq: kabuto} to all $p\in(0,\infty]$.

Denote by $\cA_{p}(\Omega)$ the class of all $p$-elliptic matrix functions on $\Omega$. It is known, see \cite{CD-DivForm}, that
$\mn{\cA_{p}(\Omega)}{p\in[2,\infty)}$
is a decreasing chain of matrix classes such that
%$$
%\begin{array}{rcc}
%\{\text{elliptic matrices on }\Omega\} & = & \cA_2(\Omega)\,, \\
%\{\text{real elliptic matrices on }\Omega\} & = &
%{\displaystyle\bigcap_{p\in[2,\infty)}\cA_{p}(\Omega)}\,.
%\end{array}
%$$
\begin{align}
\{\text{elliptic matrices on }\Omega\} & =\ \  \cA_2(\Omega)\,, \nonumber\\
\{\text{real elliptic matrices on }\Omega\} & = 
{\displaystyle\bigcap_{p\in[2,\infty)}\cA_{p}(\Omega)}\,.
\label{eq: real elliptic}
\end{align}

In \cite{CD-DivForm} it was argued that $p$-ellipticity could be of interest for the $L^p$-theory of elliptic PDEs with complex coefficients. So far, several examples have been found that support this thought.
The current paper aims to 
continue further in this direction,
by establishing new applications of $p$-ellipticity: trilinear embeddings, paraproducts, square function estimates and Kato--Ponce inequalities.

 A condition similar to \eqref{eq: Sparky 21}, yet slightly weaker,
 was formulated in a different manner by Cialdea and Maz'ya in \cite[(2.25)]{CM}; see \cite[Remark 5.14]{CD-DivForm}. It was a result of their study of a condition on sesquilinear forms known as $L^{p}$-dissipativity. In \cite{CD-DivForm} the first two authors of the present paper arrived at 
 \eqref{eq: Sparky 21}
 from a different direction, that is, via bilinear embeddings and generalized convexity of power functions; see Section \ref{s: Radu Simion Geamparalele} for more background.

While the first two authors of the present paper were preparing \cite{CD-DivForm}, M. Dindo\v s and J. Pipher were working on their own article \cite{DiPi}. They
discovered remarkable connections between \eqref{eq: Sparky 21} and the regularity theory of elliptic PDEs.
More precisely, they
found the following sharp condition which implies reverse H\"older inequalities for weak solutions of elliptic operators in divergence form with complex coefficients:
for some $\e=\e(A,p)>0$ and almost all $x\in\Omega$,
$$
\sk{\Re A(x)\lambda}{\lambda}_{\R^d}+\sk{\Re A(x)\eta}{\eta}_{\R^d}+\sk{\big(\sqrt{p'/p}\,\Im A(x)-\sqrt{p/p'}\,\Im A(x)^T\big)\lambda}{\eta}_{\R^d}
\geqslant\e\big(|\lambda|^2+|\eta|^2\big)
$$ 
for all $\lambda,\eta\in\R^d$. Here $p'=p/(p-1)$ is the conjugate exponent of $p$. 
It turned out that the above condition of theirs, devised
independently of \cite{CD-DivForm}, namely, as a strengthening of \cite[(2.25)]{CM}, was precisely a reformulation of \eqref{eq: Sparky 21}.
The same authors have since then been successfully continuing their line of exploration of $p$-ellipticity in PDEs; see their recent paper \cite{DiPi20} and a preprint with Li \cite{DiLiPi20}.

\subsection{Genesis of $p$-ellipticity}
\label{s: Radu Simion Geamparalele}
The idea of attaching a {\it number} to the pair $(A,p)$,
highlighting positivity of that number as a key condition,
writing it as in \eqref{eq: kabuto},
studying its dynamics with respect to $A$ and $p$,
etc.,
came as a synthesis of the first two authors' long-term study of {\it Bellman functions, heat flows} and {\it generalized convexity}.

Among the very first works where the Bellman-heat approach started getting developed were Petermichl--Volberg \cite{PV} and Nazarov--Volberg \cite{NV}.
Afterwards, in a series of papers written by A. Volberg and the second author of the present paper \cite{DV-Kato, DV-Sch, DV} or by the first two authors of the present paper \cite{CD-mult, CD-OU, CD-Riesz}, the heat flow method associated with a particular Nazarov--Treil function $Q$ (an example of a Bellman function), found in \cite{NV}, was applied and developed further.

For example, in \cite{DV-Sch} a scrutinous analysis of $Q$ revealed fine convexity properties that were, on one hand, indispensable for the main goal of \cite{DV-Sch} (a so-called {\it bilinear embedding}), and unexpected on the other. That suggested that $Q$ may possess further important convexity properties, which was subsequently confirmed in above-cited works of the first two authors.
Another milestone on that path was \cite{CD-mult}, where the flow associated with $Q$ was for the first time studied for {\it complex times}. This led to a universal spectral multiplier theorem for generators of symmetric contractive semigroups with the optimal angle, thus introducing the Bellman-heat method into spectral multipliers and answering a question posed several years earlier
by Stein.

As indicated above,
the gist of each of those works was establishing the convexity of $Q$.
In \cite{CD-DivForm} the concept of convexity of a function with respect to a matrix was introduced (see Section \ref{s: cajkovskij grand sonata} for the definition).
When the matrix in question is the identity matrix, this is the usual convexity.
In the same work the focus was fully shifted from the convexity of $Q$ to the convexity of its principal building blocks -- power functions. The $p$-ellipticity condition \eqref{eq: Sparky 21} was conceived there first as the uniform convexity of power functions $|\zeta|^p$ with respect to $A$, and then reshaped into \eqref{eq: Sparky 21}; see \cite[Remark 5.9]{CD-DivForm}.
Therefore the $p$-ellipticity emerged in \cite{CD-DivForm} after several years of gradually distilling the heat flow method through \cite{DV-Kato, DV-Sch, DV, CD-mult, CD-OU, CD-Riesz}.

Other works using the Bellman-heat approach include
Domelevo--Petermichl \cite{DoPe},
Mauceri--Spinelli \cite{MaSp},
Wr\'obel \cite{W18a},
Betancor--Dalmasso--Fari\~na--Scotto \cite{BDFS},
Kucharski \cite{K21},
the article \cite{KS18} by the latter two authors of the present paper, 
and \cite{CD-Mixed} by the former two.

\bigskip
Since we will be dealing with pairs and triples of matrices, it is useful to introduce further notation, as in \cite{CD-DivForm, CD-Mixed}:
\begin{equation*}
\aligned
\Delta_{p}(A_1,\hdots,A_{N})&=\min\{\Delta_{p}(A_1),\hdots,\Delta_{p}(A_{N})\}, \\
\lambda(A_1,\hdots,A_{N})&=\min\{\lambda(A_1),\hdots, \lambda(A_{N})\}, \\
\Lambda(A_1,\hdots,A_{N})&=\max\{\Lambda(A_1),\hdots, \Lambda(A_{N})\}.
\endaligned
\end{equation*}

\subsection{Elliptic partial differential operators in divergence form}
\label{s: Heian}

Suppose that either:
\begin{enumerate}[(a)]
\item
$\oU=H_0^1(\Omega)$,
\item
$\oU=H^1(\Omega)$, or
\item
$\oU$ is the closure in $H^1(\Omega)$ of the set of restrictions $\mn{u|_\Omega}{u\in C_c^\infty(\R^d\backslash\Gamma)}$, where $\Gamma$ is a (possibly empty) closed subset of $\partial\Omega$.
\end{enumerate}
Here $H_0^1(\Omega)$ stands for the closure of $C_c^\infty(\Omega)$ in the Sobolev space $H^1(\Omega)=W^{1,2}(\Omega)$.
Recall that $H_0^1(\R^{d})=H^1(\R^{d})$; see \cite[Corollary~3.19]{Adams} for a reference.
\label{koronavirus - skoro na rivu}

We would like to define the {\it divergence form} operator $L_{A}u=-\div(A\nabla u)$.
A standard way of achieving this is to use sesquilinear forms. Before proceeding we state that all the integrals in this paper will be taken with respect to the Lebesgue measure. As the ambient space we will typically take the complex Hilbert space $\cH=L^2(\Omega)$.

Let the sesquilinear form $\a=\a_{A,\oU}$ be given by $\cD(\a)$ and
\begin{equation*}
\a(u,v):=\int_\Omega\sk{A\nabla u}{\nabla v}_{\C^d}\,
\hskip 40pt
\text{ for }u,v\in\oU.
\end{equation*}
We define $L=L_A=L_{A,\oU}$ to be the operator associated with $\a_{A,\oU}$. See \cite[Section 1.2.3]{O} for information about this construction.
The bottomline is that
\begin{equation}
\label{eq: int by parts}
\sk{L_{A} u}{v}_{\cH}=
\int_\Omega\sk{A\nabla u}{\nabla v}_{\C^{d}},
\hskip 20pt
\forall u\in\cD(L_A),
v\in\oU,
\end{equation}
where the domain $\mathcal{D}(L_A)$ is the set of all $u\in \oU$ for which the right-hand side, regarded as an antilinear functional on $\oU$ with input $v$, extends boundedly to the whole $L^2(\Omega)$.
Depending on the choice (a)--(c) of $\oU$, we say that $L$ is subject to (a) {\it Dirichlet}, (b) {\it Neumann} or (c) {\it mixed boundary conditions}, see \cite[Section 4.1]{O}.

Consider the operator semigroup  on $L^2(\Omega)$ generated by $-L_{A,\oU}$:
\[
T^{A,\oU}_t := \exp(-t L_{A,\oU})\quad \text{for } t\in(0,\infty).
\]
The semigroup $(T^{A,\oU}_{t})_{t>0}$ is known to be contractive and analytic on $L^{2}(\Omega)$; see \cite[Chapter VI]{Kat}, \cite{AuTc} and \cite[Chapters 1 and 4]{O}.

\subsection{Trilinear embedding}
The main result of this paper is the following dimension-free \emph{trilinear embedding theorem}.
Recall that by the $p$-ellipticity constants of $A$ we mean the numbers $\Delta_p(A)$, $\lambda(A)$, $\Lambda(A)$.

\begin{theorem}
\label{t: trilinemb}
Let $\Omega\subseteq\R^d$ be an open set and let the spaces $\oU,\oV,\oW$ be as in Section \ref{s: Heian}.
Take $p,q,r\in(1,\infty)$ such that
\begin{equation}\label{eq:Hoeldexppqr}
\frac{1}{p} + \frac{1}{q} + \frac{1}{r} = 1.
\end{equation}
Suppose that accretive matrices
$A,B,C:\Omega\rightarrow\C^{d\times d}$
are $\max\{p,q,r\}$-elliptic.
Then for
$f\in\big(L^p\cap L^2\big)(\Omega)$,
$g\in\big(L^q\cap L^2\big)(\Omega)$
and
$h\in\big(L^r\cap L^2\big)(\Omega)$
we have
\begin{equation}
\label{eq: trilinemb}
\int_0^\infty\int_{\Omega}
\left| \nabla T_t^{A,\oU} f \right|
\left| \nabla T_t^{B,\oV} g \right|
\left| T_t^{C,\oW} h \right|
\wrt x\wrt t
\ \leqsim\
\nor{f}_{p}\nor{g}_{q}\nor{h}_{r}.
\end{equation}
When $\Omega=\R^d$, the same conclusion holds
under milder assumptions, namely, when
\begin{itemize}
\item
$A$ is $p$-elliptic and $(1+p/q)$-elliptic,
\item
$B$ is $q$-elliptic and $(1+q/p)$-elliptic,
\hfill
{\rm({\Large $\star$})}\label{star}
\item
$C$ is $r$-elliptic.
\end{itemize}

The implied embedding constants only depend on $p,q,r$ and $*$-ellipticity constants of $A,B,C$ alluded to in the theorem's assumptions. % , and do not depend on the dimension $d$.
\end{theorem}

The proof of Theorem~\ref{t: trilinemb} will be given in Section \ref{s: Extracorporeal Shock Wave Therapy} (case of $\Omega=\R^d$) and Section \ref{s: general case} (for general $\Omega$), after all the needed preparatory results are established in Sections~\ref{s: prelims}--\ref{s: RICE}.

\begin{remark}
Introduce, for $p\in(1,\infty)$, the notation
\begin{equation*}
\label{eq: p*}
p^*:=\max\{p,p'\},
\end{equation*}
where $p'$ is the conjugate exponent of $p$, that is, $1/p+1/p'=1$. Owing to the symmetry and monotonicity properties of $\Delta_{p}(A)$ specified in Section \ref{s: Tendinitis Achilaris dex},
we have that
$$
A\text{ is }s\text{-elliptic and }t\text{-elliptic }
\hskip 15pt
\Longleftrightarrow
\hskip 15pt
A\text{ is }\max\{s^*,t^*\}\text{-elliptic}.
$$
In particular, in the special case $p=q$ the assumptions ({\Large $\star$})
read that $C$ is $r$-elliptic, while $A,B$ are $p$-elliptic for $2/p+1/r=1$.

Observe also that $1+p/q$ and $1+q/p$ are conjugate exponents, therefore $(1+p/q)$-ellipticity is the same as $(1+q/p)$-ellipticity.
Furthermore, note that $1+p/q=p/r'$ and $1+q/p=q/r'$.
\end{remark}

\subsection{Motivation. Bilinear versus trilinear}
A {\it bilinear embedding} is, roughly, an estimate of the following type:
$$
\int_0^\infty\int_X|\nabla T_tf||\nabla T_tg|\wrt \mu(x)\wrt t\,\leqsim\,\nor{f}_p\nor{g}_q,
$$
where $(T_t)_{t>0}$ is an operator semigroup acting on complex functions $f,g$ which map from some measure space $(X,\mu)$, and $1/p+1/q=1$. Variants of bilinear embeddings have been instrumental in proving an array of sharp results, e.g. Riesz transform estimates, general spectral multiplier theorems, maximal regularity etc.; see \cite{CD-DivForm, CD-Mixed, CD-mult} for the relevant literature. The first two authors of the present paper proved bilinear embeddings for elliptic operators in divergence form with complex coefficients \cite{CD-DivForm, CD-Mixed}; see Theorem \ref{t: N bil} for the result's statement. 
Our colleagues C. Thiele and J. Bennett have asked us
whether bilinear estimates in the context of elliptic matrices admit a reasonable {\it trilinear} counterpart. Our answer to these inquiries is Theorem \ref{t: trilinemb}.
Apart from pushing the Bellman-heat technique further, this result and its applications described in Section \ref{s: Applications} extend the scope of $p$-ellipticity. We believe that Theorem~\ref{t: trilinemb} bears potential for other applications in the future.

The special case $\Omega=\R^d$ and $A=B=C=I$ of Theorem \ref{t: trilinemb} follows by applying classical tools from harmonic analysis. The latter two authors of the present paper gave an alternative proof for $d=1$ in \cite[Section 3.3]{KS18}, using a Bellman-heat argument and utilizing an original function constructed in the same paper. Hence that function will be a natural starting point for the proof of our Theorem \ref{t: trilinemb}.

On $\R^d$, embedding-type theorems for general complex elliptic matrices were first studied by the first two authors of the present paper \cite{CD-DivForm}.
As mentioned above, their embedding differs from
the one in this paper 
by being bilinear.
It was however also based on a Bellman-heat argument, % of course 
yet 
associated with another (simpler) function.
Recently, the same authors proved bilinear embedding for complex-coefficient operators under mixed boundary conditions on  arbitrary open sets $\Omega\subseteq \R^d$, see \cite{CD-Mixed, CD-Potentials}. Passing from $\R^d$ to general open sets $\Omega$ was technically demanding and called for a major modification of the approach from \cite{CD-DivForm}. This is also the reason why the two cases of Theorem \ref{t: trilinemb} (namely, $\Omega=\R^d$ and $\Omega\subsetneq\R^d$) will be treated in separate sections.

On top of grossly extending \cite[Section 3.3]{KS18}, the trilinear embedding of Theorem \ref{t: trilinemb} also directly implies the above-said bilinear embeddings from \cite{CD-DivForm,CD-Mixed}. See Section \ref{s: One Penny} for a more precise formulation of this statement and its proof.

The main challenge of the {\it trilinear} context when compared to the bilinear one is that symmetry is lost, i.e., knowing $p$ does not yet determine any of the other two exponents, unlike in the bilinear case with a pair of mutually conjugate exponents. On the other hand, the generalized convexity coefficient $\Delta_p(A)$ is naturally fit for conjugation, as it is actually invariant under conjugation of $p$.
This is what makes the trilinear case significantly different and more difficult.

Thus the proof of Theorem~\ref{t: trilinemb} is
based on combining the elements from \cite{KS18} on one hand, and the technique from \cite{CD-DivForm} (for $\Omega=\R^d$) and \cite{CD-Mixed} (for general $\Omega$) on the other.

\subsection{Trilinear implies bilinear}
\label{s: One Penny}

Here we present a simple argument demonstrating that our {\it trilinear} embedding (Theorem \ref{t: trilinemb}) implies the following {\it bilinear} estimate. 
The latter
appeared in \cite{CD-DivForm} in the special case of $\Omega=\R^d$ and in \cite{CD-Mixed} with general open $\Omega\subseteq\R^d$. We formulate it here for the sake of convenience.

\begin{theorem}
\cite{CD-DivForm,CD-Mixed}
\label{t: N bil}
Let $\oU, \oV\subseteq H^1(\Omega)$ be as in Section~\ref{s: Heian}. If $p,q>1$ are conjugate exponents and $A,B\in\cA_p(\Omega)$, then for all $f,g\in (L^{p}\cap L^{q})(\Omega)$ we have
\begin{equation}
\label{eq: N bil}
\int^{\infty}_{0}\!\int_{\Omega}\mod{\nabla T^{A,\oU}_{t}f(x)}\mod{\nabla T^{B,\oV}_{t}g(x)}\wrt x\wrt t\,\leqsim\,\norm{f}{p}\norm{g}{q},
\end{equation}
with the implied constants depending only on $p$, $\Delta_{p}(A,B)$, $\lambda(A,B)$ and $\Lambda(A,B)$.
\end{theorem}

Let $p,q,A,B,\Omega,\oU,\oV,f,g$ satisfy the assumptions of Theorem \ref{eq: N bil}. We further require that $p\ne q$, hence, by symmetry of \eqref{t: N bil} with respect to the interchange $p\leftrightarrow q$, we may assume that $p>2>q$. Also, suppose that either $\Omega=\R^d$ or $A,B$ are real.

Choose
$0<\e<1-(p-1)^{-1}$.
Then $q+\e<2<p-\e$ and $(p-\e)^{-1}+(q+\e)^{-1}<1$,
so that there exists a unique $r=r(\e)>1$ determined by
\begin{equation}
\label{eq: Dallas vs. Clippers Game 6}
\frac{1}{p-\e}+\frac{1}{q+\e}+\frac1{r(\e)}=1.
\end{equation}
By interpolation of $L^p$ spaces we also have $f\in (L^{p-\e}\cap L^2)(\Omega)$ and $g\in (L^{q+\e}\cap L^2)(\Omega)$.

Choose also $\delta>0$ and let $C=\delta I_{\C^d}$. By \cite[Lemma 5.20]{CD-DivForm} we have that $C$ is $r$-elliptic.
Furthermore, take $R>0$ such that the ball $B(0,R)$ intersects the open set $\Omega$ and define
$h=\chi_{\Omega\cap B(0,R)}$, with $\chi$ denoting the characteristic function. Clearly, $h\in (L^{r}\cap L^2)(\Omega)$. Finally, let $\oW = H_0^1(\Omega)$.

We would like to apply Theorem \ref{t: trilinemb} for the triples of functions $f,g,h$, matrices $(A,B,\delta I)$ and indices $(p-\e,q+\e,r(\e))$ as chosen above and with arbitrarily small $\e>0$. To this end we need to make sure that the corresponding $*$-ellipticity conditions are satisfied.
\begin{itemize}
\item
If $\Omega$ is arbitrary $A,B$ are real, then for every $s>0$ they are automatically $s$-elliptic \cite{CD-DivForm}. In particular, they are $\max\{p,q,r\}$-elliptic regardless of the choice of $p,q,r>1$. We also know from \cite{CD-DivForm} that {\it real} elliptic matrices are the only ones with this property.
\item
If $\Omega=\R^d$ and $A,B$ are arbitrary (complex elliptic), then we have to verify that {\rm({\Large $\star$})} holds for indices
$(p-\e,q+\e,r(\e))$.

Recall that $p,q$ were conjugate expontents. Since $A,B\in\cA_p(\Omega)=\cA_q(\Omega)$, we infer from \cite[Corollary 5.16]{CD-DivForm} and the choice of $\e$ that $A\in\cA_{p-\e}(\Omega)$ and $B\in\cA_{q+\e}(\Omega)$.
Similarly, for
$$
s(\e):=1+\frac{p-\e}{q+\e}
$$
we have
$$
\aligned
2<
s(\e)
=\frac{q+p}{q+\e}
<\frac{q+p}{q}
=p,
\endaligned
$$
which means that $A,B$ are also $s(\e)$-elliptic.
\end{itemize}

We may now apply \eqref{eq: trilinemb} as follows:
\begin{equation}
\label{eq: Smarjeske toplice}
\int_0^\infty\int_{\Omega}
\left| \nabla T_t^{A,\oU} f \right|
\left| \nabla T_t^{B,\oV} g \right|
\left| T_t^{\delta I,\oW} h \right|
\wrt x\wrt t
\ \leqsim\
\nor{f}_{p-\e}\nor{g}_{q+\e}\nor{h}_{r(\e)}.
\end{equation}

We first want to send $\e\rightarrow0$.

By the continuity of $L^s$ norms at $s=p,q,\infty$ we have
$$
\lim_{\e\rightarrow0}\nor{f}_{p-\e}\nor{g}_{q+\e}\nor{h}_{r(\e)}
=\nor{f}_{p}\nor{g}_{q}\nor{h}_{\infty}
=\nor{f}_{p}\nor{g}_{q}.
$$
Since $\e$ does not appear on the left-hand side in \eqref{eq: Smarjeske toplice}, we proved
\begin{equation}
\label{eq: Sarah Vaughan - Misty}
\int_0^\infty\int_{\Omega}
\left| \nabla T_t^{A,\oU} f \right|
\left| \nabla T_t^{B,\oV} g \right|
\left| T_t^{\delta I,\oW} h \right|
\wrt x\wrt t
\ \leqsim\
\nor{f}_{p}\nor{g}_{q}.
\end{equation}
A comment is in place at this point. As stated in Theorem \ref{t: trilinemb}, the embedding constants implied in \eqref{eq: trilinemb} also depend on $r$ and $\Delta_r(C)$. So when we move
$\e\rightarrow0$,
one has to make sure that the embedding constants stay finite. In fact this is indeed the case; see Corollary \ref{c: cc}.
Moreover, the same Corollary \ref{c: cc} tells us that the embedding constants from
\eqref{eq: Sarah Vaughan - Misty} do {\it not} depend on $\delta$.

Next we want to send $\delta\rightarrow0$. By the strong continuity of $T_s^{I,\oW}$ in $L^2(\Omega)$ we have
$$
h=\lim_{s\searrow0}T_s^{I,\oW} h
$$
in the $L^2(\Omega)$ sense.
Choose $k,n\in\N$. From \eqref{eq: Sarah Vaughan - Misty} we trivially get
\begin{equation}
\label{Rahmaninov - Vsenoschnoe bdenie}
\int_0^\infty\int_{\Omega}
\min\left\{\left| \nabla T_t^{A,\oU} f \right|,k\right\}
\left| \nabla T_t^{B,\oV} g \right|
\left| T_t^{(1/n) I,\oW} h \right|
\wrt x\wrt t
\ \leqsim\
\nor{f}_{p}\nor{g}_{q}.
\end{equation}
Therefore the integrand on the left-hand side \eqref{Rahmaninov - Vsenoschnoe bdenie} can be decomposed as $X_k Y Z_n$,
where for every fixed $t>0$ we have $X_k(t) \in L^\infty(\Omega)$, $Y(t) \in L^2(\Omega)$, while $Z_n(t) \in L^2(\Omega)$ and converges in $L^2(\Omega)$ to $h\in L^2(\Omega)$ as $n\rightarrow\infty$.
Consequently, for any fixed $t>0$ one obtains $\int_\Omega (X_k Y Z_n)(t) \rightarrow \int_\Omega (X_k Y)(t) h$ as $n\rightarrow\infty$, which through
Fatou's lemma
transforms \eqref{Rahmaninov - Vsenoschnoe bdenie} into
\begin{equation*}
\int_0^\infty\int_{\Omega}
\min\left\{\left| \nabla T_t^{A,\oU} f \right|,k\right\}
\left| \nabla T_t^{B,\oV} g \right|
|h|\wrt x\wrt t
\ \leqsim\
\nor{f}_{p}\nor{g}_{q}.
\end{equation*}
The embedding constants are the same as in \eqref{eq: Sarah Vaughan - Misty}, since those were independent of $\delta$.

Recalling that we had $h=\chi_{\Omega\cap B(0,R)}$, we finally pass $R\rightarrow\infty$ and $k\rightarrow\infty$ which gives \eqref{eq: N bil}.

\begin{remark}
The reason for restricting ourselves in this section to the cases of either $\Omega=\R^d$ or $A,B$ real, was to
apply Theorem \ref{t: trilinemb} for triples $(A,B,\delta I)$ and $(p,q,r)$ with {\it arbitrarily large} $r$.
The underlying observation is that % the condition 
$A,B,C\in \cA_{\max\{p,q,r\}}(\Omega)$ holds for all $r>1$ only when $A,B,C$ are real, see \cite[p. 3179]{CD-DivForm} or \eqref{eq: real elliptic}.
We believe that Theorem \ref{t: trilinemb} might hold under the more convenient conditions ({\Large $\star$}) for any open $\Omega$, not only $\Omega=\R^d$,  
yet at this moment we
cannot prove this. If we could,
then the reasoning from this section would automatically extend to the case of arbitrary open $\Omega$ and complex elliptic $A,B$.

Note also that if $r\rightarrow\infty$ (as in this section), then $p,q$ become conjugate exponents, therefore the condition $\Delta_{1+q/p}(A)>0$ from Theorem \ref{t: trilinemb} becomes just $\Delta_p(A)>0$, which is familiar to us from Theorem \ref{t: N bil}.
This makes the connection between the two results and the limiting procedure from this section look more natural. Also, it supports our impression that the set of conditions ({\Large $\star$}) could indeed be optimal for trilinear embeddings on any open sets.
\end{remark}

\subsection{Applications}
\label{s: Applications}

\subsubsection{Paraproducts associated with semigroups}
Let $p,q,r,\Omega,A,B,C,\oU,\oV,\oW,f,g,h$ be as in Theorem~\ref{t: trilinemb}.
Define the trilinear form 
$\Theta$ 
by
$$
\Theta(f,g,h)
:= - \int_{0}^{\infty} \int_{\mathbb{R}^d} \Big(\frac{\textup{d}}{\textup{d}t} T^A_t f\Big) (T^B_t g) (T^C_t h) \wrt x \wrt t
= \int_{0}^{\infty} \int_{\mathbb{R}^d} (L_A T^A_t f) (T^B_t g) (T^C_t h) \wrt x \wrt t.
$$
Using \eqref{eq: int by parts} and the product rule it can be rewritten as
$$
\Theta(f,g,h) = \int_{0}^{\infty} \int_{\mathbb{R}^d} (T^C_t h) \big\langle A \nabla T^A_t f, \overline{\nabla T^B_t g} \big\rangle_{\mathbb{C}^d} \wrt x \wrt t
+ \int_{0}^{\infty} \int_{\mathbb{R}^d} (T^B_t g) \big\langle A \nabla T^A_t f, \overline{\nabla T^C_t h} \big\rangle_{\mathbb{C}^d} \wrt x \wrt t.
$$
Theorem~\ref{t: trilinemb} gives
\begin{equation}\label{eq:pprodest}
|\Theta(f,g,h)| \leqsim \nor{f}_{p}\nor{g}_{q}\nor{h}_{r}.
\end{equation}
In particular, $\Theta$ can be uniquely extended to a bounded trilinear form on 
$L^p(\Omega)\times L^q(\Omega)\times L^r(\Omega)$. 
It is natural to call it the {\it paraproduct} associated with semigroups $(T^{A}_{t})_{t\geq 0}$, $(T^{B}_{t})_{t\geq 0}$ and $(T^{C}_{t})_{t\geq 0}$.
In order to arrive at a more familiar paraproduct-type expression, we define
\begin{equation}\label{eq:choiceofphipsi}
\varphi(z) := e^{-z}, \quad \psi(z) := z e^{-z},
\end{equation}
and then rewrite $\Theta$ using the functional calculus notation, as
$$
\Theta(f,g,h) = \int_{0}^{\infty} \int_{\mathbb{R}^d} \big(\psi(tL_A)f\big) \big(\varphi(tL_B)g\big) \big(\varphi(tL_C)h\big) \wrt x \,\frac{\textup{d}t}{t}.
$$
It is more common to reserve the word ``paraproduct'' for certain bilinear operators $(f,g)\mapsto\Pi(f,g)$, but one quickly arrives at a trilinear form by dualizing them with a third function $h$.

Paraproducts associated with semigroups were previously studied by several authors in various settings, often more general than ours; see the papers by  Bernicot \cite{B12}, Bernicot and Sire \cite{BS13}, Frey \cite{F13}, Frey and Kunstmann \cite{FK13}, Bernicot and Frey \cite{BF14}, and Wr\'{o}bel \cite{W18}.
However, in the context of divergence form operators, Theorem~\ref{t: trilinemb} allows the study of paraproducts associated with three different semigroups, while the aforementioned literature was confined to taking $A=B=C$, $A^\ast=B=C$ or $A=B^\ast=C$. Moreover, the implicit constant in estimate \eqref{eq:pprodest} only depends on the exponents $p,q,r$ and the stated $\ast$-ellipticity constants of $A,B,C$, so the result has a ``dimensionless'' flavor.
We regard \eqref{eq:pprodest} rather as the first hint of a possible general dimension-free theory of paradifferential calculus for complex elliptic operators.

\subsubsection{Square functions}

In \cite{CD-DivForm} a question was raised about a connection between $p$-ellipticity and {\it square function} estimates on $L^p$. Here we present a few observations in this direction.

Fix $\Omega\subseteq\R^{d}$ and $\oU$ as in Section~\ref{s: Heian}. Set $T^{A}_{t}=T^{A,\oU}_{t}$. Consider the {\it vertical Littlewood--Paley--Stein square function} defined for $f\in L^{2}(\Omega)$ by the rule
$$
\cG^{A}_{\oU}(f)(x):=\left(\int^{\infty}_{0}\mod{(\nabla T^{A}_{t}f)(x)}^{2}\wrt t\right)^{1/2}.
$$

When % \underline{$\Omega=\R^d$}, 
$\Omega=\R^d$, 
it was proved by Auscher \cite[Chapter 6]{Auscher1} that $\cG^{A}_{\oU}$ is bounded on $L^q(\R^d)$ for $q_-(A)<q<q_+(A)$ and unbounded for $q<q_-(A)$ or $q>q_+(A)$, where 
$(q_{-}(L_{A}),q_{+}(L_{A}))$ is the maximal open interval of exponents $p\in[1,\infty]$ for
which $(\sqrt{t}\nabla T_t^A)_{t>0}$ is uniformly bounded on $L^p$.
We have $q_-(A)=1$ for all real $A$ \cite[Corollary 6.6]{Auscher1}. While $q_+(I)=\infty$, in principle $q_+(A)$ can be arbitrarily close to $2$ even for real elliptic matrices $A$.
In order to have a more convenient result for all $q\in(2,\infty)$, Auscher, Hofmann and Martell \cite{AHM12} considered the {\it conical square function}, defined as
$$
\cC^{A}(f)(x):=
{\displaystyle
\left(
\iint_{V_x}\left|\nabla(T^{A}_{t}f)(y)\right|^2\,\frac{\textup{d}y\wrt t}{t^{d/2}}
\right)^{1/2}},
$$
where $V_x=\mn{(y,t)\in\R^{d}\times(0,\infty)}{|x-y|<\sqrt t}$ is a cone with respect to the {\it parabolic} metric on $\R^d\times (0,\infty)$, given by $d\big((x,t),(y,s)\big):=\sqrt{|x-y|^2+|t-s|}$.

Again following Auscher \cite{Auscher1}, denote by $(p_{-}(L_{A}),p_{+}(L_{A}))$ the maximal open interval of exponents $p$ for
which $(T_t^A)_{t>0}$ is uniformly bounded on $L^p$.
Auscher, Hofmann and Martell \cite[Theorem 3.1.(2)]{AHM12} showed that $\cC^{A}$ is bounded on $L^{p}(\R^{d})$ whenever $p\in (p_{-}(L_{A}),\infty)$.
Moreover, the lower bound is optimal and equals 1 when $A$ is real \cite[p.19]{Auscher1}.

\medskip
The situation of 
$\Omega\subsetneq\R^d$
is notably different, as
it can be proved that if $d\geq 3$, then for every $p>3$ there exists a bounded, connected open set $\Omega\subseteq\R^{d}$ with strongly Lipschitz boundary such that, if either $\oU=H^{1}(\Omega)$ (pure Neumann boundary conditions) or $\oU=H^{1}_{0}(\Omega)$ (pure Dirichlet conditions), then $\cG^{A}_{\oU}$ is {\it not} bounded on $L^{p}(\Omega)$ even for $A=I$ (see \cite{CD20-SqF}).

\medskip
In the generality we consider in this paper, we cannot prove the trilinear inequality \eqref{eq: trilinemb} by means of the semigroup maximal operator $T^{A}_{*}:f\mapsto \sup_{t>0}\mod{T^{A}_{t}f}$ and the vertical square function $\cG^{A}_{\oU}$ (compare with \cite[p. 460]{KS18}). Namely, we do not know whether $T^{A}_{*}$ is bounded on $L^p$ \cite{CD20-SqF}. Moreover, when $A=B=C$ are real matrices, Theorem~\ref{t: trilinemb} holds true for any $p,q,r\in (1,\infty)$ such that $1/p+1/q+1/r=1$, while, as discussed above, the square function $\cG^{A}_{\oU}$ could be unbounded on $L^p$ for $p>3$, even for $A=I$.

In light of the discussion above, it is natural to consider yet another square function,
in the spirit of the P.-A. Meyer's modified square function originally associated with a Poisson diffusion semigroup \cite{Meyer1,Meyer2} (see also \cite[Section~3]{Coulhon-Duong}).
Denote by $(T_{t})_{t>0}$ the Neumann heat semigroup on $\Omega$, i.e., $T_{t}=T^{I,H^{1}(\Omega)}_{t}$. The {\it modified square function} is defined by the rule
$$
\widetilde{\cG}^{A}_{\oU}(f)(x):=
\left(\int^{\infty}_{0}T_{t}\left(\mod{\nabla T^{A}_{t}f}^{2}\right)(x)\wrt t\right)^{1/2},\quad f\in L^{2}(\Omega).
$$
We emphasize that the definition of $\widetilde{\cG}^{A}_{\oU}(f)$ features {\it two} different semigroups. Let us also remark that in principle we could replace $T_t$ by a more general sub-Markovian semigroup.

From the classical formula for the heat kernel on $\R^d$ we notice that when $\Omega=\R^{d}$, we have
\begin{equation}
\label{eq: Eminem Forgot About Dre}
\cC^{A}f
\leq e^{1/8}(4\pi)^{d/4}\,
\widetilde{\cG}^{A}f
\end{equation}
pointwise on $\R^d$.
Conversely, a bit more work shows that the boundedness of $\cC^{A}$ on $L^p(\R^d)$ implies the boundedness of $\widetilde{\cG}^{A}$ on the same space.

The following result of ours deals with the functional $\widetilde{\cG}_\oU^{A}$, which thus by \eqref{eq: Eminem Forgot About Dre} majorizes $\cC^{A}$ when $\Omega=\R^d$. Our theorem applies to arbitrary domains $\Omega$, and, unlike most of the results cited above, features dimension-free estimates.

\begin{theorem}
\label{t: LPS 2}
Let $\oU$ be as in Section~\ref{s: Heian}.
If $A\in\cA(\Omega)$ and $p\geq2$ are such that $A$ is $p$-elliptic, then the modified square function $\widetilde{\cG}^{A}_{\oU}$ is bounded on $L^{p}(\Omega)$.
The norm estimates do not depend on the underlying dimension $d$.
\end{theorem}
\begin{proof}
Take $f\in (L^p\cap L^2)(\Omega)$ and write $G=\widetilde{\cG}^{A}_{\oU}f$.
By duality, for $s=(p/2)'$ we have
$$
\nor{G}_p^2=\nor{G^2}_{p/2}=\sup\mn{\Big|\int_\Omega G^2\cdot\bar\phi\Big|}{\phi\in (L^s\cap L^2)(\Omega),\ \nor{\phi}_s=1}.
$$
Since the identity matrix is $r$-elliptic for any $r\geq1$, see e.g. \cite[Lemma 5.20]{CD-DivForm}, we may apply Theorem~\ref{t: trilinemb} with the triples of indices $(p,p,s)$, matrices $(A,A,I)$ and functions $(f,f,\bar\phi)$.
\end{proof}

\begin{remark}
After the proofs of Theorems~\ref{t: trilinemb} and \ref{t: LPS 2} were completed, the first two authors of the present paper realized that Theorem~\ref{t: LPS 2} could also be proven by a variant of a {\it bilinear} embedding based on a
modification of the technique from \cite{CD-Mixed}.
This will be explained in their paper in preparation \cite{CD20-SqF}.
\end{remark}

\subsubsection{Kato--Ponce-type inequalities and Bessel--Sobolev algebras}
\label{s: KP}

It was proven in \cite[Theorem~1.3]{CD-DivForm}, \cite[Theorem~2]{Egert2018} and \cite[Lemma~17]{CD-Mixed} that $(T^{A}_{t})_{t>0}$ extends to a contractive analytic semigroup on $L^{\wp}(\Omega)$ whenever $A$ is $\wp$-elliptic;
see Proposition \ref{p: Zbunjen i osamucen} for a more explicit statement.
In this section we slightly abuse the notation and maintain the symbol $L_{A}$ for the negative generator of $(T^{A}_{t})_{t>0}$ in $L^{\wp}(\Omega)$. 
Let $\cD_\wp(L_A)$ denote its domain. 
The generator $L_{A}$ is a {\it sectorial} operator (see \cite{CDMY}, \cite[Theorem II.4.6]{EN} or \cite{Haase} for references) on $L^{\wp}(\Omega)$ of angle $<\pi/2$, so 
$$
L^{\wp}(\Omega)={\rm N}_{\wp}(L_{A})\oplus \overline{\Ran_{\wp}(L_{A})},
$$
where ${\rm N}_{\wp}(L_{A})$ and $\Ran_{\wp}(L_{A})$ denote the kernel (null space) and the range of the operator $L_A$. 
Moreover,  the projection $Q_{\wp}^{A}$ onto ${\rm N}_{\wp}(L_{A})$ is given by
$Q_{\wp}^{A}f=\lim_{t\rightarrow+\infty}T^{A}_{t}f$ in $L^\wp$; see \cite{CDMY}. In particular, 
$\mn{Q_{\wp}^{A}}{A\text{ is }\wp\text{-elliptic}}
$ 
is a consistent family of contractive projections. See \cite[Theorem 3.8 and p. 63]{CDMY}.

From this,
\eqref{eq: int by parts} and \eqref{eq: ellipticity} 
it is not hard to see that, if $A$ is $\wp$-elliptic, then the kernel ${\rm N}_\wp(L_A)$ consists of
all $L^p$ functions which are constant on each connected component of $\Omega$ of finite measure whose boundary does not intersect $\Gamma$, and are $0$ on other connected components of $\Omega$.
Here $\Gamma$ is as in Section \ref{s: Heian} (c).

In order to simplify our proofs, we shall always assume that ${\rm N}_{2}(L_{A})=\{0\}$. 
Under this assumption, $L_{A}$ is thus an injective sectorial operator of angle $\vartheta\in(0,\pi/2)$ on $L^{\wp}(\Omega)$, provided that $A$ is $\wp$-elliptic. Then the complex powers $L^{z}_{A}$ are well defined on $L^{\wp}(\Omega)$ for every 
$z\in\C$; see \cite[Section 3.2]{Haase} or \cite[Chapter 2, Section 7]{Yagi2010} for the construction.

Let $A$ be $\wp$-elliptic. By \cite[Theorem 3]{CD-Mixed}, $L_{A}$ admits a bounded $H^{\infty}$-calculus of angle $<\pi/2$ on $L^{\wp}(\Omega)$ in the sense of \cite{CDMY}. It follows that $L_{A}$ has bounded {\it imaginary powers} (see \cite{CDMY} and \cite[Proposition 3.5.5]{Haase} for the main properties) on $L^{\wp}(\Omega)$ and there exist $\theta_{\wp}\in (0,\pi/2)$ and $C>0$ such that
\begin{equation}
\label{eq: impowers}
\norm{L^{iu}_{A}}{\wp}\leq Ce^{\theta_{\wp}|u|},\quad \forall u\in\R.
\end{equation}
See \cite[Theorems 4.2 and 5.1]{CDMY}.
%\begin{remark}
%Boundedness of imaginary powers on $L^{\wp}(\Omega)$ implies that
%\[
%\Dom_{\wp}(L^{\beta}_{A})=\left[L^{\wp}(\Omega);\Dom_{\wp}(L_{A})\right]_{\beta},\quad \forall \beta\in [0,1],
%\]
%where $\left[L^{\wp}(\Omega);\Dom_{\wp}(L_{A})\right]_{\beta}$ denotes the complex Calder\'on interpolated space at $\beta$. See, for example, \cite[Section 5]{CDMY}, \cite[Theorem~16.5]{Yagi2010} or \cite[Theorem~6.6.9]{Haase}.
%\end{remark}

\begin{theorem}[Kato--Ponce-type inequality]
\label{t: princ}
Choose numbers $p_{1},q_{1},p_2,q_2,r
\in (1,\infty)$ and denote $\varrho=r'=r/(r-1)$. 
Assume that
 $1/{p_1}+1/{q_1}=1/{p_2}+1/{q_2}=1/\varrho$.
Suppose that $A\in\cA(\Omega)$ is $\max\{p_{1},p_{2},q_{1},q_{2},r\}$-elliptic. Let $\beta\in (0,1/\varrho)$. Then 
for all $f\in \Dom_{p_{1}}(L^{\beta}_{A})\cap L^{p_{2}}(\Omega)$ and $g\in\Dom_{q_{2}}(L^{\beta}_{A})\cap L^{q_{1}}(\Omega)$ we have $fg\in\Dom_{\varrho}(L^{\beta}_{A})$ and 
\begin{equation*}
\label{eq: KatoPonce}
\|L^{\beta}_{A}(fg)\|_{\varrho}
\leqsim
\|L^{\beta}_{A}f\|_{p_{1}}\norm{g}{q_{1}}+\norm{f}{p_{2}}\|L^{\beta}_{A}g\|_{q_{2}}.
\end{equation*} 
When $\Omega=\R^d$,
the same conclusion holds
under milder assumptions, namely, when the triplets $(p_1,q_1,r)$ and $(p_2,q_2,r)$ satisfy 
{\rm({\Large $\star$})} from page \pageref{star}.

The implied constants depend on $\beta,p_1,q_1,p_2,q_2,r$ and $*$-ellipticity constants of $A,B,C$ alluded to in the theorem's assumptions.
\end{theorem}

The inequalities of the above kind, known also as {\it fractional Leibniz rule}, 
bear name after the classical 1988 paper by Kato and Ponce \cite{KP88}
and have since then been studied profusely. 
See e.g. Bernicot--Coulhon--Frey \cite{BCF16}, 
Grafakos--Oh \cite{GO14} and Li \cite{Li19} for just a few of the very 
many recent works on the subject.  

To the best of our knowledge, Theorem \ref{t: princ} is the first instance of a Kato--Ponce-type inequality for general divergence-form operators with complex coefficients and on arbitrary open sets $\Omega$, and thus the first one that relates $p$-ellipticity to this type of estimates. 
We emphasize that the condition it imposes ($p$-ellipticity) is algebraic and thus easy to verify. 
Together with other examples presented in this paper it suggests that 
the trilinear embedding and the route to its proof have 
a significant
potential that might be explored in the future.

\begin{remark} 
In Theorem~\ref{t: princ} the upper bound $\beta<1/r^{\prime}$  is consistent with the upper bound $\alpha<p_{0}/p$, $p_{0}=2$, in Bernicot and Frey
\cite[Theorem~1.3]{BF2018}. Note that in \cite[Theorem~1.3]{BF2018} (restricted to the special case of divergence form operators on open sets satisfying the doubling condition) 
$p_{0}=2$ corresponds to \cite[$({\rm G}_{2})$]{BF2018}: the 
estimate $\sup_{t>0}\norm{\sqrt{t}\nabla T^{A}_{t} }{2}<+\infty$, which comes for free from the definition of $L_{A}$ by means of the sesquilinear form $\gota$ and the analyticity of $(T^{A}_{t})_{t>0}$ in $L^{2}$.  
\end{remark}

\subsection{Organization of the paper}
Here is the summary of each section.
\begin{itemize}
\item
In Section \ref{s: prelims} we summarize some of the main notions needed in the paper.
\item
In Section \ref{s: KS} we state the main result regarding the Bellman function $\gX$.
\item
In Section \ref{s: mocne funkcije} we discuss the building blocks of $\gX$, namely, the power functions.
\item
In Section \ref{s: RICE} we define $\gX$ and prove the theorem announced in Section \ref{s: KS}.
\item
In Section \ref{s: Extracorporeal Shock Wave Therapy} we prove the trilinear embedding for $\R^d$.
\item
In Section \ref{s: general case} we prove the trilinear embedding for general $\Omega$.
\item
In Section \ref{s: KPproof} we prove the Kato--Ponce inequality announced in Section \ref{s: KP}.
\end{itemize}

\section{More notation and preliminaries}
\label{s: prelims}

For $a_1,a_2>0$ we write $a_1\geqsim a_2$ if there is a constant $c>0$ such that $a_1\geqslant c a_2$. Similarly we define $a_1\leqsim a_2$. If both $a_1\geqsim a_2$ and $a_1\leqsim a_2$, then we write $a_1\sim a_2$.

If $z=(z_1,\hdots,z_d)\in\C^d$ and $w$ is likewise, we write
$$
\sk{z}{w}_{\C^d}=\sum_{j=1}^dz_j\overline w_j\,
$$
and $|z|^2=\sk{z}{z}_{\C^d}$.

Let $\N$ denote the set of all positive integers.

\subsection{Real form of complex operators}
We explicitly identify $\C^{d}$ with $\R^{2d}$ as follows.
For each $d\in\N$ consider the operator
$\cV_{d}:\C^{d}\rightarrow\R^{d}\times\R^{d}$, defined by
$$
\cV_{d}(\xi_{1}+i\xi_{2})=
(\xi_{1},\xi_{2}).
$$
Let $N,d\in\N$. We define another identification operator,
$$
\cW_{N,d}:\underbrace{\C^{d}\times\hdots\times\C^{d}}_{N\text{ times}}\longrightarrow \underbrace{\R^{2d}\times\hdots\times\R^{2d}}_{N\text{ times}},
$$
 by the rule
$$
\cW_{N,d}(\xi^{1},\dots,\xi^{N})
=\left(\cV_{d}(\xi^{1}),\dots,\cV_{d}(\xi^{N})\right)
$$
with $\xi^{j}\in\C^{d}$ for $j=1,\dots,N$.

\subsection{Gradient and Hessian forms}
\label{s: Ella Fitzgerald - Caravan}

Let $\Omega\subseteq \R^d$ be an open set and $u:\Omega\rightarrow\C$.
Following \cite[Appendix A.3]{Evans}, we will denote by $Du$ the {\it gradient}
$(\partial_{x_1}u,\hdots,\partial_{x_d}u)$ of $u$,
while $D^2u$ will denote the {\it Hessian matrix} of $u$, that is, the matrix of all second-order derivatives of $u$.
We will also regard $Du$ and $D^2u$ as
{\it sets}
of all first- and, respectively, second-order derivatives of $u$. In accordance with this notation we also have
\begin{equation*}
|Du|=\left(\sum_{j=1}^{d}|\partial_{x_j}u|^2\right)^{1/2}
% \sim\sum_{j=1}^{d}|\partial_{x_j}u|
\end{equation*}
and likewise
\begin{equation*}
\left|D^2u\right|=\left(\sum_{j,k=1}^{d}|\partial^2_{x_jx_k}u|^2\right)^{1/2}
%\sim\sum_{j,k=1}^{d}\left|\partial^2_{x_jx_k}u\right|
.
\end{equation*}
Sometimes we may also write $D_ju$ for $\partial_{x_j}u$ and similarly $D^2_{jk}u$ for $\partial^2_{x_jx_k}u$.

% The equivalence constants ``in the lower direction'' above in general depend on the dimension $d$, however this is not going to affect our main estimates. See Section \ref{s: general case} for details.

When the entries of $u$ are complex, that is, if function $u$ is defined on a subset of $\C^N$, then by $Du$ we mean $
\left[D\big(u\circ \cW_{N,1}^{-1}\big)\right]\circ \cW_{N,1}$ and the same for $D^2u$.

\subsection{Generalized Hessians and generalized convexity}
\label{s: cajkovskij grand sonata}

Given a matrix $A\in\C^{d\times d}$, we introduce, as in \cite{CD-DivForm, CD-Mixed}, its derived real $(2d)\times(2d)$ matrix-form as follows:
$$
\cM(A)=\left[
\begin{array}{rr}
\Re A  & -\Im A\\
\Im A  & \Re A
\end{array}
\right].
$$
Let $N,d\in\N$ and $\Phi\colon\mathbb{C}^N\to\mathbb{R}$ of class $C^2$. Choose and, respectively, denote
\begin{equation}
\label{eq: Bare Garaza 2014}
\aligned
\omega_1,\hdots,\omega_N&\in \C &\hskip 30pt &\hskip 3pt \omega:=(\omega_1,\hdots,\omega_N)\\
X_1,\hdots,X_N&\in \C^d &\hskip 30pt &X:=(X_1,\hdots,X_N)\\
A_1,\hdots,A_N&\in \C^{d\times d} &\hskip 30pt &\bA:=(A_1,\hdots,A_N).
\endaligned
\end{equation}
Following \cite{CD-DivForm, CD-Mixed}, we define the {\it generalized Hessian form of $\Phi$ with respect to $\bA$} as
\begin{equation*}
H_{\Phi}^{\bA}[\omega;X]=
 \sk{
 \left[
D^2\Phi(\omega)
    \otimes I_{\R^d}
 \right]
 \cW_{N,d}(X)}
 {\left[\cM(A_1)\oplus\hdots\oplus\cM(A_N)\right]\cW_{N,d}(X)}_{\R^{2Nd}}.
\end{equation*}
We recall that
$D^2\Phi(\omega)$
stands for the Hessian matrix of the function $\Phi\circ\cW_{N,1}^{-1}:(\R^{2})^N\rightarrow\R$, calculated at the point $\cW_{N,1}(\omega)\in(\R^{2})^N$.

Observe that
$$
\left[\cM(A_1)\oplus\hdots\oplus\cM(A_N)\right]\cW_{N,d}(X)
=\cW_{N,d}\left(A_1X_1,\hdots,A_NX_N\right).
$$
In particular, when $N=1$ we have the formula already stated in \cite[(2.4)]{CD-DivForm}:
$$
\cV_d(A\xi)=\cM(A)\cV_{d}(\xi)
$$
for $A\in\C^{d\times d}$ and $\xi\in\C^d$.

To shed more light onto the notion of $H_\Phi^\bA$, let us illustrate it in the special case $N=2$. There, applying the block notation, we have
$$
H^{(A_1,A_2)}_{\Phi}[\omega;X]=
\sk{
D^2\Phi(\omega)
\left[\begin{array}{c}\Re X_1 \\ \Im X_1 \\ \Re X_2 \\ \Im X_2\end{array}\right]}
{\left[\begin{array}{crcr}\Re A_1 & -\Im A_1 &  &  \\ \Im A_1 & \Re A_1&  &  \\ &  & \Re A_2 &-\Im A_2  \\ &  & \Im A_2 & \Re A_2 \end{array}\right]
\left[\begin{array}{c}\Re X_1 \\ \Im X_1 \\ \Re X_2 \\ \Im X_2\end{array}\right]
}_{(\R^d)^4}.
$$

We say that $\Phi$ is {\it convex with respect to $\bA$} if
$H_{\Phi}^{\bA}[\omega;X]\geq 0$ for all $\omega,X$.

\medskip
The sharp condition \eqref{eq: Sparky 21} is equivalent to the strict generalized convexity of power functions $F_p$ (see \eqref{eq: Oblivion} for the definition), that is, to the (uniform) positivity 
$H_{F_p}^{A}[\omega;X]$.
Prior to \cite{CD-DivForm}, the importance of this positivity 
was recognized and studied in a few special cases: when $A$ is either the identity \cite{NT, DV-Sch}, real 
% and uniformly positive definite 
accretive 
\cite{DV-Kato}, of the form $e^{i\phi}I$ \cite{CD-mult}, or of the form $e^{i\phi}B$ with $B$ real, constant and with a symmetric part which is positive definite \cite{CD-OU}. 
Such problems are also related to similar questions considered earlier by Bakry; see \cite[Th\'eor\`eme 6]{Bakry3}. The paper \cite{CD-DivForm} brought a systematic approach to convexity of power functions in presence of arbitrary uniformly strictly accretive complex matrix functions $A$. See also Section \ref{s: mocne funkcije}.

\section{The Bellman function}
\label{s: KS}

{\it Bellman functions} have been a powerful tool in harmonic analysis since the 1980's papers by Burkholder, e.g. \cite{Bu1,Bu2}.
The method was given a huge boost by the pioneering works by Nazarov, Treil and Volberg \cite{NTV}, who also gave it its name \cite{NTV1}. The exploration of heat flows associated with Bellman functions started in the works by Petermichl and Volberg \cite{PV} and Nazarov and Volberg \cite{NV}. See e.g. \cite{CD-DivForm} for more recent references on the method and its applications.

We use a three-variable {\it Bellman function} found by the latter two authors of the present paper \cite{KS18}. It is modelled over a two-variable Bellman function due to Nazarov and Treil \cite{NT}. Compared to \cite[Theorem 1.1]{KS18}, the main novelty in the convexity estimate formulated here is the presence of complex elliptic matrices, as specified below.

\medskip
Denote
$$
\Upsilon=\mn{(u,v,w)\in\C^3}{
(uvw=0) \vee (|u|^{p}=|v|^{q})
\vee (|u|^{p}=|w|^{r})
\vee (|v|^{q}=|w|^{r})}.
$$

\begin{theorem}
\label{t: Voltaren}
Suppose that $p,q,r\in(1,\infty)$ satisfy $1/p+1/q+1/r=1$ and $\Omega\subseteq \R^d$ is open.
Let the accretive matrices $A,B,C:\Omega\rightarrow\C$
satisfy {\rm({\Large $\star$})} from page \pageref{star}, that is, assume that:
\begin{itemize}
\item
$A$ is $p$-elliptic and $(1+p/q)$-elliptic
\item
$B$ is $q$-elliptic and $(1+q/p)$-elliptic
\item
$C$ is $r$-elliptic.
\end{itemize}

There exists a function $\gX:\C^3\rightarrow\R_+$ of class $C^1$ on $\C^3$ and of class $C^2$ on $\C^3\backslash\Upsilon$, such that:
\begin{enumerate}[a)]
\item
\label{eq: Burger}
for all $u,v,w\in\C$,
$$
\gX(u,v,w)\,\leqsim\,|u|^{p} + |v|^{q} + |w|^{r};
$$
\item
\label{eq: Fest}
for all $u,v,w\in\C$,
$$
\aligned
|\partial_{\bar u}\gX|&\leqsim\left(\max\left\{|u|^{p},|v|^{q},|w|^{r}\right\}\right)^{1-1/p},\\
|\partial_{\bar v}\gX|&\leqsim\left(\max\left\{|u|^{p},|v|^{q},|w|^{r}\right\}\right)^{1-1/q},\\
|\partial_{\bar w}\gX|&\leqsim\,|w|^{r-1}
%%%%      When $p\geq q$, the middle row holds even without |u|^{p}.    %%%%
\endaligned
$$
\item
\label{eq: sova}
for almost every $x\in\Omega$ we have,
\begin{equation}
\label{eq: Trifonov Kinderszenen}
H^{(A,B,C)(x)}_{\gmX}[(u,v,w);(\zeta,\eta,\xi)]
\,\geqsim\,|w||\zeta||\eta|
\end{equation}
for all $(u,v,w)\in\C^3\backslash\Upsilon$ and $(\zeta,\eta,\xi)\in\left(\C^{d}\right)^3$.
\end{enumerate}
The implied constants depend on $p,q,r$ and $*$-ellipticity constants of $A,B,C$ alluded to in the theorem's assumptions.
\end{theorem}

The theorem will be proven in Section \ref{s: RICE}.
% Of course, the 
The crucial and the most difficult to verify is property {\it \ref{eq: sova})}. As indicated above, we will prove that the function from \cite{KS18} satisfies the above requirements, possibly after refining its parameters.

\section{Power functions}
\label{s: mocne funkcije}

For $p>0$ and $N\in\N$ define the {\it power function} (by which we actually mean powers of the {\it modulus})
$F_{p}\colon \C^{N} \rightarrow [0,\infty)$ by
\begin{equation}
\label{eq: Oblivion}
F_{p}(\zeta)=|\zeta|^{p}.
\end{equation}
While these functions are, technically speaking, different for different values of the dimension $N$, we will use the same symbol $F_p$ to denote all of them.
We also set $F_0=\bena$, where $\bena$ denotes the constant function of value $1$ on $\C^{N}$.
These functions played a fundamental r\^ole in \cite{CD-Potentials, CD-mult, CD-Mixed, CD-OU, CD-DivForm}. In fact, $p$-ellipticity was in \cite{CD-DivForm} initially introduced as the (uniform strict) positivity of generalized Hessians of power functions on $\C$; see \cite[Remark 5.9]{CD-DivForm}.
Namely,
the operator $\cI_{p}\colon\C^d\rightarrow\C^d$, defined by
\begin{equation}
\label{eq: I_p}
\cI_{p}\xi=\xi+(1-2/p)\overline\xi,
\end{equation}
which clearly features in \eqref{eq: Sparky 21} and \eqref{eq: kabuto}, appeared as a result of our  expressing the Hessian of $F_{p}$ (as a function on $\C$, that is, with $N=1$) in an {\it appropriate} manner. See Lemma \ref{l: Batman} for a more general statement.
Here is the formula, first stated in \cite[(5.5)]{CD-DivForm}, that for $N=1$ directly relates $D^2F_{p}$ to $\cI_{p}$:
$$
D^2F_{p}(\zeta)\xi=\frac{p^2}{2}|\zeta|^{p-2}\sign{\zeta}\cdot\cI_{p}(\sign\bar{\zeta}\cdot\xi)
$$
for $\zeta\in\C\backslash\{0\}$ and $\xi\in\C^d$. 
(Here $\sign z=z/|z|$; furthermore, we implicitly identified $\C^d$ with $\R^d\times\R^d$.) See
also
\cite[Remark 2.4]{CD-Potentials}.

Recall that if $f,g$ are complex functions on some sets $X,Y$ respectively, then their {\it tensor product} $f\otimes g$ is the function on $X\times Y$ mapping $(x,y)\mapsto f(x)g(y)$.
In \cite{CD-mult, CD-OU, CD-DivForm} the first two authors of the present paper analyzed the heat flow associated with the Nazarov--Treil function $Q$ introduced in \cite{NT}.
Given that $Q$ was a combination of tensor products $F_{p}\otimes\bena$, $\bena\otimes F_{q}$ and $F_2\otimes F_{2-q}$, where $p,q$ are conjugate exponents, this eventually drew our attention to the power functions $F_p$.
However, tensor products of power functions also represent the core of another Bellman function, namely, function $\gX$ which was constructed by the latter two authors of the present paper in \cite{KS18} and whose generalized convexity we study in Theorem \ref{t: Voltaren}. The main difference is that function $\gX$ is a function of {\it three} complex variables and is made of tensor products of up to {\it three} power functions, while $Q$ was a function of two variables made of tensor products of two power functions.
This amounts to a significantly higher degree of complexity of $\gX$ compared to $Q$.

\medskip
We continue with lower estimates of generalized Hessians of (tensor products of) power functions.
We single out the cases of 1, 2 or 3 factors, for they constitute the function $\gX$.
In fact, it turns out that it is slightly more convenient to express $\gX$ by means of the following multiples of power functions:
$$
G_{p}(\zeta):=\frac{|\zeta|^{p}}{p}.
$$
Likewise, in the context of power functions we find it fitting to renormalize the entries of generalized Hessians and
work with a modified quantity defined as
\begin{equation}
\label{eq: Htilde}
\widetilde{H}^{{\mathbf A}}_{\Phi}[\omega;X] := H^{{\mathbf A}}_{\Phi}[\omega;(\omega_1X_1,\hdots,\omega_NX_N)]
\end{equation}
for
$\omega,X,{\mathbf A}$ as in \eqref{eq: Bare Garaza 2014}.
With this notation introduced, we are in a position to formulate the estimates that we will need in the proof of Theorem \ref{t: Voltaren}.

\begin{proposition}{\cite[Corollary 5.10]{CD-DivForm}}
\label{p: Vinjerac}
For $p>0$, $u\in\C\backslash\{0\}$, $\alpha \in\C^d$ and $A\in\C^{d\times d}$
we have
$$
{\widetilde H}_{G_{p}}^{A}\left[u;\alpha \right]
\geq
\frac{p}2\Delta_{p}(A)|u|^{p}|\alpha |^2.
$$
\end{proposition}

\begin{proposition}
\label{p: Metajna}
For $r,s>0$, $u_1,u_2\in\C\backslash\{0\}$, $\alpha_1,\alpha_2\in\C^d$ and $A,B\in\C^{d\times d}$
we have
$$
\widetilde{H}_{G_{r}\otimes G_s}^{(A,B)}\left[(u_1,u_2);(\alpha_1,\alpha_2)\right]
\geq
\frac{|u_1|^{r}|u_2|^{s}}{rs}
\left(
\frac{r^2}{2}\Delta_{r}(A)|\alpha_1|^2+\frac{s^2}{2}\Delta_s(B)|\alpha_2|^2-2rs\Lambda(A,B)|\alpha_1||\alpha_2|\right).
$$
\end{proposition}

\begin{proof}
Follow the proof of \cite[Corollary 5.12]{CD-DivForm}.
\end{proof}

By now we see how to generalize this to arbitrary $N$-tuples of Lebesgue exponents and matrix functions.
Let us only explicate the case $N=3$ which we will use in our proofs.

\begin{proposition}
\label{p: Zhenodraga}
For $r,s,t>0$, $u_1,u_2,u_3\in\C\backslash\{0\}$, $\alpha_1,\alpha_2,\alpha_3\in\C^d$ and $A,B,C\in\C^{d\times d}$ we have
$$
\aligned
\widetilde{H}_{G_{r}\otimes G_s\otimes G_t}^{(A,B,C)}&\left[(u_1,u_2,u_3);(\alpha_1,\alpha_2,\alpha_3)\right]\\
& \geq
\frac{|u_1|^{r}|u_2|^{s}|u_3|^{t}}{rst}
\bigg(
\frac{r^2}2\Delta_{r}(A)|\alpha_1|^2
+\frac{s^2}2\Delta_{s}(B)|\alpha_2|^2
+\frac{t^2}2\Delta_{t}(C)|\alpha_3|^2\\
&\hskip 20pt
-2\Lambda(A,B,C)
\left(
rs|\alpha_1||\alpha_2|
+rt|\alpha_1||\alpha_3|
+st|\alpha_2||\alpha_3|
\right)\bigg).
\endaligned
$$
\end{proposition}

\begin{proof}
The proof can be carried out very similarly to the proofs of \cite[Corollary 5.10]{CD-DivForm} and \cite[Corollary 5.12]{CD-DivForm}. Alternatively, one can first establish the identity
\begin{align*}
& \widetilde{H}_{G_{r}\otimes G_s\otimes G_t}^{(A,B,C)}  \left[(u_1,u_2,u_3);(\alpha_1,\alpha_2,\alpha_3)\right] \\
& \hskip 30pt
= \frac{|u_1|^r|u_2|^s|u_3|^t}{2} \mathop{\textup{Re}} \bigg(
r^2\langle A\alpha_1,\cI_{r}\alpha_1\rangle_{\C^d}
+ rs\langle A\alpha_1,\cI_{\infty}\alpha_2\rangle_{\C^d}
+ rt\langle A\alpha_1,\cI_{\infty}\alpha_3\rangle_{\C^d} \\
& \hskip 122pt
+ rs\langle B\alpha_2,\cI_{\infty}\alpha_1\rangle_{\C^d}
+ s^2\langle B\alpha_2,\cI_{s}\alpha_2\rangle_{\C^d}
+ st\langle B\alpha_2,\cI_{\infty}\alpha_3\rangle_{\C^d} \\
& \hskip 123pt
+ rt\langle C\alpha_3,\cI_{\infty}\alpha_1\rangle_{\C^d}
+ st\langle C\alpha_3,\cI_{\infty}\alpha_2\rangle_{\C^d}
+ t^2\langle C\alpha_3,\cI_{t}\alpha_3\rangle_{\C^d}
\bigg)
\end{align*}
and then simply use the definitions of the appropriate ellipticity constants.
\end{proof}

\subsection{Generalized convexity of power functions in higher dimensions}

Recall that the power functions $F_p$ and the operators $\cI_p$ were introduced in \eqref{eq: Oblivion} and \eqref{eq: I_p}, respectively.

\begin{lemma}\cite[Lemma 8]{CD-Mixed}
\label{l: Batman}
Suppose that $p>1$ and $N,d\in\N$.
Let $\omega, X,\bA$ be as in \eqref{eq: Bare Garaza 2014}.
Then, for $\omega\ne0$,
\begin{equation}
\label{eq: smoked bacon burger}
H_{F_p}^{\bA}[\omega;X]=|\omega|^{p-2}H_{F_p}^{\bA}[\omega/|\omega|;X].
\end{equation}
In case when $|\omega|=1$ we have, for
$Y_j:=\overline{\omega_j}\,X_j$,
the following formul\ae:
\begin{enumerate}[(I)]
\item
\ \vskip -20pt
$$
\hskip -15pt
p^{-1}H_{F_p}^{\bA}[\omega;X]
 =\sum_{j=1}^{N}\Re\sk{A_{j}Y_{j}}{Y_{j}}
+(p-2)\sum_{j,k=1}^{N}
\Re\sk{A_{j}Y_{j}}{\Re Y_{k}}
$$
\item
\ \vskip -20pt
$$
\aligned
p^{-1}H_{F_p}^{\bA}[\omega;X]  &=\sum_{j=1}^{N}\big(1-|\omega_j|^2\big)\Re\sk{A_{j}Y_{j}}{Y_{j}}
+\frac{p}{2}\sum_{j=1}^{N}
\Re\sk{A_{j}Y_{j}}{\cI_pY_{j}}
\\
&\hskip 107pt
+(p-2)\sum_{j\not =k}
\Re\sk{A_{j}Y_{j}}{\Re Y_{k}}.
\endaligned
$$
\end{enumerate}

\end{lemma}

We note that in the special case $N=1$, Lemma \ref{l: Batman} was proven earlier in \cite[Lemma 5.6]{CD-DivForm}. The case of foremost interest for us in this paper will be $N=3$.

\begin{corollary}\cite[Corollary 9]{CD-Mixed}
\label{c: Debeluh}
Repeating the assumptions of Lemma \ref{l: Batman}
and introducing the notation
$$
\Delta_p(\bA):=\min_{j=1,\hdots,N}\Delta_p(A_j)
\hskip30pt
\text{and}
\hskip 30pt
\Lambda(\bA):=\max_{j=1,\hdots,N}\Lambda(A_j),
$$
we have, for $|\omega|=1$ and $p\geq2$,
$$
p^{-1}H_{F_p}^{\bA}[\omega;X]
\geq\Delta_p(\bA)|X|^2
-(p-2)\Lambda(\bA)\sum_{j\not =k}|\omega_{j}||\omega_k|
|X_{j}||X_{k}|.
$$
\end{corollary}

We remarked at the beginning of Section \ref{s: mocne funkcije} that, when $N=1$ and $A\in\C^{d\times d}$, the $A$-convexity of $F_p$ is closely related to the condition $\Delta_p(A)>0$; in fact, it is equivalent to $\Delta_p(A)\geq0$ \cite[Proposition 5.8]{CD-DivForm}. The case of $N>1$ is in striking contrast with this:  when $N>1$, the power function $F_p:\C^N\rightarrow[0,\infty)$ may fail to be $\bA$-convex even if $\Delta_p(\bA)>0$; see \cite[Example 7]{CD-Mixed}.

We may however fix this by perturbing the power function of $\omega\in\C^N$ by a sufficiently small linear combination of power functions of individual components of $\omega$. This is the content of the following lemma. It extends \cite[Section 4.3]{CD-Mixed}, where the case $N=2$ was treated and proven in a different, more geometrically flavoured manner.
The lemma will be used in Section \ref{s: general case}.

\begin{lemma}
\label{l: GCDFEsDCD}
If ${s}>2$ and $\bA:=(A_1,\hdots,A_N)\in\big(\C^{d\times d}\big)^N$ are such that $\Delta_{s}(\bA)>0$, then there exists $c=c({s},\bA,N)>0$ such that the function $P_{s}:\C^N\rightarrow[0,\infty)$, defined as
\begin{equation}
\label{eq: GFEsDCEs}
P_{s}(u_1,\hdots,u_N):=F_{s}(u_1,\hdots,u_N)+c\sum_{j=1}^NF_{s}(u_j),
\end{equation}
satisfies
$$
H^\bA_{P_{s}}[u;X]\ \geqsim\ |u|^{{s}-2}|X|^2
\hskip 30pt
\forall u\in\C^N, X\in\big(\C^d\big)^N.
$$
The implied constants depend on ${s}$, $\bA$ and $N$.
In particular, $P_{s}$ is $\bA$-convex.
\end{lemma}

\begin{proof}
For the moment we will assume just that $c>0$; the range of admissible $c$'s will be
getting restricted as the proof will progress.

We have, by Corollary \ref{c: Debeluh} and \eqref{eq: smoked bacon burger},
$$
H^\bA_{P_{s}}[u;X]
\ \geqsim\ |u|^{{s}-2}\left(|X|^2-\sigma\sum_{j<k}\frac{|u_j|}{|u|}\cdot\frac{|u_k|}{|u|}\cdot|X_j|\cdot|X_k|\right)+c\sum_{j=1}^N|u_j|^{{s}-2}|X_j|^2.
$$
Here $\sigma=\sigma({s},\bA)>0$.

By writing
$$
v_j:=\frac{|u_j|}{|u|}
\hskip 40pt
\text{and}
\hskip 40pt
y_j:=\frac{|X_j|}{|X|},
$$
and using the trivial inequality
$$
\sum_{j<k}(v_jy_j)(v_ky_k)
\leq
\frac{N-1}{2}\sum_{j=1}^N(v_jy_j)^2,
$$
we get
$$
H^\bA_{P_{s}}[u;X]
\ \geqsim\
|u|^{{s}-2}|X|^2
\left(1-\frac{(N-1)\sigma}{2}\sum_{j=1}^N(v_jy_j)^2+c\sum_{j=1}^Nv_j^{{s}-2}y_j^2\right).
$$
Denote $\tilde\sigma:=(N-1)\sigma/2$.

We see that it suffices to show the following: there exists $c=c({s},\bA,N)>0$ such that
\begin{equation*}
c\sum_{j=1}^Nv_j^{{s}-2}y_j^2+\frac12\geq{\tilde\sigma}\sum_{j=1}^N(v_jy_j)^2
\end{equation*}
for all $(v_1,\hdots,v_N),(y_1,\hdots,y_N)\in S^{N-1}\cap[0,\infty)^N$.

\medskip
If \underline{${s}\leq4$} then, since $v_j\leq1$, we have $v_j^{{s}-2}\geq v_j^2$, while for
$$
T:=\sum_{j=1}^N(v_jy_j)^2
$$
we clearly obtain $cT+1/2\geq{\tilde\sigma} T$ for $c\geq{\tilde\sigma}$.

\medskip
If \underline{${s}>4$}, write ${s}-2=2+\e$, where $\e>0$. 
Since $y_1^2+\hdots+y_N^2=1$,  the elements $\big\{y_1^2,\hdots,y_N^2\big\}$ represent a weighted counting measure on $\{1,\hdots,N\}$ of mass $1$. Therefore, recalling that $\e>0$, H\"older's inequality gives
$$
\sum_{j=1}^Nv_j^{2+\e}y_j^2\geq\Bigg(\sum_{j=1}^Nv_j^2y_j^2\Bigg)^{1+\e/2}.
$$
So it is enough to prove that there exists $c>0$ such that
$
cT^{1+\e/2}+1/2\geq{\tilde\sigma} T
$
for all $T\in[0,1]$.
An elementary analysis of the function $f:[0,\infty)\rightarrow\R$, defined by $f(T):=cT^{1+
\e/2}+1/2-{\tilde\sigma} T$, confirms the existence of such $c$.
\end{proof}

\subsection{Regularization}
\label{s: molto cantabile}
We would like to replace $\gX$ by a function which satisfies the inequality {\it \ref{eq: sova})} of Theorem \ref{t: Voltaren} but is, in addition, also of class $C^2$ everywhere on $\C^3$ (not only on $\C^3\backslash\Upsilon$). A standard way of achieving this involves {\it mollifiers}.

Fix $N\in\N$ and a radial function $\f\in C_c^\infty(\C^N)$ such that $0\leq\f\leq1$, ${\rm supp}\, \f\subseteq  B_{\C^{N}}(0,1)$ and $\int\f=1$.
For $\nu\in(0,1]$ and $\omega\in\C^{N}$ set $\varphi_\nu(\omega)=\nu^{-2N}\varphi(\omega/\nu)$.
If $\gZ:\C^N\rightarrow\R$ is locally integrable, let the {\it convolution}
$
\gZ*\f_\nu:\C^N\rightarrow\R
$
be defined for $\omega_0\in\C^N$ as
$$
(\gZ*\f_\nu)(\omega_0)
=\int_{\C^N}
\gZ(\omega_0-\omega)\f_\nu(\omega)
\wrt \omega.
$$
Explicit connection between this definition and the (perhaps slightly more standard) notion of convolution on real euclidean spaces goes, as expected, through the identificaton operator $\cW_{N,1}$:
if $\Psi:\C^N\rightarrow\R$, we define
$$
\int_{\C^N}\Psi(\omega)\wrt \omega:=\int_{\R^{2N}}\left(\Psi\circ\cW_{N,1}^{-1}\right)(x)\wrt x.
$$
We will frequently use the abbreviation $\gX_\nu=\gX*\f_\nu$.

Taking Theorem \ref{t: Voltaren} as a starting point and arguing as in \cite[Section 5.1]{CD-DivForm} we obtain the following estimates.
\begin{corollary}
\label{c: Naklofen}
Let $\gX$ be as in Theorem \ref{t: Voltaren}. Then for any $u,v,w\in\C$ and $\nu\in(0,1]$ we have:
\begin{enumerate}[a')]
\item
\label{eq: X1}
$$
\gX_\nu(u,v,w)\,\leqsim\,(|u|+\nu)^{p} + (|v|+\nu)^{q} + (|w|+\nu)^{r};
$$
\item
\label{eq: X2}
$$
\aligned
  |\partial_{\bar u}\gX_\nu|&\leqsim\left(\max\left\{(|u|+\nu)^{p},(|v|+\nu)^{q},(|w|+\nu)^{r}\right\}\right)^{1-1/p},\\
  |\partial_{\bar v}\gX_\nu|&\leqsim\left(\max\left\{(|u|+\nu)^{p},(|v|+\nu)^{q},(|w|+\nu)^{r}\right\}\right)^{1-1/q},\\
  |\partial_{\bar w}\gX_\nu|&\leqsim\,(|w|+\nu)^{r-1};
\endaligned
$$
\item
\label{eq: X3}
for $A,B,C$ satisfying the % same 
$*$-ellipticity conditions as in Theorem \ref{t: Voltaren} and almost every $x\in\Omega$,
\begin{equation*}
H^{(A,B,C)(x)}_{\gmX_\nu}[(u,v,w);(\zeta,\eta,\xi)]
\,\geqsim\,(|w|-\nu)|\zeta||\eta|
\hskip 30pt
\forall u,v,w\in\C,\ \zeta,\eta,\xi\in\C^d.
\end{equation*}
\end{enumerate}
The implied constants are the same as in the corresponding estimates of Theorem \ref{t: Voltaren}.
\end{corollary}

\section{Proof of Theorem \ref{t: Voltaren}}
\label{s: RICE}

We can assume that $p\geq q$. Indeed, suppose that in such a case the function $\gX$ from Theorem \ref{t: Voltaren} exists. Now take a triplet of exponents $(p,q,r)$ with $1/p+1/q+1/r=1$ and $1<p<q$ and a triplet of matrices $(A,B,C)$ which together with $(p,q,r)$ satisfy the assumptions of Theorem \ref{t: Voltaren}. It is easy to verify that in this case the desired function can be taken to be $\gX(u,v,w)=\widetilde\gX(v,u,w)$, where $\widetilde\gX$ is a function which Theorem \ref{t: Voltaren} gives in the case of the triplets $(q,p,r)$ and $(B,A,C)$.
In particular,
$$
H^{(A,B,C)}_{\gmX}[(u,v,w);(\zeta,\eta,\xi)]=H^{(B,A,C)}_{\widetilde\gmX}[(v,u,w);(\eta,\zeta,\xi)].
$$

So we assume that $p\geq q$. This implies \label{Transformers} $p>2$, for $1-2/p\geq1-1/p-1/q=1/r>0$.
As announced earlier, we will prove that the function from \cite{KS18}, constructed by the latter two authors of the present paper, can be marginally adapted so as to fit the requirements of Theorem \ref{t: Voltaren}. Before recalling the definition of the function and embarking on the proofs, let us introduce notation that we deem handy, since it simplifies the function's coefficients.

For $u\in\C$ and $p>0$ introduce the {\it ad hoc} notation
$$
[u]^{p}:=\frac{|u|^{p}}p,
$$
with $[u]:=[u]^{1}=|u|$.
Note that $p$ has to be interpreted simply as the {\it upper index} in $[u]^p$ and not as an exponent of a power. In other words, $[u]^p$ is not equal to the ``pure power'' $|u|^p$, but rather to a renormalized version of it. This will not cause confusions and, after all, the distinction will only affect the coefficients of the function $\gX$ below, making them appear more elegant. Also observe that, in fact, $[u]^p=G_p(u)$, with $G_p$ defined as in the previous section.

Next, we present the function from \cite{KS18}.

\subsection{\framebox{Case $p>q$}}
\label{subsec:case1}
Define
\renewcommand{\arraystretch}{2.4}
\begin{align*}
& \gX(u,v,w): = \\
& \left\{
\begin{array}{ll}
%%%%  DOMAIN 1  %%%%
[u]^{p}+ D [v]^{q} + E [w]^{r};
& |w|^r\leqslant |v|^q \leqslant |u|^{p}, \\
%%%%  DOMAIN 2  %%%%
{\displaystyle
[u]^{p}
+ D [v]^{1+q/p}[w]
+ \left(E -\frac{D }{1+q/p}\right)[w]^{r}};
& |v|^q\leqslant |w|^r \leqslant |u|^{p}, \\
%%%%  DOMAIN 3  %%%%
{\displaystyle
 [u]^{1+p/q}[w]
+ D [v]^{1+q/p} [w]
+ \left(E  -\frac{D +q/p}{1+q/p} \right)[w]^{r}
};
& |v|^q\leqslant |u|^{p}\leqslant |w|^r, \\
%%%%  DOMAIN 4  %%%%
{\displaystyle
(1-q/p)[u]^2[v]^{1-q/p}[w]
+\left(D -\frac{1-q/p}{2}\right)[v]^{1+q/p}[w]
+\left(E  -\frac{D  + q/p}{1+q/p} \right)[w]^{r}
};
& |u|^{p}\leqslant |v|^q\leqslant |w|^r, \\
%%%%  DOMAIN 5  %%%%
{\displaystyle
(1-q/p)   [u]^2[v]^{q-2q/p}
+   [u]^2[w]^{r-2r/p}
+ \left(D -\frac{1-q/p}{2}\right) [v]^{q}
+ (E-1/2)[w]^{r}
};
 & |u|^{p}\leqslant |w|^r\leqslant |v|^q, \\
%%%%  DOMAIN 6  %%%%
{\displaystyle
 \frac{1/r}{1-2/p}[u]^{p}
+ (1-q/p)[u]^2[v]^{q-2q/p}
+\left(D  - \frac{1-q/p}{2}\right)[v]^{q}
+E [w]^{r}
};
& |w|^r\leqslant |u|^{p} \leqslant |v|^q.
\end{array}
\right.
\end{align*}
\renewcommand{\arraystretch}{1}

\begin{remark}
The function $\gX$ admits a certain type of homogeneity, as explained below.

Write
$$
\aligned
 U&:=[u]^p, &  \hskip 20pt {a}&:=1/p,\\
 V&:=[v]^q, &  \hskip 20pt {b}&:=1/q,\\
 W&:=[w]^r, &  \hskip 20pt {c}&:=1/r
\endaligned
$$
and
$$
\aligned
\Omega_1&:=\left\{(u,v,w)\in\C^3\right.\!\!\! &;\ & \left. |w|^r\leq|v|^q\leq|u|^p\right\},\\
\Omega_2&:=\left\{\right. &;\ & \left. |v|^q\leq|w|^r\leq|u|^p\right\},\\
\Omega_3&:=\left\{\right. &;\ & \left. |v|^q\leq|u|^p\leq|w|^r\right\},\\
\Omega_4&:=\left\{\right. &;\ & \left. |u|^p\leq|v|^q\leq|w|^r\right\},\\
\Omega_5&:=\left\{\right. &;\ & \left. |u|^p\leq|w|^r\leq|v|^q\right\},\\
\Omega_6&:=\left\{\right. &;\ & \left. |w|^r\leq|u|^p\leq|v|^q\right\}.
\endaligned
$$

Then we can observe that $\gX(u,v,w)$ is a linear combination of
\begin{align*}
& \left\{
\begin{array}{ll}
%%%%  DOMAIN 1  %%%%
U,\, V,\, W
& \text{ in }\Omega_1, \\
%%%%  DOMAIN 2  %%%%
U,\, V^{1-c} W^{c} ,\, W
& \text{ in }\Omega_2, \\
%%%%  DOMAIN 3  %%%%
U^{1-c}W^{c} ,\, V^{1-c}W^{c},\,  W
& \text{ in }\Omega_3, \\
%%%%  DOMAIN 4  %%%%
U^{2a}V^{b-a}W^{c} ,\, V^{1-c}W^{c},\, W
& \text{ in }\Omega_4, \\
%%%%  DOMAIN 5  %%%%
U^{2a}V^{1-2a},\, U^{2a}W^{1-2a},\, V,\, W
& \text{ in }\Omega_5, \\
%%%%  DOMAIN 6  %%%%
U,\, U^{2a}W^{1-2a},\, V,\, W
& \text{ in }\Omega_6.
\end{array}
\right.
\end{align*}
\renewcommand{\arraystretch}{1}
\noindent
The above-said homogeneity is intended in the sense that the sum of exponents of any of the monomials above is always one.
Consequently, we have
$$
\gX\left(t^{1/p}u,t^{1/q}v,t^{1/r}w\right)=t\gX(u,v,w)
$$
valid for any $u,v,w\in\C$ and $t\geq0$.
\end{remark}

One has
$$
\gX(u,v,w)=\frac{1}{p}\cA_{r,p,q}(|w|,|u|,|v|),
$$
where $\cA=\cA_{p,q,r}$ (note the permutation of variables and indices) is the function from \cite[p.464]{KS18}, with
% $A=E p/r$, $B=1$ and $C= D p/q$.
an adequate choice of its parameters $A,B,C$.

This function is of class $C^1$ on $\C^3$ and of class $C^2$ on the complement of $\Upsilon$; see \cite{KS18}. The properties
{\it \ref{eq: Burger})} and {\it \ref{eq: Fest})} are rather straightforward to verify, therefore we focus on proving
{\it \ref{eq: sova})}.
First we notice that % our targeted inequality, \eqref{eq: Trifonov Kinderszenen}, 
the inequality \eqref{eq: Trifonov Kinderszenen}
is equivalent to
\begin{equation}
\label{eq:lowerest}
\widetilde{H}^{(A,B,C)}_{\gmX}[(u,v,w);(\alpha,\beta,\gamma)]
\,\geqsim\,|u||v||w||\alpha||\beta|.
\end{equation}
This is the inequality that we will actually be proving.

Using the estimates from Section \ref{s: mocne funkcije} we want to choose coefficients $D,E>0$, which may depend on $p,q,r$ and all the $*$-ellipticity constants of $A,B,C$ that were assumed to be positive in the formulation of Theorem \ref{t: Voltaren}, such that all of the coefficients above are positive and the estimate \eqref{eq:lowerest} holds in the interior of each of the six domains $\Omega_1,\hdots,\Omega_6$. Note that we will not have positivity of all ellipticity constants of $A,B,C$ appearing in the computation below. For instance, $\Delta_1(C)$ is never positive \cite{CD-DivForm}, but the only control of this quantity we will need is that $|\Delta_1(C)|$ is at most a constant times $\Lambda(C)$.

Write
$$
\aligned
(\sA,\sB,\sC)&:=(|\alpha|,|\beta|,|\gamma|)\\
(\su,\sv,\sw)&:=(|u|,|v|,|w|).
\endaligned
$$
Denote also $\Lambda:=\Lambda(A,B,C)$.

\subsubsection*{\underline{Domain \#1: $\sw^{r}<\sv^{q}<\su^{p}$}}

By combining \eqref{eq: Htilde} with Proposition \ref{p: Vinjerac} we have
$$
\aligned
\widetilde{H}^{(A,B,C)}_{\gmX}[(u,v,w);(\alpha,\beta,\gamma)]
 &\geq
  \frac{p \Delta_{p}(A)}{2}  \su^{p} \sA^2
 + \frac{D q \Delta_{q}(B)}{2} \sv^{q} \sB ^2
 + \frac{E r \Delta_{r}(C)}{2}  \sw^{r} \sC^2
 \\
 &\geq \sqrt{D pq\Delta_{p}(A) \Delta_{q}(B)\su^p\sv^q}\, \sA   \sB + 0.
\endaligned
$$
Taking into account the characteristic inequalities of the current subdomain, and recalling that $p\geq2$, gives
$$
\sqrt{\su^p\sv^q}
=\su\cdot\su^{p(1/2-1/p)}\sv^{q/2}
\geq\su\cdot\sv^{q(1/2-1/p)}\sv^{q/2}
=\su\cdot\sv\cdot\sv^{q/r}
\geq\su\sv\sw.
$$
Hence we proved
$$
\widetilde{H}^{(A,B,C)}_{\gmX}[(u,v,w);(\alpha,\beta,\gamma)]
\geq
\sqrt{D pq\Delta_{p}(A) \Delta_{q}(B)}\,  \su \sv \sw  \sA   \sB .
$$

\subsubsection*{\underline{Domain \#2: $\sv^{q}<\sw^{r}<\su^{p}$}}

By combining \eqref{eq: Htilde} and Propositions \ref{p: Vinjerac}, \ref{p: Metajna}, we obtain
\begin{align*}
\widetilde{H}^{(A,B,C)}_{\gmX}&[(u,v,w);(\alpha,\beta,\gamma)]\\
&\geq
\frac{p}2\Delta_{p}(A)\su^{p}\sA ^2
\\
&
\hskip 15pt
+D\sv^{1+q/p}\sw
\left(
  \frac{ (1+q/p)\Delta_{1+q/p}(B)}{2}\sB^2
   -2\Lambda\sB \sC + \frac{\Delta_1(C)}{2(1+q/p)}\sC^2\right)
   \\
& \hskip 15pt
+\left(E -\frac{D}{1+q/p}\right)
\frac{r}2\Delta_{r}(C)\sw^{r}\sC^2.
\end{align*}

Recall that $\Delta_1(C)\leq0$ always and that $\Delta_{1+q/p}(B)>0$ by our assumption.
Moreover, in the current domain we have
$\sv^{1+q/p}\sw\leq\sw^{r}$.
Hence we may continue as
\begin{align}
\widetilde{H}^{(A,B,C)}_{\gmX}&[(u,v,w);(\alpha,\beta,\gamma)]\nonumber \\
&\geq
\label{eq: Mtkvari}
\frac{p\Delta_{p}(A)}2\su^{p}\sA^2
+\frac{D (1+q/p)\Delta_{1+q/p}(B)}{2}\sv^{1+q/p}\sw\sB^2
\\
&
\hskip 15pt
\label{eq: Narikala}
-(2D \Lambda)\sv^{1+q/p}\sw\sB\sC
+\left(\frac{E r\Delta_{r}(C)}2
 -\frac{D(r\Delta_{r}(C)+|\Delta_1(C)|)}{2(1+q/p)}
\right)
\sv^{1+q/p}\sw\sC^2.
\end{align}

The key term is the one containing $\sB^2$.
We will split it into a sum of two parts, by splitting the factor $1+q/p$ in that term: one summand (for example, the one containing 1) will be added to the bottom row \eqref{eq: Narikala} in order to make it nonnegative, while the other summand (containing $q/p$) will be left in the middle row \eqref{eq: Mtkvari} for the purpose of obtaining \eqref{eq:lowerest}.

In the current domain we have
$$
\su^{p-2}\geq\sw^{r(p-2)/p}=\sw^{r(1/q-1/p)}\cdot\sw\geq\sv^{1-q/p}\sw,
$$
therefore
$$
\su^{p}\sv^{1+q/p}\sw=\su^{p-2}\cdot\left(\su^{2}\sv^{1+q/p}\sw\right)\geq(\su\sv\sw)^2.
$$
Hence \eqref{eq: Mtkvari} can be estimated as
$$
\aligned
&\geq
\sqrt{Dq\Delta_{p}(A)\Delta_{1+q/p}(B)}\,
\su\sv\sw\sA\sB.
\endaligned
$$

Regarding \eqref{eq: Narikala}, when augmented by the other term with $\sB^2$ and divided by $\sv^{1+q/p}\sw$, it becomes
\begin{equation}
\label{eq: The Great British Bake Off}
\geq
\frac{D \Delta_{1+q/p}(B)}{2}\sB^2
-(2D \Lambda)\sB\sC
+\left(
  \frac{Er\Delta_{r}(C)}2
 -\frac{D(r\Delta_{r}(C)+|\Delta_1(C)|)}{2(1+q/p)}
\right)
\sC^2.
\end{equation}
Clearly, if $E$ is large enough, the above expression becomes nonnegative uniformly in $\sB,\sC$.

\subsubsection*{\underline{Domain \#3: $\sv^{q}<\su^{p}<\sw^{r}$}}
By combining \eqref{eq: Htilde} and Propositions \ref{p: Vinjerac}, \ref{p: Metajna} we get
\begin{align*}
\widetilde{H}^{(A,B,C)}_{\gmX}&[(u,v,w);(\alpha,\beta,\gamma)]\\
&\geq
\hskip 22.7pt
\su^{1+p/q}\sw
\left(
  \frac{ (1+p/q)\Delta_{1+p/q}(A)}{2}\sA^2
   -2\Lambda\sA\sC
   +  \frac{\Delta_1(C)}{2(1+p/q)}\sC^2
\right)\\
&
\hskip 15pt
+D\sv^{1+q/p}\sw
\left(
  \frac{ (1+q/p)\Delta_{1+q/p}(B)}{2}\sB^2
   -2\Lambda\sB\sC
   +  \frac{\Delta_1(C)}{2(1+q/p)}\sC^2
\right) \\
&
\hskip 15pt
+   \left(E -\frac{D +q/p}{1+q/p}\right)
\frac{r}2\Delta_{r}(C)\sw^{r}\sC^2.
\end{align*}

Similarly as before, we actually estimate explicitly the terms with $\sA^2$ and $\sB^2$, while the terms containing $\sC$ will just be dismissed as positive, which will be achieved by the splitting trick, used in domain \#2, and the parameter $E$ being sufficiently large. This means that we break up each of the terms containing $\sA^2$ or $\sB^2$ into a sum of two parts, by splitting the factors $1+p/q$ and $1+q/p$ in those terms, in the sense of treating each of the summands separately. 

Parts of the terms with $\sA^2$, $\sB^2$ containing the summand 1 from $1+p/q$ or $1+q/p$ will serve for dismissing the terms with $\sC$, for which we apply the estimates $\sv ^{1+q/p}\sw\leq\su^{1+p/q}\sw\leq\sw^{r}$.

On the other hand, parts of the terms with $\sA^2$, $\sB^2$ containing $p/q$ resp. $q/p$ will get estimated through the inequality $\su^{1+p/q}\sv^{1+q/p}\geq(\su \sv )^2$ as follows:
$$
\frac{\sw}{2}
\left(
\frac{p}{q}\Delta_{1+p/q}(A)\su^{1+p/q}\sA^2
+
\frac{Dq}{p} \Delta_{1+q/p}(B)\sv^{1+q/p}\sB^2
\right)
\geq\sqrt{ D  \Delta_{1+p/q}(A) \Delta_{1+q/p}(B) }\
\su \sv \sw \sA \sB .
$$

\subsubsection*{\underline{Domain \#4: $\su^{p}<\sv^{q}<\sw^{r}$}}

By combining \eqref{eq: Htilde} and Propositions \ref{p: Vinjerac}, \ref{p: Metajna}, \ref{p: Zhenodraga}, we have
\begin{align*}
\widetilde{H}^{(A,B,C)}_{\gmX}&[(u,v,w);(\alpha,\beta,\gamma)]\\
&\geq
\frac{\su^2\sv^{1-q/p}\sw}{2}
\bigg[
2\Delta_{2}(A)\sA^2+\frac{(1-q/p)^2}2\Delta_{1-q/p}(B)\sB^2+\frac{1}2\Delta_{1}(C)\sC^2
\\
&\hskip 100pt
-2\Lambda
\left(
2(1-q/p)\sA\sB
+2\sA\sC
+(1-q/p)\sB\sC
\right)\bigg]
\\
&\hskip 15pt
+\left(D  -\frac{1-q/p}{2}\right)\sv^{1+q/p}\sw
\left(
 \frac{ (1+q/p)\Delta_{1+q/p}(B)}{2}\sB^2
   -2\Lambda\sB\sC
   + \frac{\Delta_1(C)}{2(1+q/p)}\sC^2
\right)
\\
&\hskip 15pt
 +  \left(E -\frac{D + q/p}{1+q/p}\right)
\frac{r}2\Delta_{r}(C)\sw^{r}\sC^2.
\end{align*}

By using that in the current domain we have
$$
\su^{2}\sv^{1-q/p}\sw
\leq\su\sv\sw
\leq\sv^{1+q/p}\sw
\leq\sw^{r},
$$
let us rewrite the above estimate of $\widetilde{H}^{(A,B,C)}_{\gmX}$
in a manner that focuses on the essential:
for certain $\gamma_1,\dots,\gamma_{10}>0$, depending on $p,q,r$ and the $*$-ellipticity constants of $A,B,C$, we have
\begin{align*}
\ \hskip -300pt &&
\widetilde{H}^{(A,B,C)}_{\gmX}[(u,v,w);(\alpha,\beta,\gamma)]\hskip 150pt \\
&&\hskip -50pt \geq\hskip 50pt
%%%%  AA  %%%%%
 \gamma_1\ \su^{2}\sv^{1-q/p}\sw\ \sA^2\hskip 100pt\\
%%%%  BB  %%%%%
&&\hskip 10pt +(\gamma_2D-\gamma_3)\ \sv^{1+q/p}\sw\ \sB^2\hskip 100pt\\
%%%%  CC  %%%%%
&&\hskip 10pt +(\gamma_4E-\gamma_5D-\gamma_6)\ \sw^{r}\ \sC^2\hskip 100pt\\
%%%%  AB  %%%%%
&&\hskip 10pt -\gamma_7\ \su\sv\sw\ \sA\sB\hskip 100pt\\
%%%%  AC  %%%%%
&&\hskip 10pt - \gamma_8\ \su^{2}\sv^{1-q/p}\sw\ \sA\sC\hskip 100pt\\
%%%%  BC  %%%%%
&&\hskip 10pt -(\gamma_9D+\gamma_{10})\ \sv^{1+q/p}\sw\ \sB\sC.\hskip 100pt\!\!
\end{align*}
In order to obtain \eqref{eq:lowerest} we first estimate
$$
\aligned
\frac12
\left(
   \gamma_1\ \su^{2}\sv^{1-q/p}\sw\ \sA^2
   + (\gamma_2D-\gamma_3)\ \sv^{1+q/p}\sw\ \sB^2
\right)
&\geq \sqrt{\gamma_1(\gamma_2D-\gamma_3)}\su\sv\sw\sA\sB.
\endaligned
$$
Choose $D$ large enough so that $\sqrt{\gamma_1(\gamma_2D-\gamma_3)}>\gamma_7$.
Since in the current domain we also have
$$
\aligned
\sqrt{\su^{2}\sv^{1-q/p}\sw^{r+1}} & \geq \su^{2}\sv^{1-q/p}\sw\\
\sqrt{\sv^{1+q/p}\sw^{r+1}} & \geq \sv^{1+q/p}\sw,
\endaligned
$$
we may next choose $E$ large enough so that, simultaneously,
\begin{align*}
\frac12\left(
\gamma_1\ \su^{2}\sv^{1-q/p}\sw\ \sA^2
+ (\gamma_4E-\gamma_5D-\gamma_6)\ \sw^{r}\ \sC^2
\right)
&\geq\gamma_8\ \su^{2}\sv^{1-q/p}\sw\ \sA\sC\\
\frac12
\left(
 (\gamma_2D-\gamma_3)\ \sv^{1+q/p}\sw\ \sB^2
 + (\gamma_4E-\gamma_5D-\gamma_6)\ \sw^{r}\ \sC^2
\right)
& \geq (\gamma_9D+\gamma_{10})\ \sv^{1+q/p}\sw\ \sB\sC.
\end{align*}
By combining all the above inequalities, we prove, for such choices of $D$ and $E$, that
\begin{equation}
\label{eq: minuta}
\widetilde{H}^{(A,B,C)}_{\gmX}[(u,v,w);(\alpha,\beta,\gamma)]
\geq\left(
\sqrt{\gamma_1(\gamma_2D-\gamma_3)}-\gamma_7\right)\su\sv\sw\sA\sB.
\end{equation}
Recall that, by our choice of $D$, the constant in parentheses is strictly positive.

\subsubsection*{\underline{Domain \#5: $\su^{p}<\sw^{r}<\sv^{q}$}}
By combining \eqref{eq: Htilde} and Propositions \ref{p: Vinjerac} and \ref{p: Metajna}, we obtain
\begin{align*}
\widetilde{H}^{(A,B,C)}_{\gmX}&[(u,v,w);(\alpha,\beta,\gamma)] \\
& \geq
(1-q/p)
\su^{2}\sv^{q-2q/p}
\left(
\frac{\Delta_2(A)}{q-2q/p}\sA^2
-2\Lambda\sA\sB
+\frac{(q-2q/p)\Delta_{q-2q/p}(B)}{4}\sB^2
\right)\\
&\hskip 45pt
+ \su^{2} \sw^{r-2r/p}
\left(
\frac{\Delta_2(A)}{r-2r/p}\sA^2
-2\Lambda\sA\sC
+\frac{(r-2r/p)\Delta_{r-2r/p}(C)}{4}\sC^2
\right)\\
&\hskip 45pt
+  \left(D -\frac{1-q/p}{2}\right)
\frac{q}{2}\Delta_{q}(B)\sv^{q}\sB^2\\
&\hskip 45pt
+\left(E -1/2\right)
\frac{r}{2}\Delta_{r}(C)\sw^{r}\sC^2.
\end{align*}

In the current domain we have $\sw^{r}\geq \su^{2}\sw^{r-2r/p}$.
Furthermore, $p>q$ implies that $p>2$, therefore $r-2r/p>0$.
Putting this together, we see that the
sum of the third and the fifth line above
is nonnegative for $E$ sufficiently large.

Regarding the sum of the second and the fourth line,
use that $\su^{2}\sv^{q-2q/p}\leq\su\sv^{1+q/r}\leq\sv^{q}$, which returns (for $E$ large as explained above),
\begin{align*}
\widetilde{H}^{(A,B,C)}_{\gmX}&[(u,v,w);(\alpha,\beta,\gamma)] \\
& \geq
\frac{1/q-1/p}{1-2/p}\Delta_2(A)
\su^{2}\sv^{q-2q/p}\sA^2
-2\Lambda(1-q/p)\su\sv^{1+q/r}\sA\sB
\\
&\hskip 15pt
+ \frac{q}{2} \bigg[
\left(D -\frac{1-q/p}{2}\right)
\Delta_{q}(B)
-\frac{(1-q/p)(1-2/p)|\Delta_{q-2q/p}(B)|}{2}
\bigg]
\sv^{q}\sB^2.
\end{align*}
Since $\sqrt{\su^{2}\sv^{q-2q/p}\cdot\sv^{q}}=\su\sv^{1+q/r}\geq\su\sv\sw$,
the last two lines can be estimated as $\geqsim\su\sv\sw\sA\sB$ if $D$ is sufficiently large.

\subsubsection*{\underline{Domain \#6: $\sw^{r}<\su^{p}<\sv^{q}$}}
We have

\begin{align*}
\widetilde{H}^{(A,B,C)}_{\gmX}&[(u,v,w);(\alpha,\beta,\gamma)] \\
& \geq
\frac{1/r}{1-2/p}\cdot\frac{p}{2}\Delta_{p}(A)\su^{p}\sA^2
\\
&\hskip 15pt
+ (1-q/p)
\su^{2}\sv^{q-2q/p}
\left(
\frac{\Delta_2(A)}{q-2q/p}\sA^2
-2\Lambda\sA\sB+\frac{(q-2q/p)\Delta_{q-2q/p}(B)}{4}\sB^2
\right)\\
&\hskip 15pt
+  \left(D -\frac{1-q/p}{2}\right)
\frac{q}{2}\Delta_{q}(B)\sv^{q}\sB^2\\
&\hskip 15pt
+\frac{Er}{2}\Delta_{r}(C)\sw^{r}\sC^2.
\end{align*}
The first and the last term on the right-hand side
are positive by our assumptions, while the remaining two can be estimated exactly as in Domain \#5, since we have the same expression and the relations between $\su,\sv$ are also the same.

\subsection{\framebox{Case $p=q$}}
\label{subsec:case2}
Define
\renewcommand{\arraystretch}{1.2}
\begin{align*}
 \gX(u,v,w) :=
 \left\{
\begin{array}{ll}
%%%%  DOMAIN I  %%%%
[u]^p + [v]^p + E[w]^r;
& |w|^r \leqslant \min\{|u|^p,|v|^p\}, \\[1mm]
%%%%  DOMAIN II  %%%%
{\displaystyle
[u]^p + [v]^{2}[w] + (E-1/2)[w]^r};
& |v|^p \leqslant |w|^r \leqslant |u|^p, \\[1mm]
%%%%  DOMAIN III  %%%%
{\displaystyle
[u]^{2}[w] + [v]^p + (E-1/2)[w]^r};
& |u|^p \leqslant |w|^r\leqslant |v|^p, \\[1mm]
%%%%  DOMAIN IV  %%%%
{\displaystyle
[u]^2[w] + [v]^{2}[w] +  (E-1)[w]^r};
& \max\{|u|^p,|v|^p\}\leqslant |w|^r.
\end{array}
\right.
\end{align*}
\renewcommand{\arraystretch}{1}
This function is of class $C^1$ on $\C^3$ and of class $C^2$ outside of $\Upsilon$. It satisfies estimates {\it \ref{eq: Burger})} and {\it \ref{eq: Fest})} of Theorem \ref{t: Voltaren}. Now let us prove {\it \ref{eq: sova})}.
This time we only have available the positivity of the following $*$-ellipticity constants: $\Delta_{p}(A)$, $\Delta_{p}(B)$, $\Delta_{r}(C)$ and, consequently, also that of $\Delta_{2}(A)$, $\Delta_{2}(B)$.

\subsubsection*{\underline{Domain \#1: $\sw^{r}<\min\{\su^{p},\sv^{p}\}$}}
We have
$$
\aligned
\widetilde{H}^{(A,B,C)}_{\gmX}[(u,v,w);(\alpha,\beta,\gamma)]
&\geq
   \frac{p}{2}  \Delta_{p}(A) \su^{p} \sA^2
+ \frac{p}{2}  \Delta_{p}(B) \sv^{p} \sB^2
+ \frac{Er}{2}  \Delta_{r}(C) \sw^{r} \sC^2
\\
&\geq
p\sqrt{\Delta_{p}(A)\Delta_{p}(B)\su^{p}\sv^{p}}\sA\sB+0.
\endaligned
$$
We conclude noting that $\su^{p}\geq\su^2\sw$ and $\sv^{p}\geq\sv^2\sw$.

\subsubsection*{\underline{Domain \#2: $\sv^{p}<\sw^{r}<\su^{p}$}}
We have
$$
\aligned
\widetilde{H}^{(A,B,C)}_{\gmX}[(u,v,w);(\alpha,\beta,\gamma)]
&\geq
 \frac{p}{2}  \Delta_{p}(A) \su^{p} \sA^2
  \\
&\hskip 15pt
+\frac{1}{2}\,\sv^{2}\sw
\left(2\Delta_2(B)\sB^2-4\Lambda(B,C)\sB\sC+\frac{\Delta_1(C)}{2}\sC^2\right)\\
&\hskip15pt
+\left(
       E-\frac{1}{2}
  \right)
\frac{r}{2}  \Delta_{r}(C) \sw^{r} \sC^2.
\endaligned
$$
By the splitting trick and using that $\sw^{r}\geq\sv^2\sw$ and $\su^{p}\sv^2\sw\geq(\su\sv\sw)^2$, we get the desired estimate for adequately large $E$.

\subsubsection*{\underline{Domain \#3: $\su^{p}<\sw^{r}<\sv^{p}$}}
This case is completely symmetric to the previous one (with respect to switching $\su\leftrightarrow\sv$) and so is the desired conclusion, so there is nothing left to prove.

\subsubsection*{\underline{Domain \#4: $\max\{\su^{p},\sv^{p}\}<\sw^{r}$}}
We have
$$
\aligned
\widetilde{H}^{(A,B,C)}_{\gmX}[(u,v,w);(\alpha,\beta,\gamma)]
&\geq
\frac{(E-1)r}{2}  \Delta_{r}(C) \sw^{r} \sC^2 \\
&\hskip 15pt
+\frac{1}{2}\,\su^{2}\sw
\left(
2\Delta_2(A)\sA^2-4\Lambda(A,C)\sA\sC+\frac{\Delta_1(C)}{2}\sC^2
\right)\\
&\hskip 15pt
+\frac{1}{2}\,\sv^{2}\sw
\left(
2\Delta_2(B)\sB^2-4\Lambda(B,C)\sB\sC+\frac{\Delta_1(C)}{2}\sC^2
\right).
\endaligned
$$
Again we split the two terms involving $\sA^2$ and $\sB^2$ into two halves. One part we estimate as
$$
\left(\Delta_2(A)\su^{2}\sA^2+\Delta_2(B)\sv^{2}\sB^2\right)\sw\geq2\sqrt{\Delta_2(A)\Delta_2(B)}\ \su\sv\sw\sA\sB.
$$
The remaining parts of the terms involving $\sA^2$ and $\sB^2$ we add to the other terms so as to make them
 nonnegative for large $E$, just as we have been doing in the case $p>q$ (this time we use that $\sw^{r}\geq\max\{\su^2\sw,\sv^2\sw\})$.

The proof of Theorem \ref{t: Voltaren} is now complete. \qed

\begin{remark}
Verification of convexity properties of the function $\gX$ performed in Sections~\ref{subsec:case1} and \ref{subsec:case2} can be automatized to a large extent using a \emph{computer algebra system}. In fact, this has already been done in \cite{KS18}, but for a weaker (i.e., purely scalar) type of convexity (case $A=B=C=I$). In \cite{KS18} a computer-based algebraic manipulation also helped in finding the exact formula for $\gX(u,v,w)$. In the present paper we included the ``manual'' verification of the aforementioned properties in order to show that their checking is actually manageable without any aid of computer. This even resulted in our simplifying a few lengthy expressions from \cite{KS18}. Still, we redid computer-based verification for our own peace of mind.
\end{remark}

\subsection{Supplementary estimate}
The next estimate reflects (i.e., is in the spirit of) \cite[(30)]{CD-Mixed}.
It may be viewed as a supplement to Theorem \ref{t: Voltaren}.
It will be used in Section \ref{s: general case}.

\begin{proposition}
\label{p: Drska:CCA 92:46}
Function $\gX$ admits the estimates
\begin{equation}
\label{eq: 20022020}
\aligned
\left|D^2\gX(u,v,w)\right|\ \leqsim
& \ \sv^{-(1-q/p)}\sw\\
& + \su ^{p/q}  + \su^{p-2} \\
& + \sv ^{q/p} + \sv^{q-2} \hskip3pt +\sv^{q(1-2/p)} \hskip2pt  + \sv\\
& \hskip 32pt + \sw^{r-2}  + \sw^{r(1-2/p)} + \sw.
\endaligned
\end{equation}
Here, as before, $(\su,\sv,\sw):=(|u|,|v|,|w|)$.

The implied constants depend on $p,q,r,A,B,C$ and their $*$-ellipticity constants.
\end{proposition}

\begin{proof}
As usual when dealing with $\gX$, one considers each of the six defining domains separately. We will present the complete proof in one of them only (case $p>q$, domain \#2), as an exemplary sample; the others will be left out.

Recall that $\gX$ is constituted of tensor products of power functions $F_s(\zeta)=|\zeta|^s$ with 1, 2 or 3 nontrivial factors. Let us single out their estimates:
$$
\aligned
|D^2F_{p}(u)|\ & \leqsim\,\su^{p-2}\\
|D^2(F_{p}\otimes F_{t})(u,v)|\ &\leqsim\,\su^{p-2}\sv^{t}+\su^{p-1}\sv^{{t}-1}+\su^{p}\sv^{t-2}\\
|D^2(F_{p}\otimes F_{t}\otimes F_{s})(u,v,w)|
\ &\leqsim\hskip 4pt
\su^{p-2}\sv^{t}\sw^{s}+
\su^{p-1}\sv^{t-1}\sw^{s}+
\su^{p-1}\sv^{t}\sw^{s-1}\\
&\hskip 3pt+\su^{p}\sv^{t-1}\sw^{s-1}+
\su^{p}\sv^{t-2}\sw^{s}+
\su^{p}\sv^{t}\sw^{s-2}.
\endaligned
$$

\medskip
\underline{Upper estimate of $|D^2\gX|$ in case $p>q$, domain \#2.}

Here we have $\sv^q<\sw^r<\su^p$ and $\gX(u,v,w)$
is a linear combination of $\su^p$, $\sv^{1+q/p}\sw$ and  $\sw^r$.
We obtain
$$
|D^2\gX| \ \leqsim \ \su^{p-2}  + (\sv^{q/p-1}\sw + \sv^{q/p} + \sv^{1+q/p}\sw^{-1}) + \sw^{r-2}.
$$
Use that in the given domain we have
$
\sv^{1+q/p} \sw^{-1} \leq \sw^{r-2}.
$
Consequently,
$$
|D^2\gX| \leqsim \ \su^{p-2} + \sv^{q/p}  + \sw^{r-2} + \sv^{q/p-1}\sw.
$$

Of course, other terms in \eqref{eq: 20022020} are contributions of analogous estimates in other regions.
\end{proof}

\subsection{Auxiliary results on the H\"{o}lder triples of exponents}

Here we gather a few simple results regarding the types of triples $(p,q,r)$ that we are considering in this paper --- that is, triples of numbers $p,q,r\in(1,\infty)$ satisfying \eqref{eq:Hoeldexppqr}.
Such $(p,q,r)$ are said to be the {\it H\"{o}lder triple of exponents} or to fulfill the {\it H\"{o}lder scaling condition}.

The first two results below are motivated by the estimates of $D\gX$ and $D^2\gX$; see
Theorem \ref{t: Voltaren} {\it (\ref{eq: Fest})} and Proposition \ref{p: Drska:CCA 92:46}, respectively.
They
have very similar proofs, so we will only prove the second one of these two.

\begin{lemma}
\label{l: T34}
For a H\"{o}lder triple of exponents $p,q,r\in(1,\infty)$ satisfying $p\geq q$ denote
$$
\cP:=\left\{p-1,\,q-1,\,r-1,\,q(1-1/p),\,r(1-1/p),\,p(1-1/q),\,r(1-1/q)\right\}.
$$
Then
$$
\aligned
\max\cP&=\max\{p,r\}-1\\
\min\cP&=\min\{q,r\}-1.
\endaligned
$$
\end{lemma}

\begin{lemma}
\label{l: fade away}
For a H\"{o}lder triple of exponents $p,q,r\in(1,\infty)$ satisfying $p\geq q$ denote
$$
\cS:=
\left\{
p-2,\,q-2,\,r-2,\,1,\,q(1-2/p),\,r(1-2/p),p/q,q/p
\right\}.
$$
Then
$$
\aligned
\max\cS&=\max\{p,r\}-2\\
\min\cS&=\min\{q,r\}-2.
\endaligned
$$
\end{lemma}

\begin{proof}
Write $M:=\max\{p,r\}$ and $\gm:=\min\{q,r\}$.
Recall that $p>2$, as we saw on page \pageref{Transformers}.
Similarly, $\gm\leq3\leq M$.

Let us start with the first identity. We have:
\begin{itemize}
\item
$q-2\leq p-2\leq \max\{p-2,r-2\}= M-2$;
\item
$q(1-2/p)\leq p(1-2/p)=p-2\leq M-2;$
\item
$r(1-2/p)\leq M(1-2/M)=M-2$;
\item
$q/p\leq1\leq
p/q
=p(1-1/p-1/r)
\leq
M(1-1/M-1/M)=M-2$.
\hfill{\rm ($\spadesuit$)}
\end{itemize}

By mirroring the above inequalities we prove the second identity.
%identity:
%\begin{itemize}
%\item
%$p-2\geq q-2\geq \min\{q-2,r-2\}=\gm-2$;
%\item
%$q(1-2/p)\geq q(1-2/q)=q-2\geq\gm-2$;
%\item
%$r(1-2/p)
%\geq r(1-2/q)
%\geq\gm(1-2/\gm)=\gm-2$;
%\item
%$p/q\geq1\geq q/p=q(1-1/q-1/r)\geq \gm(1-1/\gm-1/\gm)=\gm-2$.
%\end{itemize}
%
%This finishes the proof of the lemma.
\end{proof}

The last auxiliary result of this section will be used for the proof of Corollary \ref{c: 5AM} below.

\begin{lemma}
\label{l: Ce rau am facut la lume}
Take a H\"{o}lder triple of exponents $p,q,r\in(1,\infty)$. Then
$$
p,q,r,1+p/q,1+q/p\in[M',M],
$$
where $M:=\max\{p,q,r\}$ and $M'$ is the conjugate exponent of $M$.
\end{lemma}

\begin{proof}
Since the assumptions and the desired conclusions of the lemma are invariant under the reversal of $p$ and $q$, we may assume that $p\geq q$.

From {\rm ($\spadesuit$)} above we see that $2\leq 1+p/q\leq M-1<M$. Since $1+q/p$ is the conjugate exponent to $1+p/q$, it follows that $M'\leq 1+q/p\leq2$.

It remains to show that $p,q,r\in[M',M]$. We prove this when $M=p$; the proof in the case $M=r$ is identical.
Since $1/p'=1/q+1/r$, one gets $1/p'>\max\{1/q,1/r\}$, meaning that $M'=p'<\min\{q,r\}$.
\end{proof}

\subsection{More about the constants in Theorem \ref{t: Voltaren}}
A scrutinous analysis of the proof of Theorem \ref{t: Voltaren} reveals information on the growth of constants there. We record some of these pieces of information here, as we may need them later.

\begin{proposition}
\label{p: Pat Riley}
Let $\gX$ be the Bellman function from Section \ref{subsec:case1} adapted to the case $p>q$ and $D,E$ the parameters that appear in its definition.
The estimates of $\gX$, as specified 
in Theorem \ref{t: Voltaren}, involve constants that can be bounded as follows:
\begin{itemize}
\item
in {\it \ref{eq: Burger})},
\begin{equation}
\gX \leqslant \frac{1}{p}|u|^p + \frac{D}{q}|v|^q + \frac{E}{r}|w|^r;
\label{eq:est_gX}
\end{equation}
\item
in {\it \ref{eq: sova})},
$$
H_\gmX^{(A,B,C)}
\geq
\left(\sqrt{\alpha_1D-\alpha_2}-\alpha_3\right)|w||\zeta||\eta|,
$$
where the constants $\alpha_{1,2,3}$ are all positive and depend on $p,q$ and the $*$-ellipticity and $L^\infty$ constants of $A,B$, while $D$ is chosen in a manner such that  $\sqrt{\alpha_1D-\alpha_2}-\alpha_3>0$.
For example,
\begin{equation}
\label{2544}
D=D_{p,q,A,B}=\frac{\alpha_2+(\alpha_3+1)^2}{\alpha_1}\,,
\end{equation}
making $\sqrt{\alpha_1D-\alpha_2}-\alpha_3=1$.
Such a choice is possible if $E$ is large enough.
\end{itemize}

Assume that $E$ is admissible in the sense of $\gX$ satisfying the conclusions of Theorem \ref{t: Voltaren}.
Then for any $0<\delta\leq1$, the constant $\delta^{-1}E$ is admissible if we replace the triple $(A,B,C)$ by ($A,B,\delta C)$.

\smallskip
In the special case $C=I$, we may arrange for $E$ to stay bounded as we move $r\rightarrow\infty$ in the sense of sending $\e\rightarrow0$ in \eqref{eq: Dallas vs. Clippers Game 6}.
\end{proposition}

\begin{proof}
Consider each of the six domains separately and put together the calculations performed there. For parts \eqref{eq: Burger} and \eqref{eq: sova} they can be done rather easily. The size of $E$ is mostly determined by the requirement that the ``discarded'' terms in the calculation of $D^2\gX$ in each of the six subdomains -- indeed be positive. For example, in the subdomain $\Omega_4$, the identity \eqref{2544} reflects \eqref{eq: minuta} from page \pageref{eq: minuta}.

For the sake of completeness, we will present the proof of \eqref{eq:est_gX} in subdomain $\Omega_4$ where the function $\gX$ is defined as
\[\gX(u,v,w)=\frac{1}{2}|u|^2|v|^{1-q/p}|w|+\frac{1}{1+q/p}\left(D-\frac{1-q/p}{2}\right)|v|^{1+q/p}|w|+\left(\frac{E}{r}-\frac{D+q/p}{r(1+q/p)}\right)|w|^r.\]
We can apply Young's inequality to get
\begin{align*}
|u|^2|v|^{1-q/p}|w| &\leqslant  \frac{2}{p}|u|^p+\left(\frac{1}{q}-\frac{1}{p}\right)|v|^q+\frac{1}{r}|w|^r\\
|v|^{1+q/p}|w| &\leqslant \left(1-\frac{1}{r}\right)|v|^q+\frac{1}{r}|w|^r
\end{align*}
and hence
\begin{align*}
\gX(u,v,w) &\leqslant \frac{1}{2}\cdot \frac{2}{p}|u|^p+\left[\frac{1}{2}\left(\frac{1}{q}-\frac{1}{p}\right)+\frac{1}{1+q/p}\left(D-\frac{1-q/p}{2}\right)\cdot\left(1-\frac{1}{r}\right)\right]|v|^q \\
&\qquad + \left[\frac{1}{2r}+\frac{1}{r(1+q/p)}\left(D-\frac{1-q/p}{2}\right)+\frac{E}{r}-\frac{D+q/p}{r(1+q/p)}\right]|w|^r\\
&=\frac{1}{p}|u|^p + \frac{D}{q}|v|^q + \frac{E}{r}|w|^r,
\end{align*}
which gives us exactly \eqref{eq:est_gX}. All the other subdomains are treated in a similar way.

\smallskip
Let us now address the second statement. Again we present the proof in the case of only one subdomain, namely, $\Omega_2$. The principle that works there also applies to other subdomains.

Starting from \eqref{eq: The Great British Bake Off}, in order for it to become nonnegative, we see that it suffices for $E$ to satisfy the inequality
\begin{equation}
\label{eq: ramen}
\left(
  Er\Delta_{r}(C)
 -\frac{D(r\Delta_{r}(C)+|\Delta_1(C)|)}{1+q/p}
\right)
\Delta_{1+q/p}(B)
\geq
4D \Lambda^2.
\end{equation}
Replace $C$ by $\delta C$, write $E_\delta$ for the corresponding constant (that we are looking for),
and notice that $\Delta_s(\delta C)=\delta\Delta_s(C)$ for $\delta>0$. Eventually we get
$$
\delta\left(
E_\delta r\Delta_{r}( C)
 -\frac{D(r\Delta_{r}( C)+|\Delta_1( C)|)}{1+q/p}
\right)
\Delta_{1+q/p}(B)
\geq 
4D \Lambda^2.
$$
At the same time we have $\delta<1$, therefore it is sufficient to have the following inequality:
$$
\left(
\delta E_\delta r\Delta_{r}( C)
 -\frac{D(r\Delta_{r}( C)+|\Delta_1( C)|)}{1+q/p}
\right)
\Delta_{1+q/p}(B)
\geq
4D \Lambda^2.
$$
But this is exactly \eqref{eq: ramen} for $E_\delta=\delta^{-1}E$.

In fact we also used that for $\delta\in(0,1]$ we have $\Lambda(A,B,\delta C)\leq\Lambda(A,B,C)$.

\smallskip
Regarding the case $C=I$, the crucial observation is that in the inequalities determining $E$, we have $r,\Delta_r(C)$ appearing only as the product $r\Delta_r(C)$; see \eqref{eq: ramen} for the situation in the subdomain $\Omega_2$, for example. But for $r\geq2$ we have $r\Delta_r(I)=2$, \cite[(5.18)]{CD-DivForm},
which of course stays bounded as $r\rightarrow\infty$.
Finally use that $\Delta_s(A)$ depends on $s$ in a continuous manner.
\end{proof}

\section{Proof of Theorem \ref{t: trilinemb} for $\Omega=\R^d$}
\label{s: Extracorporeal Shock Wave Therapy}

As noted on page \pageref{koronavirus - skoro na rivu}, in case of $\Omega=\R^d$ we have $\oU=\oV=\oW=H_0^1(\R^{d})=H^1(\R^{d})$,  so that there we may drop the symbol for space from the notation for operators and semigroups. That is, we will write $L_A$ instead of $L_{A,\oU}$ and $T_t^A$ instead of $T_t^{A,\oU}$, and the same for other operators and semigroups.

We closely follow \cite[Section 6]{CD-DivForm}. We will for the moment {\it assume that the entries of $A,B,C$ are bounded $C^1$ functions with bounded derivatives}.
Once the proof for smooth $A,B,C$ is over, we will apply the regularization argument from \cite[Appendix]{CD-DivForm} to pass to the case of arbitrary (nonsmooth) $A,B,C$.
Also, we will initially work with $f,g,h\in C_c^\infty(\R^d)$ and then pass to $f,g,h$ as in the formulation of Theorem \ref{t: trilinemb}.

\subsection*{The heat flow}

Let $\gX$ be as in Theorem~\ref{t: Voltaren}. Here we use the $*$-ellipticity assumptions on $A,B,C$. Take $f,g,h\in C_c^\infty(\R^d)$. Suppose that $\psi\in C^\infty_c(\R^d)$ is radial, $\psi\equiv 1$ in the unit ball, $\psi\equiv 0$ outside the ball of radius $2$, and $0< \psi< 1$ elsewhere. For $R>0$ define the dilates $\psi_R(x) := \psi(x/R)$.

Let $\f_\nu$ be as in Section \ref{s: molto cantabile}.
Abbreviate $ \gX*\f_\nu= \gX_\nu$ and $\gamma_t=(T_{t}^{A}f,T_{t}^{B}g,T_{t}^{C}h)$.
With these choices made and fixed, define for $t\geq0$ the quantity
$\cE_{R,\nu}$ by
\begin{equation*}
\cE_{R,\nu}(t)=\int_{\R^{d}}\psi_R\cdot \gX_\nu(\gamma_t)\,.
\end{equation*}
This flow is regular \cite[Section 4.1]{CD-DivForm}. Fix $T>0$. We want to estimate the integral
\begin{equation}
\label{eq: Messiah}
-\int_0^T\cE_{R,\nu}'(t)\wrt t
\end{equation}
from above and below.

\subsection*{Upper estimate of the integral (\ref{eq: Messiah})}

Corollary \ref{c: Naklofen} {\it (\ref{eq: X1})} leads to
$$
-\int_0^T\cE_{R,\nu}'(t)\wrt t
\leqslant \cE_{R,\nu}(0)
=\int_{\R^d}\psi_R\cdot \gX_\nu(f,g,h)
\leqsim\int_{\R^d}\psi_R\left[(|f|+\nu)^{p}+(|g|+\nu)^{q}+ (|h|+\nu)^{r}\right].
$$
By the Lebesgue dominated convergence theorem we may send first $\nu\rightarrow0$ and then $R\rightarrow\infty$ and obtain
\begin{equation}
\label{eq: Velbon}
\limsup_{R\rightarrow\infty}\limsup_{\nu\rightarrow0}\left(-\int_{0}^T\cE_{R,\nu}'(t)\wrt t\right)
\leqsim
\nor{f}_{p}^{p}+\nor{g}_{q}^{q}+\nor{h}_{r}^{r}\,.
\end{equation}

\subsection*{Lower estimate of the integral (\ref{eq: Messiah})}
Recall the differential operators
$$
\partial_z=\frac{\partial_x-i\partial_y}2
\hskip30pt
\text{ and }
\hskip30pt
\partial_{\bar z}=\frac{\partial_x+i\partial_y}2.
$$

For $R>0$ define $\omega_R=\mn{x\in\R^{d}}{R\leqslant|x|\leq 2R}$,
so that ${\rm supp}\ \nabla\psi_R\subseteq \omega_R$.
By applying  \cite[Proposition 4.3]{CD-DivForm} and Corollary \ref{c: Naklofen} {\it (\ref{eq: X3})} we obtain
$$
\aligned
-&\int_0^T\cE_{R,\nu}'(t)\wrt t\ \geqsim
\int_0^T\int_{\R^d}\psi_R(|T_{t}^{C}h|-\nu)|\nabla T_{t}^{A}f||\nabla T_{t}^{B}g|\\
&+2\Re
\int_0^T\int_{\omega_R}
\Big( \left[(\partial_{\bar u} \gX_\nu)\circ\gamma\right]\cdot\sk{\nabla\psi_R}{A\nabla T_{t}^{A} f}_{\C^{d}}
\\
&\hskip 70pt
+\left[(\partial_{\bar v} \gX_\nu)\circ\gamma\right]\cdot\sk{\nabla\psi_R}{B\nabla T_{t}^{B} g}_{\C^{d}}
+\left[(\partial_{\bar w} \gX_\nu)\circ\gamma\right]\cdot\sk{\nabla\psi_R}{C\nabla T_{t}^{C} h}_{\C^{d}}
\Big).
\endaligned
$$
(Note that the assumption on the entries of $A,B,C$ being $C^1$ with bounded derivatives was used in order to justify applying \cite[Proposition 4.3]{CD-DivForm}.)
We would like to study this inequality as $\nu\rightarrow0$ and $R\rightarrow\infty$.
Since $ \gX$ is of class $C^1$, we have
$\partial_{\bar u} \gX_\nu\rightarrow \partial_{\bar u} \gX$,
$\partial_{\bar v} \gX_\nu\rightarrow \partial_{\bar v} \gX$ and
$\partial_{\bar w} \gX_\nu\rightarrow \partial_{\bar w} \gX$
pointwise on $\C^3$, as $\nu\rightarrow 0$.

Recall the estimates of Corollary \ref{c: Naklofen} {\it (\ref{eq: X2})}.
Thus, by using \cite[Lemma 6.1]{CD-DivForm} and the dominated convergence theorem,
\begin{eqnarray*}
&\hskip -120pt
{\displaystyle
  \liminf_{\nu\rightarrow 0}\left(-\int_0^T\cE_{R,\nu}'(t)\wrt t\right)
     \geqsim
  \int_0^T\int_{\R^d}\psi_R|\nabla T_{t}^{A}f||\nabla T_{t}^{B}g||T_{t}^{C}h|}\\
&\hskip -80pt{\displaystyle
+2\,\Re
\int_0^T\int_{\omega_R}
\Big( \left[(\partial_{\bar u} \gX)\circ\gamma\right]\cdot\sk{\nabla\psi_R}{A\nabla T_{t}^{A} f}_{\C^{d}}
}\\
&\hskip 60pt
{\displaystyle
+\left[(\partial_{\bar v} \gX)\circ\gamma\right]\cdot\sk{\nabla\psi_R}{B\nabla T_{t}^{B} g}_{\C^{d}}
+\left[(\partial_{\bar w} \gX)\circ\gamma\right]\cdot\sk{\nabla\psi_R}{C\nabla T_{t}^{C} h}_{\C^{d}}
\Big).}
\nonumber
\end{eqnarray*}
Hence, by Theorem \ref{t: Voltaren} {\it(\ref{eq: Fest})}, the second part of \cite[Lemma 6.1]{CD-DivForm} and Fatou's lemma,
\begin{equation}
\label{eq: Sherpa 4370D}
\aligned
\liminf_{R\rightarrow \infty}\liminf_{\nu\rightarrow 0}
\left(-\int_0^T\cE_{R,\nu}'(t)\wrt t\right)
& \geqsim
\int_0^T\int_{\R^{d}}|\nabla T_{t}^{A}f||\nabla T_{t}^{B}g||T_{t}^{C}h|.
\endaligned
\end{equation}

\subsection*{Summary (for smooth $A,B,C$)}
The combination of \eqref{eq: Velbon} and \eqref{eq: Sherpa 4370D} immediately gives
$$
\int_0^T\int_{\R^{d}}|\nabla T_{t}^{A}f||\nabla T_{t}^{B}g||T_{t}^{C}h|\ \leqsim\
\nor{f}_{p}^{p}+\nor{g}_{q}^{q}+\nor{h}_{r}^{r}\,.
$$
Now replace $f,g,h$ by $a_1 f,a_2 g, a_3 h$ where $a_j$ are positive numbers such that $a_1a_2a_3=1$.
\label{polar1}
After minimizing over all triples $(a_1,a_2, a_3)$ as above and applying Lemma \ref{l: Mos Def} below, we
get
$$
\int_0^T\int_{\R^{d}}|\nabla T_{t}^{A}f||\nabla T_{t}^{B}g||T_{t}^{C}h|\ \leqsim\
\nor{f}_{p}\nor{g}_{q}\nor{h}_{r}\,.
$$
At this point send $T\rightarrow\infty$ and use the monotone convergence theorem. This gives the trilinear embedding for {\it smooth} $A,B,C$.

\subsection*{Proof for nonsmooth $A,B,C$}
In order to extend the trilinear embedding to arbitrary (nonsmooth) $A,B$ and $C$, we adapt the argument from \cite[Section 6]{CD-DivForm} as follows. We shall assume the notation from \cite[Appendix]{CD-DivForm}.

Suppose that $A,B,C$ satisfy conditions ({\Large $\star$}) of Theorem \ref{t: trilinemb} and let $A_\e,B_\e,C_\e$ be their smooth approximations as in \cite[Section A.1]{CD-DivForm}. It follows from \cite[Lemma A.5 {\it (iv)}]{CD-DivForm} that, for sufficiently small $\e>0$, matrices $A_\e,B_\e,C_\e$ also satisfy  ({\Large $\star$}). We will work with such $\e$.

Denote
$$
\wt{f}^{A}(x,t):=\big(P_t^Af\big)(x).
$$

Choose $0<a<b<\infty$.
Applying the Minkowski's integral inequality to the second identity from the proof of \cite[Lemma A.4]{CD-DivForm} gives
$$
\aligned
\nor{\nabla_x \wt{f}^{A}-\nabla_x \wt{f}^{A_\e}}_{L^2(\R^d\times(a,b))}
&=\frac1{2\pi}\left(
\int_{\R^d\times(a,b)}\left|\int_\gamma e^{-t\zeta}\nabla U(\e,\zeta)f(x)\wrt \zeta\right|^2\wrt x\wrt t\right)^{1/2}\\
&\leq\frac1{2\pi}
\int_\gamma
\left(
\int_{\R^d\times(a,b)}
\left|e^{-t\zeta}\nabla U(\e,\zeta)f(x)\right|^2
\wrt x\wrt t
\right)^{1/2}
\wrt |\zeta|\\
&=\frac1{2\pi}
\int_\gamma
\left(
\int_a^b
e^{-2t\,\Re\zeta}\wrt t
\right)^{1/2}
\left(
\int_{\R^d}
\left|\nabla U(\e,\zeta)f(x)\right|^2
\wrt x
\right)^{1/2}
\wrt |\zeta|.
\endaligned
$$
Write
$$
F(\delta)=F_{a,b}(\delta):=\int_a^b
e^{-2t\delta}\wrt t.
$$
Of course, $F$ is a continuous function on $\R$ that can be expressed explicitly.
 On the other hand, by \cite[Lemma A.1]{CD-DivForm} we have
$$
\left(
\int_{\R^d}
\left|\nabla U(\e,\zeta)f(x)\right|^2
\wrt x
\right)^{1/2}
\leqsim
\nor{M_{A-A_\e}T_0(\zeta)f}_{L^2(\R^d)},
$$
with implicit constants depending on $A$.
Therefore
\begin{equation}
\label{eq: karantena dan 1}
\nor{\nabla_x \wt{f}^{A}-\nabla_x \wt{f}^{A_\e}}_{L^2(\R^d\times(a,b))}
\leqsim
\int_\gamma
\sqrt{F(\Re\zeta)}
\nor{M_{A-A_\e}T_0(\zeta)f}_{L^2(\R^d)}
\wrt |\zeta|.
\end{equation}
Since $\Lambda(A),\Lambda(A_\e)\leq\Lambda$, we have
$
\nor{M_{A-A_\e}T_0(\zeta)f}_{L^2(\R^d)}
\leqsim
\nor{T_0(\zeta)f}_{L^2(\R^d)}
\leqsim
|\zeta|^{-1/2}\nor{f}_{L^2(\R^d)}
$
as in \cite[proof of Lemma A.4]{CD-DivForm}.
Hence the integrands in \eqref{eq: karantena dan 1} admit, uniformly in $\e>0$, the majorant
$\sqrt{F(\Re\zeta)/|\zeta|}$,
which belongs to $L^1(\gamma,\textup{d}|\zeta|)$. Thus we may use Lebesgue's dominated convergence when passing $\e\rightarrow0$. Since for $\zeta\in\gamma$ we have $\nor{M_{A-A_\e}T_0(\zeta)f}_{L^2(\R^d)}\rightarrow0$ as $\e\rightarrow0$ \cite[proof of Lemma A.4]{CD-DivForm}, we conclude that
$$
\lim_{\e\rightarrow0}
\nor{\nabla_x \wt{f}^{A}-\nabla_x \wt{f}^{A_\e}}_{L^2(\R^d\times(a,b))}
=0.
$$
In  a similar fashion we show that
$$
\lim_{\e\rightarrow0}
\nor{\wt{f}^{A}-\wt{f}^{A_\e}}_{L^2(\R^d\times(a,b))}
=0.
$$
That is,
$$
\wt{f}^{A}=\lim_{\e\rightarrow0}
\wt{f}^{A_\e}
\hskip 30pt
\text{in }H^1\big(\R^d\times(a,b)\big).
$$
Of course, the same is valid for $B,B_\e$ and $C,C_\e$.

Now a standard theorem in measure theory implies that there exists a sequence $(\e_l)_{l\in\N}$ such that
$\wt{f}^{A_{\e_l}}\rightarrow \wt{f}^{A}$ and $\nabla\wt{f}^{A_{\e_l}}\rightarrow \nabla\wt{f}^{A}$ as $l\rightarrow\infty$ almost everywhere on $\R^d\times(a,b)$, and the same for $(B,g)$ and $(C,h)$. Consequently, Fatou's lemma gives
$$
\aligned
\int_a^b\int_{\R^{d}}|\nabla T_{t}^{A}f||\nabla T_{t}^{B}g||T_{t}^{C}h|\wrt x\wrt t
&\leq\liminf_{l\rightarrow\infty}
\int_a^b\int_{\R^{d}}|\nabla T_{t}^{A_{\e_l}}f||\nabla T_{t}^{B_{\e_l}}g||T_{t}^{C_{\e_l}}h|\wrt x\wrt t\\
&\leq\liminf_{l\rightarrow\infty}
\int_0^\infty\int_{\R^{d}}|\nabla T_{t}^{A_{\e_l}}f||\nabla T_{t}^{B_{\e_l}}g||T_{t}^{C_{\e_l}}h|\wrt x\wrt t.
\endaligned
$$
Recall that we have already established the trilinear embedding for triples of smooth matrices and that $\Delta_s(A_\e)\rightarrow\Delta_s(A)$ as $\e\rightarrow0$ \cite[Lemma A.5 {\it (iv)}]{CD-DivForm}.
Passing $a\rightarrow0$ and $b\rightarrow\infty$ finishes the proof of Theorem \ref{t: trilinemb} for functions in $f,g,h\in C_c^\infty$.
It remains to prove it for a larger class of functions, as specified in the formulation of the theorem.

\subsection{Proof for arbitrary $f,g,h$}
\label{s: Jugoplastika : Denver 1989}

We assume that the trilinear embedding \eqref{eq: trilinemb} holds for triples of functions from $C_c^\infty(\R^d)$.
Take $f\in\left(L^p\cap L^2\right)(\R^d)$.
Let $\big(f_n\big)_{n\in\N}$ be such a sequence in $C_c^\infty(\R^d)$ that $f_n\rightarrow f$ in both $L^p$ and $L^2$. (Scrutinous reading of the proof of convergence in a single $L^r$ reveals that simultaneous convergence in a pair of Lebesgue spaces may be achieved.) By the {\it bilinear embedding theorem} \cite[Theorem 1.1]{CD-DivForm} on $L^2\times L^2$ we conclude that $\big(\nabla\wt{f}_n^A\big)_{n\in\N}$ is a Cauchy sequence in $L^2\big(\R^d\times(0,\infty)\big)$.
Therefore it is Cauchy in $\gH=\gH_M:=L^2\big(\R^d\times(0,M)\big)$ for every fixed $M>0$. Consequently, there exists $\Psi\in\gH$ such that $\nabla\wt{f}_n^A\rightarrow\Psi$ in $\gH$ as $n\rightarrow\infty$.
At the same time we know that the semigroup $(T_t^Af)_{t>0}$ is contractive on $L^2(\R^d)$, which implies that $\wt{f}_n^A\rightarrow\wt{f}^A$ in $\gH$ as $n\rightarrow\infty$. Here we use that we work on $\R^d\times(0,M)$ instead of working on $\R^d\times(0,\infty)$ straight away. Hence
$\big(\wt{f}_n^A\big)_{n\in\N}$ is a Cauchy sequence in $H^1\big(\R^d\times(0,M)\big)$. By using completeness of $H^1$, the fact that convergence in $H^1$ implies convergence in $\gH$ and, once again, that $\wt{f}_n^A\rightarrow\wt{f}^A$ in $\gH$, we conclude that $\wt{f}_n^A\rightarrow\wt{f}^A$ in $H^1\big(\R^d\times(0,M)\big)$. In particular, $\nabla\wt{f}_n^A\rightarrow\nabla\wt{f}^A$ in $\gH$. Now a standard theorem provides the existence of an increasing sequence of positive integers $(n_k)_{k\in\N}$ so that
\renewcommand{\arraystretch}{1.5}
\begin{equation}
\label{eq: SR71}
\text{a.e. }(x,t)\in\R^d\times(0,M):
\hskip 20pt
\left\{
\begin{array}{rcl}
\wt{f}_{n_k}^A&\rightarrow&\wt{f}^A\\
\nabla\wt{f}_{n_k}^A&\rightarrow&\nabla\wt{f}^A
\end{array}
\right.
\hskip 20pt
\text{as }k\rightarrow\infty.
\end{equation}
\renewcommand{\arraystretch}{1.0}
Of course, we can analogously find $g_n,h_n\in C_c^\infty(\R^d)$ and arrange that \eqref{eq: SR71} also holds (for the same subsequence) for $(g,B)$ and $(h,C)$. Now use that
$$
\int_0^M\int_{\Omega}
\left| \nabla T_t^{A} f _n\right|
\left| \nabla T_t^{B} g_n \right|
\left| T_t^{C} h_n \right|
\wrt x\wrt t
\ \leqsim\
\nor{f_n}_{p}\nor{g_n}_{q}\nor{h_n}_{r},
$$
apply \eqref{eq: SR71}, Fatou's lemma and the fact that $f_n\rightarrow f$ in $L^2(\R^d)$ and the same for $g,h$. Since the implied constants do not depend on $M$, we may finally send $M\rightarrow\infty$ and complete the proof of Theorem \ref{t: trilinemb} for
$f\in L^p\cap L^2 $,
$g\in L^q\cap L^2 $
and
$h\in L^r\cap L^2 $.
\hfill
\qed

%\begin{remark}
%Observe that the above reasoning slightly strengthens the {\it bilinear embedding in $\R^d$}, which was one of the main result of \cite{CD-DivForm}. Namely, it implies that \cite[Theorem 1.1]{CD-DivForm} holds not only for $f,g\in C_c^\infty(\R^d)$ as stated there, but actually for $f\in  L^p\cap L^2  $ and
%$g\in  L^{p'}\cap L^2$.
%\end{remark}

\begin{lemma}
\label{l: Mos Def}
Suppose that
the numbers $p_1,\hdots,p_n>1$ are related by
$$
\frac{1}{p_1}+\hdots+\frac{1}{p_n}=1.
$$
Then for every sequence ${\mathsf f}_1,\hdots,{\mathsf f}_n>0$ we have
$$
\min\mn{\sum_{j=1}^n(a_j{\mathsf f}_j)^{p_j}}{a_1,\ldots,a_n>0, \prod_{j=1}^na_j=1}=\prod_{j=1}^np_j^{1/p_j}{\mathsf f}_j.
$$
\end{lemma}
\begin{proof}
The lemma follows from Young's inequality and, on the other hand, the choice of
\begin{equation*}
a_j=\frac{\left(\prod_{k=1}^{n}p_k^{1/p_k}{\mathsf f}_k\right)^{1/p_j}}{p_j^{1/p_j}{\mathsf f}_j}
\end{equation*}
which confirms that the minimum is actually attained.
\end{proof}

Let us finally prove that the assumptions ({\Large $\star$}) of Theorem \ref{t: trilinemb}, pertaining to the case $\Omega=\R^d$, are indeed milder than $A,B,C\in \cA_{\max\{p,q,r\}}(\R^d)$, as indicated in the formulation of Theorem
\ref{t: trilinemb}.
As a matter of fact, this is a quick consequence of Lemma \ref{l: Ce rau am facut la lume}.

\begin{corollary}
\label{c: 5AM}
Let $p,q,r\in(1,\infty)$ be such that $1/p+1/q+1/r=1$. If $A$ is $\max\{p,q,r\}$-elliptic, then $A$ is $t$-elliptic for any $t\in\{p,q,r,1+p/q,1+q/p\}$.
\end{corollary}

\begin{proof}
This follows from Lemma \ref{l: Ce rau am facut la lume} and the fact that $\Delta_t(A)$ is invariant under conjugation of $t$ and, as a function of $t$, decreases on $[2,\infty)$.
\end{proof}

The proof of Theorem \ref{t: trilinemb} for $\Omega=\R^d$ is now complete.

\section{Proof of Theorem \ref{t: trilinemb} for general $\Omega$}
\label{s: general case}

For the sake of simplifying the notation, we will only prove the theorem in the case when $\oU=\oV=\oW$. The proof of the general case is exactly the same.
Consequently, we will again omit writing $\oU,\oV,\oW$ in the notation, just like we did in Section \ref{s: Extracorporeal Shock Wave Therapy}.
We will occasionally abbreviate $\bA=(A,B,C)$.

We first remark that in order to pass from $\R^d$ to arbitrary open sets there, we cannot merely imitate Section \ref{s: Extracorporeal Shock Wave Therapy}. The reason is explained in \cite[Section 1.4]{CD-Mixed}.
Instead, here we closely follow \cite{CD-Mixed}, adapting the argument there to the current trilinear context.
Hence in this section we will always assume that $\Omega\subseteq \R^{d}$ is open and $\oU$ is one of the subspaces of Section~\ref{s: Heian}.

We first recall a basic result on $L^p$ extensions of semigroups $(T^{A}_{t})_{t>0}$.

\begin{proposition}
\cite[Lemma 17]{CD-Mixed}
\label{p: Zbunjen i osamucen}
Let $s\geq 2$ and $A\in\cA_s(\Omega)$. There exists $\vartheta=\vartheta(s)\in (0,\pi/2)$ such that $(T^{A}_{t})_{t>0}$ for any $\wp\in [s',s]$ extends to a semigroup which is analytic and contractive in $L^{\wp}(\Omega)$ in the cone $\mn{z\in\C\setminus\{0\}}{|\arg z|<\vartheta}$.
We denote the negative generator of $(T^{A}_{t})_{t>0}$ on $L^{\wp}$ by $L_A^{(\wp)}$.
\end{proposition}

Let us start with a reduction in the spirit of \cite[Section 6.1]{CD-Mixed}.

\begin{proposition}
Suppose that $A,B,C,p,q,r$ are as in the formulation of Theorem \ref{t: trilinemb} and $\gX=\gX(u,v,w)$ as in Theorem \ref{t: Voltaren}.
Assume that
\begin{equation}
\label{eq: reduction}
\int_{\Omega}\mod{\nabla\tf}\mod{\nabla\tg}\mod{\th}
\leqsim
\,\Re\int_{\Omega}
\Big(
  \partial_{u}\gX\left({\mathtt f},{\mathtt g},{\mathtt h}\right) L_{A}{\mathtt f}
+\partial_{v}\gX\left({\mathtt f},{\mathtt g},{\mathtt h}\right) L_{B}{\mathtt g}
+\partial_{w}\gX\left({\mathtt f},{\mathtt g},{\mathtt h}\right) L_{C}{\mathtt h}
\Big)
\end{equation}
for 
% ${\mathtt f},{\mathtt g},{\mathtt h}\in \oU$ 
${\mathtt f}\in\cD(L_A)$, 
${\mathtt g}\in\cD(L_B)$, 
${\mathtt h}\in\cD(L_C)$ 
such that
$\tf,\tg,\th,L_{A}\tf, L_{B}\tg,L_C\th\in\big(L^p\cap L^{p'}\cap L^{r}\big)(\Omega)$.
Then the statement of Theorem \ref{t: trilinemb} holds.
\end{proposition}

\begin{proof}
By symmetry it is enough to prove Theorem \ref{t: trilinemb} for $p\geq q$.
Let $f,g,h\in \big(L^p\cap L^{p'}\cap L^{r}\big)(\Omega)$.
This intersection is contained in $L^q$ (cf. the end of the proof of Lemma \ref{l: Ce rau am facut la lume}).

Define $\gamma\colon[0,\infty)\rightarrow\C^3$ by
$$
\gamma_t=\gamma(t):=\left(T^{A}_{t}f,T^{B}_{t}g,T_t^Ch\right)
$$
and
$\cE\colon[0,\infty)\rightarrow[0,\infty)$ by
\begin{equation*}
\cE(t)=\int_{\Omega}\gX(\gamma_t).
\end{equation*}
Recall that $\Delta_{p'}(\bA)=\Delta_{p}(\bA)$.
Theorem \ref{t: Voltaren} {\it (\ref{eq: Fest})} and Proposition \ref{p: Zbunjen i osamucen}
imply that $\cE$ is well defined, continuous on $[0,\infty)$, differentiable on $(0,\infty)$ with a continuous derivative and
\begin{equation}
\label{eq: Perasovic}
-\cE'(t)
=2\,\Re
   \int_{\Omega}
     \Big(
          \partial_{u}\gX(\gamma_t) L_{A}T^{A}_{t}f
        +\partial_{v}\gX(\gamma_t) L_{B}T^{B}_{t}g
        +\partial_{w}\gX(\gamma_t) L_{C}T^{C}_{t}h
\Big);
\end{equation}
see \cite[Proposition 4.3]{CD-DivForm} or \cite[Section 6.1]{CD-Mixed}.
Since $\cE$ is nonnegative and $\gamma(0)=(f,g,h)$, we have, by Theorem \ref{t: Voltaren} {\it (\ref{eq: Burger})},
\begin{equation}
\label{eq: Sretenovic}
-\int_0^\infty\cE'(t)\wrt t\leq\cE(0)=\int_{\Omega}\gX(f,g,h)\leqsim\nor{f}_p^p+\nor{g}_q^q+\nor{h}_r^r.
\end{equation}
Analyticity of the semigroups yields $T^{A}_tf\in\cD\big(L_{A}^{(p)}\big)\cap\cD\big(L_{A}^{(p')}\big)\cap\cD\big(L_{A}^{(r)}\big)$, and the same for $(g,B)$ and $(h,C)$.
By consistency of the semigroups and H\"older's inequality, we have
$$
\cD\big(L_{A}^{(p)}\big)\cap\cD\big(L_{A}^{(p')}\big)\cap\cD\big(L_{A}^{(r)}\big)\subseteq  \cD\big(L_{A}^{(2)}\big)\subseteq  \oU
$$
and the same for $B,C$.
Therefore we may apply \eqref{eq: reduction} with $(\tf,\tg,\th)=\left(T^{A}_{t}f,T^{B}_{t}g,T_t^Ch\right)=\gamma_t$.
Together with \eqref{eq: Perasovic} and \eqref{eq: Sretenovic} we then obtain
$$
\int_0^\infty\int_{\Omega}
\left| \nabla T_t^{A} f \right|
\left| \nabla T_t^{B} g \right|
\left| T_t^{C} h \right|
\wrt x\wrt t
\ \leqsim\
\nor{f}_p^p+\nor{g}_q^q+\nor{h}_r^r.
$$
Using again Lemma \ref{l: Mos Def} completes
\label{polar2}
the proof of Theorem \ref{t: trilinemb} for $f,g,h\in \big(L^p\cap L^{p'}\cap L^{r}\big)(\Omega)$.

In order to pass to
$f\in L^p\cap L^2$,
$g\in L^q\cap L^2$
and
$h\in L^r\cap L^2$,
we argue just as in Section \ref{s: Jugoplastika : Denver 1989}, that is, by considering triples from $C_c^\infty(\Omega)$, only replacing \cite[Theorem 1.1]{CD-DivForm} by \cite[Theorem 2]{CD-Mixed}.
\end{proof}

We will thus focus on proving \eqref{eq: reduction}.
Given functions ${\mathtt f},{\mathtt g},{\mathtt h}$, denote
\begin{equation}
\label{eq: nzapazap}
\text{\calligra L }(\gX)({\mathtt f},{\mathtt g},{\mathtt h})
=
\Re\int_{\Omega}
\Big(
  (\partial_{u}\gX)\left({\mathtt f},{\mathtt g},{\mathtt h}\right) L_{A}{\mathtt f}
+(\partial_{v}\gX)\left({\mathtt f},{\mathtt g},{\mathtt h}\right) L_{B}{\mathtt g}
+(\partial_{w}\gX)\left({\mathtt f},{\mathtt g},{\mathtt h}\right) L_{C}{\mathtt h}
\Big).
\end{equation}
Proving \eqref{eq: reduction} thus means proving
\begin{equation}
\label{eq: Cherokee}
\int_{\Omega}\mod{\nabla\tf}\mod{\nabla\tg}\mod{\th}
\leqsim
\,
\text{\calligra L }(\gX)({\mathtt f},{\mathtt g},{\mathtt h})
\end{equation}
for 
${\mathtt f}\in\cD(L_A)$, 
${\mathtt g}\in\cD(L_B)$, 
${\mathtt h}\in\cD(L_C)$ 
such that
$\tf,\tg,\th,L_{A}\tf, L_{B}\tg,L_C\th\in\big(L^p\cap L^{p'}\cap L^{r}\big)(\Omega)$.
%for all
%${\mathtt f},{\mathtt g},{\mathtt h}\in \oU$
%such that
%$\tf,\tg,\th,L_{A}\tf, L_{B}\tg,L_C\th\in\big(L^p\cap L^{p'}\cap L^{r}\big)(\Omega)$.

\setcounter{subsection}{-1}

\subsection{Scheme of the proof of (\ref{eq: reduction})}

Our plan is to start with the right-hand side of \eqref{eq: reduction}, integrate by parts in the sense of \eqref{eq: int by parts}, arrive at the (generalized) Hessian form of $\gX$, and finally use convexity properties of $\gX$ (Theorem \ref{t: Voltaren}). When doing so we will encounter several technical obstacles which we will tackle one by one, each time by approximating the integrand with an adequately constructed sequence of functions.
As said before, our inspiration for this part of the paper is \cite{CD-Mixed}.

\subsection{First approximation}

In order to apply integration by parts \eqref{eq: int by parts} to the right-hand side of \eqref{eq: reduction}, we must have $(\partial_u\gX)(\tf,\tg,\th)\in\oU$ and the same for $\partial_v\gX,\partial_w\gX$. In particular, these functions should belong to $H^1(\Omega)$. Moreover, we would like to apply the chain rule to $\nabla[(\partial_u\gX)(\tf,\tg,\th)]$ and obtain the Hessian matrix of $\gX$. Based on standard results regarding the chain rule for Sobolev functions, we would for that purpose need $\nabla_{u,v,w}\gX$ to be of class $C^1$. Hence we first replace $\gX$ by its smooth version.

Let $\f_{\nu}$ be as in Section \ref{s: molto cantabile}.
For the reasons just explained, one could try to replace $\text{\calligra L }(\gX)$ by $\text{\calligra L }(\gX*\f_{\nu})$ and then terminate the approximation by sending $\nu\rightarrow0$.

\subsection{Second approximation}

We want to integrate by parts as in \eqref{eq: int by parts}, but we still cannot prove that $\partial_{u,v,w}(\gX*\f_{\nu})(\tf,\tg,\th)\in\oU$. The problem is that $\partial_{u,v,w}(\gX*\f_{\nu})(\tf,\tg,\th)$ is not a Lipschitz function; see \cite[proof of Lemma 19]{CD-Mixed} or \cite[Lemma 4]{Egert2018} cited there.
In order to fix this, we multiply $\gX*\f_{\nu}$ by compactly supported smooth functions which are identically equal to $1$ on gradually larger sets.

Thus, choose a radial function $\psi\in C^{\infty}_{c}(\C^{3})$ such that $\psi\geq 0$, $\psi=1$ on $\{|\omega|\leq 3\}$ and $\psi=0$ on $\{|\omega|>4\}$. For $n\in\N $ define $\psi_{n}(\omega)=\psi(\omega/n)$. We try to consider the flow
$$
\text{\calligra L }\left(\psi_n\cdot(\gX*\f_{\nu})\right)({\mathtt f},{\mathtt g},{\mathtt h})
$$
and then send $n\rightarrow\infty$ and $\nu\rightarrow0$.

\subsection{Third approximation}
\label{s: banana}
As explained earlier and demonstrated in the proof of the case $\Omega=\R^d$ (Section \ref{s: Extracorporeal Shock Wave Therapy}), the gist of our method is the (global) {\it quantitative generalized convexity} -- that is, lower estimates of generalized Hessians -- of $\gX$ and its approximations. For $\gX$, this convexity is made explicit in Theorem \ref{t: Voltaren} {\it(\ref{eq: sova})}. In case of $\gX*\f_\nu$ we have a very similar estimate in Corollary \ref{c: Naklofen} {\it (\ref{eq: X3})}.
However, if we pass to $\psi_n\cdot(\gX*\f_\nu)$, then we quickly see that in the domain where $\psi_n$ is not (locally) constant, which is $\{3n<|\omega|<4n\}$, this estimate might break.
This problem was already encountered in \cite{CD-Mixed} where we considered a two-variable function $Q$ instead of $\gX$. There we solved it by adding to $\psi_n\cdot(Q*\f_\nu)$ a perturbation comprising {\it power function in several variables} and a correction factor ($\nu^{q-2}$). The size of that correction factor was determined on the basis of upper $L^\infty$ estimates of the Hessian of $\psi_n\cdot(Q*\f_\nu)$ for $|\omega|\sim n$, see \cite[(34)]{CD-Mixed}.

However, as said before, in \cite{CD-Mixed} we dealt with a different Bellman function. That is, working with $\gX$ calls for a new set of estimates.
Assessing the magnitude of $\nor{D^2\left[\psi_n\cdot(\gX*\f_\nu)\right]}_\infty$ for $|\omega|\sim n$
requires some effort. We do that next and summarize the outcome in Corollary \ref{c: schlagwerk}.

Let $s_+:=\max\{s,0\}$ denote the positive part of the real number $s$.

\begin{lemma}
\label{l: Drska:Berwolf 44:43}
Let $\nu\in(0,1]$ and $\alpha\in\R$.
Denote $\Phi_\alpha:=F_\alpha\otimes{\mathbf 1}\otimes{\mathbf 1}$, that is, $\Phi_\alpha:\C^3\rightarrow[0,\infty)$ acts as $\Phi_\alpha(u,v,w)=|u|^\alpha$.
If $\alpha>-2$ (so that $\Phi_\alpha\in L_{loc}^1$) then
for $\omega=(u,v,w)\in\C^3$ we have
$$
\left|
\left(\Phi_\alpha*\f_\nu\right)(\omega)
\right|
\,\leqsim\,
\nu^{\alpha}+|u|^{\alpha_+}.
$$
\noindent
If $-2<\beta<0$ and $\Psi_\beta:=F_1\otimes F_\beta\otimes{\mathbf 1}$,
then
$$
\left|
\left(\Psi_\beta
*\f_\nu\right)(\omega)
\right|
\,\leqsim\,
\nu^{\beta}\left(|u|+1\right).
$$
\end{lemma}

\begin{proof}
The proof is an adaptation of the proof of \cite[Lemma 14]{CD-Mixed}.
Let $N=3$. Then we have
$$
\aligned
\int_{\C^{N}}
| u^{\prime}|^{\alpha}
&\varphi_{\nu}( u- u^{\prime}, y- y^{\prime})\wrt  u^{\prime}\wrt  y^{\prime}
\\
&=\nu^{-2}\int_{\C}| u^{\prime}|^{\alpha}\left[\nu^{-2(N-1)}\int_{\C^{N-1}}\varphi\left(\frac{ u- u^{\prime}}{\nu},\frac{ y- y^{\prime}}{\nu}\right)\wrt y^{\prime}\right]\wrt u^{\prime}
\\
&=\nu^{-2}\int_{\C}| u^{\prime}|^{\alpha}\int_{\C^{N-1}}\varphi\left(\frac{u-u^{\prime}}{\nu},z'\right)\wrt z^{\prime}\wrt u^{\prime}
\\
&=\nu^{\alpha}\int_{\{|u^{\prime\prime}-u/\nu|<1\}}
\mod{u''}^\alpha
\int_{\{|z^{\prime}|<1\}}\varphi\left(u/\nu-u^{\prime\prime},z'\right)\wrt z^{\prime}\wrt u^{\prime\prime}
\\
&\leq \nu^{\alpha}\nor{\varphi}_{\infty}|B_{\C^{N-1}}(0,1)|
\int_{\{|u''-u/\nu|<1\}}\mod{u''}^\alpha\wrt u''.
\endaligned
$$
Denote the last integral by $I_\alpha(u/\nu)$. We analyze it depending on the sign of $\alpha$.

If \underline{$-2<\alpha<0$}, then we
clearly have
$$
\lim_{|\bar u|\rightarrow\infty}I_\alpha(\bar u)=0.
$$
Therefore in this case $I_\alpha$ is a bounded function on $\C$, that is, $I_\alpha(u/\nu)\leqsim1$.

If \underline{$\alpha\geq0$}, then
$$
I_\alpha(\bar u)
\leqsim_\alpha\,\int_{B_\C(\bar u,1)}\left(|\bar u|^{\alpha}+1\right)\wrt u''
\,\leqsim\, |\bar u|^{\alpha}+1.
$$
Therefore $\nu^\alpha I_\alpha(u/\nu)\,\leqsim\,|u|^\alpha+\nu^\alpha$, as claimed.

\medskip
Regarding the second part of the proof,
$$
\aligned
\int_{\C^{3}}|u^{\prime}||v^{\prime}|^{\beta}
&\varphi_{\nu}( u- u^{\prime}, v- v^{\prime},w-w')\wrt  u^{\prime}\wrt v^{\prime}\wrt w'
\\
&=\nu^{-4}\int_{\C^2}
|u^{\prime}||v^{\prime}|^{\beta}
\left[
\nu^{-2(N-2)}\int_{\C^{N-2}}\varphi\left(\frac{u-u^{\prime}}{\nu},\frac{v-v^{\prime}}{\nu},\frac{w-w^{\prime}}{\nu}\right)\wrt w^{\prime}
\right]
\wrt u^{\prime}\wrt v'
\\
&=\nu^{-4}\int_{\C^2}
|u^{\prime}||v^{\prime}|^{\beta}
\int_{\C^{N-2}}\varphi\left(\frac{u-u^{\prime}}{\nu},\frac{v-v^{\prime}}{\nu},z^{\prime}\right)\wrt z^{\prime}
\wrt u^{\prime}\wrt v'
\\
&\leq
\left(
\nu\int_{\{|u''-u/\nu|<1\}}|u''|\wrt u''
\right)
\left(
\nu^{\beta}\int_{\{|v''-v/\nu|<1\}}|v''|^\beta\wrt v''
\right)
\nor{\f}_\infty|B^{N-2}_\C(0,1)|.
\endaligned
$$
The already completed part of the proof shows that the
we can continue as
\begin{alignat*}{2}
\leqsim\,
\nu I_1(u/\nu)\cdot\nu^{\beta}I_\beta(u/\nu)
\,\leqsim\,
(|u|+\nu)\nu^\beta.
\tag*{\qedhere}
\end{alignat*}
\end{proof}

\begin{remark}
Upon obvious changes of the lemma's modification, the above proof would clearly work for any $N\in\N\backslash\{1\}$, not only $N=3$.
\end{remark}

The next result is modelled after \cite[Lemma 14]{CD-Mixed} and the estimate \cite[(34)]{CD-Mixed} derived from it.
Its gradient estimates complement those of Corollary \ref{c: Naklofen}.

\medskip
\noindent
\framebox{\bf Notation.}
Until the end of this section we will assume that $p,q,r\in(1,\infty)$ are such that $1/p+1/q+1/r=1$ (i.e., that the exponents satisfy the H\"{o}lder scaling) and $p\geq q$. Given a triple of such $(p,q,r)$, we will denote
\begin{equation}
\label{eq: mango}
\aligned
m &:=\min\{q/p+1,r\}\\
M&:=\max\{p,r\}.
\endaligned
\end{equation}

\begin{corollary}
\label{c: izumrud}
Let $\gX$ be as in Theorem \ref{t: Voltaren}.
Take $\nu\in(0,1]$.
Then for every $\omega\in\C^3$ we have
\begin{enumerate}[(i)]
\item
\label{eq: D0}
$\mod{(\gX*\f_\nu)(\omega)}\leqsim\,\su^p+\sv^q+\sw^r+1$,
\item
\label{eq: D1}
$\mod{D(\gX*\f_\nu)(\omega)}\leqsim\,
 |\omega|^{m-1}+|\omega|^{M-1}$,
\item
\label{eq: D2}
$\mod{D^2(\gX*\f_\nu)(\omega)}\leqsim\,
\nu^{m-2}\left(|\omega|^{M-2}
+|\omega|+1\right). $
\end{enumerate}
The implied constants depend on $p,q,r,A,B,C$ and their $*$-ellipticity constants arising from the assumption {\rm({\Large $\star$})}, but not on $\omega,\nu$.
The constants in {\it(\ref{eq: D1})} and {\it(\ref{eq: D2})} may also depend on $d$.
\end{corollary}

\begin{proof}
Item {\it(\ref{eq: D0})} follows from Corollary \ref{c: Naklofen} {\it(\ref{eq: X1})}.

Let us prove {\it(\ref{eq: D2})}.
Recalling the convention from Section \ref{s: Ella Fitzgerald - Caravan}, we have
$$
\aligned
\left|D^2(\gX*\f_\nu)(\omega)\right|
& \,\leq\, \sum_{j,k=1}^d\left|D_{jk}^2(\gX*\f_\nu)(\omega)\right|\\
& \,\leq\, \sum_{j,k=1}^d\int_{\C^3}\left|D_{jk}^2\gX\right|(y)\f_\nu(\omega-y)\wrt y\\
& \,\leqsim_{d}\, \int_{\C^3}\left|D^2\gX\right|(y)\f_\nu(\omega-y)\wrt y\\
& = \left(\left|D^2\gX\right|*\f_\nu\right)(\omega).
\endaligned
$$
We justify the second estimate above as in \cite[Section 5.1]{CD-DivForm}, \cite[Section 5.2]{CD-Mixed} or \cite[Section 5]{CD-mult}.
Now use, consecutively, Proposition \ref{p: Drska:CCA 92:46} and
Lemma \ref{l: Drska:Berwolf 44:43} to arrive at
$$
\aligned
\left(\left|D^2\gX\right|*\f_\nu\right)(\omega)
\ \leqsim
& \ \nu^{-(1-q/p)}(\sw+1)\\
&  +\left(\nu^{p/q}+\su^{p/q}\right)+\left(\nu^{p-2}+\su^{p-2}\right)
\\
&
+\left(\nu^{q/p}+\sv^{q/p}\right)
+\left(\nu^{q-2}+\sv^{(q-2)_+}\right)
\hskip2pt
+\left(\nu^{q(1-2/p)}+\sv^{q(1-2/p)}\right)
\hskip2pt
+ (\nu+\sv)\\
&
\hskip 77pt
+ \left(\nu^{r-2}+\sw^{(r-2)_+}\right)
+\left(\nu^{r(1-2/p)}+\sw^{r(1-2/p)}\right)
+ (\nu+\sw).
\endaligned
$$
By taking into account that $0<\nu\leq1$, we simplify this through Lemma \ref{l: fade away} as
$$
\aligned
\left(\left|D^2\gX\right|*\f_\nu\right)(\omega)
\ \leqsim &\ \nu^{\min\{q,r\}-2}+|\omega|^{M-2}
+ \nu^{-(1-q/p)}(|\omega|+1).
\endaligned
$$
Since $q/p<q(1/p+1/r)=q-1$, we have
$\min\{q,r,q/p+1\}=m$, which settles {\it(\ref{eq: D2})}.

\smallskip
We conclude by proving {\it(\ref{eq: D1})}.
Let $j\in \{1,\hdots,2N\}$. Since $\gX\circ\cW_{3,1}^{-1}$ and $\varphi_{\nu}\circ\cW_{3,1}^{-1}$ are even functions in each of the variables in $\R^6$,
the function $(\gX*\varphi_{\nu})\circ\cW_{3,1}^{-1}$ also has this property, so
\begin{equation}
\label{eq: ENG-SWE}
D_{j}(\gX*\varphi_{\nu})(0)=0.
\end{equation}
Hence, by item {\it (\ref{eq: D2})} and the mean value theorem, if $|\omega|<\nu\leq 1$ we get
$$
\mod{D_{j}(\gX*\varphi_{\nu})(\omega)}
\leq\underset{|\eta|\leq 1}{\max}\mod{D^{2}(\gX*\varphi_{\nu})(\eta)}|\omega|
\leqsim\nu^{m-2}
|\omega|
\leqsim|\omega|^{m-1}.
$$
By Corollary \ref{c: Naklofen} {\it(\ref{eq: X2})} and Lemma \ref{l: T34},
if $|\omega|\geq\nu$ we get
\begin{equation*}
\mod{D_{j}(\gX*\varphi_{\nu})(\omega)}\leqsim
|\omega|^{\max\{p,r\}-1}+|\omega|^{\min\{q,r\}-1}.
\end{equation*}
Recalling that $q\geq q/p+1$,
we complete the proof of {\it(\ref{eq: D1})}.
\end{proof}

\begin{remark}
An important property required for $\cR_{n,\nu}$ in \cite{CD-Mixed} was that
\begin{equation}
\label{eq: Problem}
F\in L^p\cap L^q
\hskip 20pt
\Rightarrow
\hskip 20pt
(D\cR_{n,\nu})(F)\in L^{p'} + L^{q'}.
\end{equation}
It was needed in \cite[(37)]{CD-Mixed}, namely, to justify the use of the Lebesgue dominated convergence theorem there, for which one needed the integrand to be in $L^1$ uniformly with respect to $n$, and to this end one applied the H\"older's inequality.
In order to achieve the above property, one may pose the following question: if $F\in L^p$, for which $\gamma$ do we necessarily have $|F|^\gamma\in L^{p'}$? Of course, the (unique) answer is $\gamma=p-1$. Therefore we achieve \eqref{eq: Problem} if we prove
$$
|(D\cR_{n,\nu})(\omega)|\leqsim|\omega|^{p-1}+|\omega|^{q-1}.
$$
This is precisely \cite[Theorem 16 (iv)]{CD-Mixed}. It follows (in part) from
\begin{equation}
\label{eq: DQ}
|D(\cQ*\f_{\nu})(\omega)|\leqsim|\omega|^{p-1}+|\omega|^{q-1},
\end{equation}
which is \cite[Lemma 14 (ii)]{CD-Mixed}.

Trying to repeat the argument from \cite{CD-Mixed} in the current (trilinear) setting, we would thus
probably like to have the estimate
$$
|D(\gX*\f_{\nu})(\omega)|\leqsim|\omega|^{p-1}+|\omega|^{q-1}+|\omega|^{r-1}.
$$
Instead, Corollary \ref{c: izumrud} {\it(\ref{eq: D1})} gives
\begin{equation}
\label{eq: DX}
|D(\gX*\f_{\nu})(\omega)|\leqsim|\omega|^{p-1}+|\omega|^{q/p}
+|\omega|^{r-1}.
\end{equation}
Yet in spite of some discrepancy at the first glance, the estimate \eqref{eq: DX} is consistent with \eqref{eq: DQ}, because when $1/p+1/q=1$ (which was the case in \cite{CD-Mixed}, but not here), we exactly have $q/p=q-1$. So we seem to be on the right path.

An estimate like \eqref{eq: DX}, however, permits the implication
$$
F\in L^p\cap L^q\cap L^r
\hskip 20pt
\Rightarrow
\hskip 20pt
\left[D(\gX*\f_{\nu})\right](F)\in L^{p'} + L^{p} + L^{r'}.
$$
For reasons hinted at here and which will become completely apparent later, see Section \ref{s: Tsukahara Bokuden}, we will initially assume that $F=(f,g,h)\in L^{p}\cap L^{p'}\cap L^{r}$.
Note that, by the end of the proof of Lemma \ref{l: Ce rau am facut la lume} and an interpolation of Lebesgue spaces, this implies $F\in L^q$. To reach the class of $f,g,h$ as stated in the formulation of Theorem \ref{t: trilinemb}, we will then apply approximation arguments.
\end{remark}

The next estimate has its roots in \cite[(34)]{CD-Mixed}.

\begin{corollary}
\label{c: schlagwerk}
If $|\omega|\sim n$ for some $n\in\N$, then for any $\nu\in(0,1]$ we have
$$
\aligned
\left|D\left[\psi_n\cdot(\gX*\f_\nu)\right](\omega)\right|
&\,\leqsim\,
|\omega|^{M-1}\\
\left|D^2\left[\psi_n\cdot(\gX*\f_\nu)\right](\omega)\right|
&\,\leqsim\,
\nu^{m-2}|\omega|^{M-2}.
\endaligned
$$
\end{corollary}

\begin{proof}
For both inequalities we use the Leibniz rule. Regarding the first-order derivatives,
$$
\aligned
\left|D\left[\psi_n\cdot(\gX*\f_\nu)\right]\right|
& \,\leqsim\,
|\psi_n\cdot D (\gX*\f_\nu)| +
|D\psi_n\cdot (\gX*\f_\nu)|\\
& \,\leqsim\,
|D (\gX*\f_\nu)| +
\frac1n|\gX*\f_\nu|.
\endaligned
$$
Now apply Corollary \ref{c: izumrud} {\it (\ref{eq: D0})}, {\it (\ref{eq: D1})}.

Regarding the second-order derivatives,
$$
\aligned
\left|D^2\left[\psi_n\cdot(\gX*\f_\nu)\right]\right|
& \,\leqsim\,
|\psi_n\cdot D^2 (\gX*\f_\nu)| +
|D\psi_n\cdot D(\gX*\f_\nu)| +
|D^2\psi_n\cdot(\gX*\f_\nu)|\\
& \,\leqsim\,
|D^2 (\gX*\f_\nu)| +
\frac1n|D(\gX*\f_\nu)| +
\frac1{n^2}|\gX*\f_\nu|.
\endaligned
$$
Now apply Corollary \ref{c: izumrud}.
\end{proof}

Now we seem to be in a position to fix the problem described at the beginning of this section (Section \ref{s: banana}). Following \cite{CD-Mixed}, we try by adding to $\psi_n\cdot(\gX*\f_\nu)$ a function which is:
\begin{itemize}
\item
$\bA$-convex everywhere;
\item
strictly $\bA$-convex for $|\omega|\sim n$,
to the extent sufficient to
compensate for the lack of convexity of $\psi_n\cdot(\gX*\f_\nu)$ there.
\end{itemize}
Based on the analysis leading to Corollary \ref{c: schlagwerk}, and on Lemma \ref{l: GCDFEsDCD}, we think of adding the function
$
\nu^{m-2}P_{M}(u,v,w),
$
where $m$ and $M$ are as in \eqref{eq: mango} and $P_M$ is as in \eqref{eq: GFEsDCEs}.
Therefore at this stage our candidate is
$$
\psi_n\cdot(\gX*\f_\nu)+\nu^{m-2}P_{M}.
$$
By this we mean that we want to consider the flow
$$
\text{\calligra L }\left(\psi_n\cdot(\gX*\f_\nu)+\nu^{m-2}P_{M}
\right)({\mathtt f},{\mathtt g},{\mathtt h})
$$
and then send $n\rightarrow\infty$ and $\nu\rightarrow0$.

\subsection{Fourth (and final) approximation}

\label{s: New Fossils Band}
In order to perform integration by parts, we need properties akin to \cite[Theorem 16]{CD-Mixed}.
We see that the perturbation $\nu^{m-2}P_{M}$ is still not sufficient for that. For example, its second-order derivatives are not in $L^\infty$, its gradient and Hessian do not converge to the gradient and, respectively, Hessian of $\gX*\f_\nu$ as $n\rightarrow\infty$, and so on.

Again we resort to \cite{CD-Mixed} for a hint on how to resolve this. Namely, we try to replace the perturbation $P_{M}$ by its smooth tamed version, as follows. This section will serve as a preparation for that step, while the actual construction will then be made in Section \ref{s: Mravaljamiero}.

\medskip
Given $a>1$ define
$$
\cD_a(t)=
\left\{
\begin{array}{ccr}
{\displaystyle t^a} & ; & 0\leq t\leq1, \\
a t - (a-1) & ; & t\geq1.
\end{array}
\right.
$$
Observe that $\cD_a\circ\cD_b=\cD_{ab}$ and that $\cD_a$ is continuously differentiable on $[0,\infty)$
with
\begin{equation*}
\cD_a'(t)=
a\left\{
\begin{array}{ccr}
{\displaystyle t^{a-1}} & ; & 0\leq t\leq1, \\
1 & ; & t\geq1.
\end{array}
\right.
\end{equation*}
For ${s}>2$, $n\in\N$, $\e>0$ and $N\in\N$ define the function $\cF_{{s},n,\e}$
by
\begin{equation*}
\cF_{{s},n,\e}(\omega)=
n^{s}\,\cD_{\frac{{s}+\e}{2}}\left(\Big|\frac{\omega}{n}\Big|^2\right)
\hskip 30pt\omega\in\C^N.
\end{equation*}
Explicitly,
\renewcommand{\arraystretch}{1.8}
$$
\cF_{{s},n,\e}(\omega)=
\left\{
\begin{array}{ccl}
\ \hskip -48pt
n^{-\e}|\omega|^{{s}+\e}&;& |\omega|\leq n,\\
{\displaystyle
   \frac{{s}+\e}{2}\,n^{{s}-2}|\omega|^2-\left(\frac{{s}+\e}{2}-1\right)n^{{s}}
}&;&|\omega|\geq n.
\end{array}
\right.
$$
\renewcommand{\arraystretch}{1}

\label{Rustavi}
Let $A_1,\hdots,A_N$ be accretive matrices on $\Omega\subseteq \R^d$.
Write $\bA=(A_1,\hdots,A_N)$.
Suppose that $\Delta_{s}(\bA)>0$ and let $c=c({s},\bA,N)>0$ be any constant fitting Lemma \ref{l: GCDFEsDCD}.
Note that the dependence of $c$ on $\bA$ is limited to $\Delta_s(\bA)$ and $\Lambda(\bA)$.

By continuity of $\Delta_t(\bA)$ in $t$ we have $\Delta_{s+\e}(\bA)>0$ for sufficiently small $\e>0$.
Choose and fix such an $\e>0$.
In analogy with \eqref{eq: GFEsDCEs} we then define function $\cP_{{s},n,\e}:\C^N\rightarrow[0,\infty)$ by
\begin{equation*}
\cP_{{s},n,\e}(u_1,\hdots,u_N):=\cF_{{s},n,\e}(u_1,\hdots,u_N)+c\sum_{j=1}^{N}\cF_{{s},n,\e}(u_j).
\end{equation*}

Define the set $\Theta_n\subseteq \C^N$ by
$$
\Theta_n=\{|(u_1,\hdots,u_N)|=n\}\cup\bigcup_{j=1}^N\{|u_j|=n\}.
$$
The next proposition is based on \cite[Proposition 12]{CD-Mixed}.
\begin{proposition}
\label{p: 1949 Buick Roadmaster}
Under the above assumptions on $\bA,s,\e$, we have:
\begin{enumerate}[\rm (i)]
\item
\label{eq: Bugatti}
$\cP_{ s,n,\e}\in C^{1}(\C^{N})\cap C^{2}(\C^{N}\setminus\Theta_{n})$ for all $n\in\N$. Moreover,
$$
\aligned
D\cP_{ s,n,\e}\rightarrow 0
&\quad  \text{pointwise in  } \C^{N},\\
D^{2}\cP_{ s,n,\e}\rightarrow 0 &
\quad\text{pointwise in  }\C^{N}\setminus\bigcup_{k\in\N}\Theta_{k}
\endaligned
$$
as $n\rightarrow\infty$.

\item
\label{eq: Delahaye}
$\cP_{ s,n,\e}$ is $\bA$-convex in $\C^{N}\setminus\Theta_{n}$, for all $n\in\N$.
Moreover, for all $n\in\N$, $X=(X_1,\hdots,X_N)$ with $X_j\in\C^d$, and $u\in \C^{N}\setminus\Theta_{n}$ with $|u|>n$, we have
$$
H^{\bA}_{\cP_{ s,n,\e}}[u;X]\geq ( s+\e)n^{ s-2}\lambda(\bA)|X|^{2}.
$$

\item
\label{eq: Hispano-Suiza}
There exists $C>0$ that does not depend on $n$ such that for all $n\in\N$ we have
$$
\aligned
\mod{(D \cP_{ s,n,\e})(\omega)}&\leq C|\omega|^{ s-1},\quad \forall\omega\in\C^{N},\\
\mod{(D^{2}\cP_{ s,n,\e})(\omega)}&\leq C|\omega|^{ s-2},\quad \forall\omega\in\C^N\setminus\Theta_{n}.
\endaligned
$$
\item
\label{eq: Tucker}
For every $n\in\N$ there exists $C(n)>0$ such that
$$
\mod{(D\cP_{ s,n,\e})(\omega)}\leq C(n)|\omega|,\quad \forall\omega\in\C^{N}.
$$
\item
\label{eq: Packard}
For all $n\in\N $ we have
$\mod{D^{2}\cP_{ s,n,\e}}\in L^{\infty}(\C^{N}\setminus\Theta_{n})$.

\end{enumerate}

\end{proposition}

\begin{proof}
Item \eqref{eq: Bugatti} holds since, fixing $u=(u_1,\hdots,u_N)$, for $n$ large enough ($n>|u|$) we have
\begin{equation}
\label{eq: Duesenberg}
\cP_{ s,n,\e}(u)=
n^{-\e}\left(|u|^{ s+\e}+c\sum_{j=1}^N|u_j|^{ s+\e}\right),
\end{equation}
where the term in the parentheses is clearly independent of $n$.

We now prove item (\ref{eq: Delahaye}). Suppose first that $|u|<n$; then $|u_j|<n$ for $j=1,\hdots,N$. In this case \eqref{eq: Duesenberg} holds, therefore $\cP_{ s,n,\e}(u)=n^{-\e}P_{ s+\e}(u)$ for all $u\in \C^{N}\setminus\Theta_{n}$ and the $\bA$-convexity follows from Lemma~\ref{l: GCDFEsDCD}.

Suppose now that
$|u|>n$. Then
$$
\cP_{ s,n,\e}(u)=\frac{ s+\e}{2}\,n^{ s-2}|u|^2+\left(1-\frac{ s+\e}{2}\right)n^{ s}
+c\sum_{j=1}^N\cF_{ s,n,\e}(u_j).
$$

Therefore, by \cite[(5.7)]{CD-DivForm},
$$
\aligned
H^{\bA}_{\cP_{ s,n,\e}}[u;X]&=( s+\e)n^{ s-2}\sum_{j=1}^N\Re\sk{A_jX_{j}}{X_{j}}
+c\sum_{j=1}^N
H^{A_j}_{\cF_{ s,n,\e}}[u_j;X_j]\\
&\geq( s+\e)n^{ s-2}\lambda(\bA)|X|^{2}+c\sum_{j=1}^N
H^{A_j}_{\cF_{ s,n,\e}}[u_j;X_j].
\endaligned
$$
Since
\renewcommand{\arraystretch}{1.8}
$$
H^{A_j}_{\cF_{ s,n,\e}}[u_j;X_j]=
\left\{
\begin{array}{rcl}
{\displaystyle n^{-\e}H^{A_j}_{F_{ s+\e}}[u_j;X_j]} & ; & |u_j|<n;\\
{\displaystyle \frac{ s+\e}{2}n^{ s-2}H^{A_j}_{F_{2}}[u_j;X_j]} & ; &|u_j|>n
\end{array}
\right.
$$
and $\Delta_{ s+\e}(\bA)>0$, we deduce from \cite[Proposition 5.8 and Corollary 5.16]{CD-DivForm}
that $H^{A_j}_{\cF_{ s,n,\e}}[u_j;X_j]\geq0$ for every $j\in\{1,\hdots,N\}$, which finishes the proof of item \eqref{eq: Delahaye}.
\renewcommand{\arraystretch}{1}

Item \eqref{eq: Hispano-Suiza} follows from the estimate
\renewcommand{\arraystretch}{1.8}
\begin{equation}
\label{eq: Horch}
\aligned
\mod{D\cF_{ s,n,\e}(\omega)}
&\leqsim
\left\{
\begin{array}{ccl}
n^{-\e}|\omega|^{ s+\e-1}&;& |\omega|\leq n,\\
n^{ s-2}|\omega|&;&|\omega|\geq n
\end{array}
\right.
\\
&\leqsim\mod{\omega}^{ s-1}
\endaligned
\end{equation}
and similarly for second-order derivatives.
\renewcommand{\arraystretch}{1}

Item \eqref{eq: Tucker} likewise follows from \eqref{eq: Horch}.

Item \eqref{eq: Packard} follows by noting that outside a ball in $\C^N$, the function $\cF_{ s,n,\e}$ is defined as $A|\omega|^2+B$ for some constants $A,B$.
\end{proof}

\medskip
Now we pass to the estimates for $\cP_{s,n,\e}*\f_\nu$.
Since $ \cP_{s,n,\e}\in C^1(\C^{N})$ and its second-order partial derivatives exist on $\C^{N}\setminus \Theta_{n}$ and extend to a locally integrable function on $\C^{N}$, by the ACL characterization of Sobolev spaces (see, for example, \cite[Th\'eor\`eme~V, p. 57]{Schwartz} or \cite[Theorem~11.45]{Leoni}) we have\begin{equation}
\label{eq: Oudinot}
\aligned
D(\cP_{s,n,\e}*\varphi_{\nu})&=(D\cP_{s,n,\e})*\varphi_{\nu},\\
D^{2}(\cP_{s,n,\e}*\varphi_{\nu})&=(D^{2}\cP_{s,n,\e})*\varphi_{\nu}.
\endaligned
\end{equation}

The following statement closely resembles \cite[Proposition 15]{CD-Mixed}.
For the reader's convenience, we give a complete proof (which to a significant extent uses Proposition \ref{p: 1949 Buick Roadmaster}).

\begin{proposition}
\label{p: rikverc}
Assume the conditions on $s,\bA,\e$ as on page \pageref{Rustavi} and let $\nu\in (0,1)$.
\begin{enumerate}[\rm (i)]
\item
\label{eq: alons}
We have
$$
\aligned
D\left(\cP_{s,n,\e}*\f_\nu\right)&\rightarrow 0,\\
D^{2}\left(\cP_{s,n,\e}*\f_\nu\right)&\rightarrow 0
\endaligned
$$
pointwise in $\C^{N}$ as $n\rightarrow\infty$.
\item
\label{eq: enfants}
The function $\cP_{s,n,\e}*\f_\nu$ is $\bA$-convex in $\C^{N}$. Moreover, for all $n\in\N$, $X=(X_1,\hdots,X_N)$ with $X_j\in\C^d$, and all $\omega\in\C^{N}$ with $|\omega|>2n$,
$$
H_{ \cP_{s,n,\e}*\varphi_{\nu}}^{\bA}[\omega;X]
\geq (s+\e)n^{s-2}\lambda(\bA)|X|^{2}.
$$
\item
\label{eq: de}
There exists $C>0$ that does not depend on $n$ and $\nu$ such that
for any $\text{\v{s}}\leq2$ we have
$$
\mod{D(\cP_{s,n,\e}*\f_\nu)(\omega)}\leq C\left(|\omega|^{s-1}+|\omega|^{\text{\v{s}}-1}\right),
\quad
\forall\omega\in\C^{N},\quad\forall n\in\N.
$$

\item
\label{eq: la}
For every $n\in\N $ there exists $C(n)>0$ (that does not depend on $\nu$) such that
$$
\mod{D(\cP_{s,n,\e}*\varphi_{\nu})(\omega)}\leq C(n)|\omega|,\quad \forall\omega\in\C^{N}.
$$

\item
\label{eq: patrie}
We have $\mod{D^{2}(\cP_{s,n,\e}*\varphi_{\nu})}\in L^{\infty}(\C^{N})$
with
$$
\nor{D^{2}(\cP_{s,n,\e}*\varphi_{\nu})}_{\infty}\leq C
$$
for some $C>0$ that does not depend on $\nu$.
\end{enumerate}
\end{proposition}
\begin{proof}
Item~(\ref{eq: alons}) follows by combining \eqref{eq: Oudinot}, Proposition~\ref{p: 1949 Buick Roadmaster} (\ref{eq: Bugatti}) and (\ref{eq: Hispano-Suiza}) with the dominated convergence theorem.

Item~(\ref{eq: patrie}) follows from \eqref{eq: Oudinot} and Proposition~\ref{p: 1949 Buick Roadmaster} (\ref{eq: Packard}).

By \eqref{eq: Oudinot} we have
$$
H_{ \cP_{s,n,\e}*\varphi_{\nu}}^{\bA(x)}[\omega;X]
=\int_{\C^N}H_{ \cP_{s,n,\e}}^{\bA(x)}[\omega-\omega^{\prime};X]\varphi_\nu(\omega^{\prime})\wrt \omega^{\prime},
$$
for all $x\in\Omega$, $\omega\in\C^{N}$ and $X\in(\C^{d})^N$.
If we assume that $|\omega|>2n$, since the support of the integrand is contained in $B_{\C^{N}}(0,\nu)$, we have $|\omega-\omega'|>2n-\nu>n$. Therefore we may estimate the integrand by means of Proposition~\ref{p: 1949 Buick Roadmaster}~(\ref{eq: Delahaye}) almost everywhere on $B_{\C^{N}}(0,\nu)$ and thus prove item~(\ref{eq: enfants}).

Let us address item (\ref{eq: de}). We proceed much as in the proof of Corollary~\ref{c: izumrud}~(\ref{eq: D1}). First consider $|\omega|\leq1$.
By smoothness and evenness properties of $\cP_{s,n,\e}*\varphi_{\nu}$,
\begin{equation}
\label{eq: Massena}
D(\cP_{s,n,\e}*\varphi_{\nu})(0)=0.
\end{equation}
Hence, the second identity in \eqref{eq: Oudinot}, the second estimate of Proposition~\ref{p: 1949 Buick Roadmaster}~(\ref{eq: Hispano-Suiza}) and the mean value theorem imply
$$
\mod{D(\cP_{s,n,\e}*\varphi_{\nu})(\omega)}\leq C|\omega|\leq C|\omega|^{\text{\it\v{s}}-1},\quad \forall |\omega|\leq 1,\quad \forall n\in\N .
$$
Now take $|\omega|>1$.
From the first identity in \eqref{eq: Oudinot}, the first estimate of Proposition~\ref{p: 1949 Buick Roadmaster}~(\ref{eq: Hispano-Suiza}) and Lemma \ref{l: Drska:Berwolf 44:43}
we get
\begin{equation*}
\mod{D (\cP_{s,n,\e}*\varphi_{\nu})(\omega)}\leq C|\omega|^{s-1}.
\end{equation*}
Thus we proved (\ref{eq: de}).

Finally, item (\ref{eq: la}) follows from item (\ref{eq: patrie}), \eqref{eq: Massena} and the  mean value theorem.
\end{proof}

\subsection{The sequence $\gX_{n,\nu}$}
\label{s: Mravaljamiero}
Here we finally present the function that will be the backbone of our heat-flow process. We start by summarizing the assumptions.

Let $p,q,r\in(1,\infty)$ satisfy $1/p+1/q+1/r=1$ and $p\geq q$. Let $M,m$ be as in \eqref{eq: mango}. Furthermore, take an open set $\Omega\subseteq \R^d$ and $A,B,C\in\cA_M(\Omega)$. Write $\bA=(A,B,C)$. Choose $\e>0$ such that we also have  $A,B,C\in\cA_{M+\e}(\Omega)$.

We will apply the results from Section \ref{s: New Fossils Band} with $N=3$ and $s=M$. Since $\e$ will stay fixed throughout the process, for the sake of transparency we will drop it from the indices. Thus we will write just  $\cP_{M,n}$ instead of $\cP_{M,n,\e}$.

The final stage of our construction is the function
\begin{equation*}
\gX_{n,\nu}:=\psi_n\cdot(\gX*\f_\nu)+\siC\nu^{m-2}(\cP_{M,n}*\f_\nu)
\end{equation*}
with $\siC>0$ to be defined in the following theorem.
The latter is a ``three-variable counterpart'' of \cite[Theorem 16]{CD-Mixed} and the function $\gX_{n,\nu}$ is the analogue of $\cR_{n,\nu}$ from \cite{CD-Mixed}.

\begin{theorem}
\label{t: Jugoplastika 1989}
Let $\nu\in (0,1]$. There exists $\siC>0$, depending on $\psi,p,q,r,\bA$ and the $*$-ellipticity constants stipulated in Theorem \ref{t: trilinemb}, but not depending on $\nu$ or $n$, such that $\gX_{n,\nu}$ is $\bA$-convex in $\C^{3}$ for all $n\in\N $. Moreover, the following statements hold.
\begin{enumerate}[{\rm (i)}]
\item
\label{eq: Maljkovic}
We have
$$
\aligned
D\gX_{n,\nu}  &\rightarrow D(\gX*\varphi_{\nu}),\\
D^2\gX_{n,\nu} &\rightarrow D^2(\gX*\varphi_{\nu})
\endaligned
$$
pointwise in $\C^3$ as $n\rightarrow\infty$.
\item
\label{eq: Kukoc}
For any $n\in\N $ there exists $C=C(n,\nu,\siC)>0$ such that
$$
\mod{(D\gX_{n,\nu})(\omega)}\leq C|\omega|,\quad \forall\omega\in\C^{3}.
$$
\item
\label{eq: Radja}
There exists $C>0$ that does not depend on $n$ such that
$$
\mod{(D\gX_{n,\nu})(\omega)}\leq C\nu^{m-2}
\left(
|\omega|^{m-1}
+
|\omega|^{M-1}
\right)
$$
for all $\omega\in\C^{3}$, $n\in\N $ and $\nu\in (0,1]$.
\item
\label{eq: Ivanovic}
For any $n\in\N $ and $\nu>0$ we have
$$
(\partial_{\bar u} \gX_{n,\nu})(0,v,w)=
(\partial_{\bar v} \gX_{n,\nu})(u,0,w)=
(\partial_{\bar w} \gX_{n,\nu})(u,v,0)=0,\quad
$$
for all $u,v,w\in\C$.
\item
\label{eq: Sobin}
$|D^{2}\gX_{n,\nu}|\in L^{\infty}(\C^{3})$.
\end{enumerate}
\end{theorem}
\begin{proof}
We address the statements of the theorem one by one.

\smallskip
$\bullet$
Let us first prove the $\bA$-convexity in the region $\{|\omega|< 3n\}\cup\{|\omega|>4n\}$.
It follows from the $\bA$-convexity of $\gX*\varphi_{\nu}$ and $\cP_{M,n}*\varphi_{\nu}$.
Indeed, from Corollary \ref{c: Naklofen} {\it(\ref{eq: X3})} we have
$$
H^{(A,B,C)(x)}_{\gmX_\nu}[\omega;X]
\,\geqsim\,-\nu|X|^2.
$$
By combining this with Proposition~\ref{p: rikverc}~(\ref{eq: enfants}) we get, for sufficiently large $\siC $,
$$
H^{(A,B,C)(x)}_{\gmX_{n,\nu}}[\omega;X]
\,\geqsim\,\left[\siC \nu^{m-2}(M+\e)n^{M-2}\lambda(\bA)-\nu\right]|X|^2.
$$
Since $\nu\in(0,1]$ and $m-2\leq0$ we thus get
$$
H^{(A,B,C)(x)}_{\gmX_{n,\nu}}[\omega;X]
\,\geqsim\,\left[\siC (M+\e)n^{M-2}\lambda(\bA)-1\right]\nu|X|^2.
$$
Thus we see that this quantity is positive if $\siC $ is large enough.

In order to achieve $\bA$-convexity in the region $\{3n\leq|\omega|\leq 4n\}$, we choose $\siC $ large enough and combine Corollary \ref{c: schlagwerk} with the second part of Proposition~\ref{p: rikverc}~(\ref{eq: enfants}).

\smallskip
$\bullet$
Item~(\ref{eq: Maljkovic}) is a trivial consequence of Proposition~\ref{p: rikverc} (\ref{eq: alons}) and the definition of $\psi_n\cdot\left(\gX*\f_\nu\right)$.

\smallskip
$\bullet$
From \eqref{eq: ENG-SWE} and the fact that $\psi_n\equiv1$ in a neighbourhood of $0$, we conclude that
$$
(D\left[\psi_n\cdot\left(\gX*\f_\nu\right)\right])(0)=0.
$$
Hence, by the mean value theorem and the fact that $\psi_n\cdot\left(\gX*\f_\nu\right)\in C_c^\infty(\C^{3})$, we get
$$
\mod{(D\left[\psi_n\cdot\left(\gX*\f_\nu\right)\right])(\omega)}\leq C(\nu,n)|\omega|.
$$
Item~(\ref{eq: Kukoc}) follows from here and Proposition~\ref{p: rikverc}~(\ref{eq: la}).

\smallskip
$\bullet$
In order to prove (\ref{eq: Radja}) we separately estimate $|D\left[\psi_n\cdot(\gX*\varphi_\nu)\right](\omega)|$ and $|D\left(\cP_{M,n}*\f_\nu\right)|$.
The estimate of $|D\left[\psi_n\cdot(\gX*\varphi_\nu)\right](\omega)|$ follows by
Corollary \ref{c: izumrud} {\it(\ref{eq: D1})} (for $|\omega|\not\in[3n,4n]$, since $D\psi_n\equiv0$ there) and Corollary \ref{c: schlagwerk} (for $|\omega|\in[3n,4n]$).

On the other hand, the estimate of $|D\left(\cP_{M,n}*\f_\nu\right)|$ is Proposition~\ref{p: rikverc}~(\ref{eq: de}), used with $s=M$ and $\text{{\it\v{s}}}=m$.

\smallskip
$\bullet$
To prove item~(\ref{eq: Ivanovic}) just observe that $\gX_{n,\nu}$ is smooth and even in each of the variables, because both $\gX*\varphi_{\nu}$ and $\cP_{M,n}*\varphi_{\nu}$ have this property.

\smallskip
$\bullet$
Item~(\ref{eq: Sobin}) follows from Proposition~\ref{p: rikverc} (\ref{eq: patrie}) and the fact that $\psi_n\cdot(\gX*\f_\nu)\in C^{2}_c(\C^{3})$.
\end{proof}

As a consequence of Theorem \ref{t: Jugoplastika 1989}, items \eqref{eq: Ivanovic} and \eqref{eq: Sobin}, we have the following invariance result, which is modelled after (and proven exactly as) \cite[Lemma 19]{CD-Mixed}.

\begin{lemma}
\label{l: Philips}
If ${\mathtt f},{\mathtt g},{\mathtt h}\in\oU$ then
$
\partial_{u}\gX_{n,\nu}\left({\mathtt f},{\mathtt g},{\mathtt h}\right)\in \oU,
$
and the same for $\partial_{v}\gX_{n,\nu}$ and $\partial_{w}\gX_{n,\nu}$.
\end{lemma}

\subsection{Completion of the proof of Theorem \ref{t: trilinemb}}
\label{s: Tsukahara Bokuden}
Recall that we need to prove \eqref{eq: Cherokee}.
So take $\tf,\tg,\th\in\oU$ such that $\tf,\tg,\th,L_{A}\tf, L_{B}\tg,L_C\th\in\big(L^p\cap L^{p'}\cap L^{r}\big)(\Omega)$. As noted before, this intersection is contained in $L^q$.

Recalling the notation \eqref{eq: nzapazap}, we will first prove that
\begin{eqnarray}
\label{eq: Shoshoni}
\text{\calligra L }(\gX)({\mathtt f},{\mathtt g},{\mathtt h})
=\lim_{\nu\searrow 0}\lim_{n\rightarrow\infty}
\text{\calligra L }(\gX_{n,\nu})({\mathtt f},{\mathtt g},{\mathtt h}).
\end{eqnarray}
Let us justify \eqref{eq: Shoshoni}. First consider the limit as $n\rightarrow\infty$. We want to use the Lebesgue dominated convergence theorem. Thus we must prove that for all $n\in\N$, the integrands in
\begin{equation}
\label{eq: Cheyenne}
\text{\calligra L }(\gX_{n,\nu})({\mathtt f},{\mathtt g},{\mathtt h})
=
\Re\int_{\Omega}
\Big(
  (\partial_{u}\gX_{n,\nu})\left({\mathtt f},{\mathtt g},{\mathtt h}\right) L_{A}{\mathtt f}
+(\partial_{v}\gX_{n,\nu})\left({\mathtt f},{\mathtt g},{\mathtt h}\right) L_{B}{\mathtt g}
+(\partial_{w}\gX_{n,\nu})\left({\mathtt f},{\mathtt g},{\mathtt h}\right) L_{C}{\mathtt h}
\Big)
\end{equation}
admit a majorant which lies in $L^1(\Omega)$ and is independent of $n$. It is enough to treat the first summand, as the other two can be estimated in exactly the same manner.

By Theorem~\ref{t: Jugoplastika 1989}~(\ref{eq: Radja}) we have, with $\omega=\left({\mathtt f},{\mathtt g},{\mathtt h}\right)$,
\begin{equation*}
\mod{  (\partial_{u}\gX_{n,\nu})\left({\mathtt f},{\mathtt g},{\mathtt h}\right) L_{A}{\mathtt f}
}
\,\leqsim\,
|\omega|^{p-1}\left|L_{A}{\mathtt f}\right|+|\omega|^{q/p}\left|L_{A}{\mathtt f}\right|+|\omega|^{r-1}\left|L_{A}{\mathtt f}\right|.
\end{equation*}
(We omitted copying from Theorem~\ref{t: Jugoplastika 1989}~(\ref{eq: Radja}) the power of $\nu$, since at this stage we consider it as constant.)
The above majorant clearly does not depend on $n$. We claim that it belongs to $L^1$. Indeed, use that $L_{A}{\mathtt f}\in L^p\cap L^{p'}\cap L^{r}$ and that $|\omega|\in L^p\cap L^q\cap L^{r}$.
It then follows that
$$
|\omega|^{p-1}\left|L_{A}{\mathtt f}\right|+|\omega|^{q/p}\left|L_{A}{\mathtt f}\right|
+|\omega|^{r-1}\left|L_{A}{\mathtt f}\right|
\in
L^{p'}\cdot L^{p} + L^{p}\cdot L^{p'} + L^{r'}\cdot L^{r},
$$
which by the Hölder's inequality belongs to $L^{1}$. Thus, by Theorem~\ref{t: Jugoplastika 1989}~(\ref{eq: Maljkovic}), we proved
\begin{equation}
\label{eq: Lakota}
\lim_{n\rightarrow\infty}
\text{\calligra L }(\gX_{n,\nu})({\mathtt f},{\mathtt g},{\mathtt h})
=\text{\calligra L }(\gX*\f_\nu)({\mathtt f},{\mathtt g},{\mathtt h}).
\end{equation}
Now let us take the limit as $\nu\searrow0$.
We argue as before, just that now the adequate estimates are provided by Corollary~\ref{c: izumrud}~(\ref{eq: D1}). From the fact that $\gX\in C^{1}(\C^{3})$ and the Lebesgue dominated convergence theorem we thus deduce that
\begin{equation}
\label{eq: Comanche}
\lim_{\nu\searrow0}
\text{\calligra L }(\gX*\f_\nu)({\mathtt f},{\mathtt g},{\mathtt h})
=\text{\calligra L }(\gX)({\mathtt f},{\mathtt g},{\mathtt h}).
\end{equation}
The combination of \eqref{eq: Comanche} and \eqref{eq: Lakota} returns \eqref{eq: Shoshoni}.

\medskip
By Lemma~\ref{l: Philips}, we can integrate by parts the integral in \eqref{eq: Cheyenne}
and, from the chain rule for the composition of smooth functions with vector-valued Sobolev functions, deduce as in \cite[Corollary 4.2]{CD-DivForm} that
\begin{equation*}
\aligned
2\,
\Re\int_\Omega
&\Big(
  \partial_{u}\gX_{n,\nu}\left({\mathtt f},{\mathtt g},{\mathtt h}\right) L_{A}{\mathtt f}
+\partial_{v}\gX_{n,\nu}\left({\mathtt f},{\mathtt g},{\mathtt h}\right) L_{B}{\mathtt g}
+\partial_{w}\gX_{n,\nu}\left({\mathtt f},{\mathtt g},{\mathtt h}\right) L_{C}{\mathtt h}
\Big)\\
&=\int_{\Omega}H^{(A,B,C)}_{\gmX_{n,\nu}}\left[\cW_{3,1}\left({\mathtt f},{\mathtt g},{\mathtt h}\right);\cW_{3,d}\left(\nabla {\mathtt f}, \nabla {\mathtt g}, \nabla {\mathtt h}\right)\right],
\endaligned
\end{equation*}
so we merge this with \eqref{eq: Lakota} into
\begin{equation}
\label{eq: Kiowa}
2\text{\calligra L }(\gX*\f_{\nu})({\mathtt f},{\mathtt g},{\mathtt h})
=\lim_{n\rightarrow\infty}
\int_{\Omega}H^{(A,B,C)}_{\gmX_{n,\nu}}\left[\cW_{3,1}\left({\mathtt f},{\mathtt g},{\mathtt h}\right);\cW_{3,d}\left(\nabla {\mathtt f}, \nabla {\mathtt g}, \nabla {\mathtt h}\right)\right].
\end{equation}
By Theorem~\ref{t: Jugoplastika 1989}, the function $\gX_{n,\nu}$ is $\bA$-convex in $\C^{3}$, so the integrand on the right-hand side of \eqref{eq: Kiowa} is nonnegative for all $n\in\N$. Hence, by \eqref{eq: Kiowa}, Fatou's lemma and Theorem~\ref{t: Jugoplastika 1989} (\ref{eq: Maljkovic}), followed by Corollary~\ref{c: Naklofen}~\eqref{eq: X3},
$$
\aligned
2\text{\calligra L }(\gX*\f_{\nu})({\mathtt f},{\mathtt g},{\mathtt h})
&\geq \int_{\Omega}H^{(A,B,C)}_{\gmX*\f_\nu}\left[\cW_{3,1}\left({\mathtt f},{\mathtt g},{\mathtt h}\right);\cW_{3,d}\left(\nabla {\mathtt f}, \nabla {\mathtt g}, \nabla {\mathtt h}\right)\right]\\
&\geqsim\int_{\Omega}|\nabla {\mathtt f}||\nabla {\mathtt g}|({\mathtt h}-\nu),
\endaligned
$$
for all $\nu\in (0,1)$, where
the implied constant may depend on $p,q,r,A,B,C$ and their $*$-ellipticity constants alluded to in Theorem \ref{t: trilinemb}, but not on $\nu$. The desired inequality \eqref{eq: reduction} now follows from \eqref{eq: Comanche}.

\subsection{Growth of the embedding constants}

Here we make more explicit the behaviour of the embedding constants appearing in Theorem \ref{t: trilinemb}.

\begin{corollary}
\label{c: cc}
Under the assumptions of Theorem \ref{t: trilinemb},
when $p>q$, the embedding constant in \eqref{eq: trilinemb} can be estimated as
$$
\leqsim
\left(\frac{1}{p}\right)^{1/p}\left(\frac{D}{q}\right)^{1/q}\left(\frac{E}{r}\right)^{1/r}.
$$
Here $D,E$ are parameters of the function $\gX$ from Section \ref{subsec:case1}, with $D$ chosen specifically as in \eqref{2544}.
In particular, while the constant $E$ in general depends on the choice of the matrix $C$, the constant $D$ does not.

If $C=I$, the constant stays bounded when considering the triple $(p-\e,q+\e,r(\e))$ with conjugate exponents $p,q$ and sending $\e\rightarrow0$, as described in Section \ref{s: One Penny}. More precisely, in the limit $\e\rightarrow0$ the constants get majorized by
$$
\leqsim \frac{D^{1/q}}{p^{1/p}q^{1/q}}.
$$
\end{corollary}

\begin{proof}
The statement follows from Proposition \ref{p: Pat Riley} and the ``polarization trick'' (see the way Lemma \ref{l: Mos Def} was used on pages \pageref{polar1} and \pageref{polar2}). Choosing $D$ as in \eqref{2544} gives $\sqrt{\alpha_1D-\alpha_2}-\alpha_3=1$.

Note that just the last factor depends on $r$. And since in the case of $C=I$ we saw (Proposition \ref{p: Pat Riley}) that $E$ stays bounded as $r\rightarrow\infty$, we proved
the last part, too, as in this situation the rightmost term disappears with $r\rightarrow\infty$.
\end{proof}

\begin{lemma}
\label{l: final}
Let $\Omega,\oU,\oV,\oW,p,q,r,A,B,C,f,g,h$ be as in the formulation of Theorem \ref{t: trilinemb}.
Then for every $s\geq1$ we have
\begin{equation*}
\label{eq: trilwithpar}
\int^{\infty}_{0}\int_{\Omega}\mod{\nabla T^{A}_{st}f}\mod{\nabla T^{B}_{st}g}\mod{T^{C}_{t}h}\wrt x\wrt t
\,\leqsim\, s^{-1/r^{\prime}}\norm{f}{p}\norm{g}{q}\norm{h}{r}.
\end{equation*}
The implied embedding constants only depend on $p,q,r$ and $*$-ellipticity constants of $A,B,C$ alluded to in the % theorem's 
assumptions. 
\end{lemma}
\begin{proof}
Change the variable by $st=t^{\prime}$ and apply Proposition \ref{p: Pat Riley} and Corollary \ref{c: cc}.
\end{proof}

\section{Proof of Theorem \ref{t: princ}}
\label{s: KPproof}

In this section we prove Theorem \ref{t: princ}. The proof consists of the following principal elements:
\begin{itemize}
\item
dualization (Proposition \ref{p: dual inequality});
\item
Littlewood--Paley decomposition (Section \ref{s: LPdec});
\item
subordination to the imaginary powers of $L_A$ (Section \ref{s: Mellin});
\item
holomorphic functional calculus for $L_A$ \cite[Theorem 3]{CD-Mixed};
\item
trilinear embedding (Theorem \ref{t: trilinemb}) with control of the embedding constants for $(A,B,\delta C)$ in terms of $\delta>0$ (Corollary \ref{c: cc}).
\end{itemize}
We first either review or create anew some necessary tools.

\subsection{Preliminaries}
Fix $A\in\cA(\Omega)$. Recall that for simplicity we assume that $L_{A}$ is injective in $L^{2}(\Omega)$. We will systematically use the holomorphic functional calculus for injective sectorial operators on Banach spaces, see \cite{Mc, CDMY,Haase}.
\subsubsection{Abstract Littlewood-Paley decomposition}
For every $\alpha>0$ define
\[
\psi_{\alpha}(z):=z^{\alpha}e^{-z}, \quad \phi_{\alpha}(z):=\frac{1}{\Gamma(\alpha)}\int^{\infty}_{1}\psi_{\alpha}(sz)\frac{\wrt s}{s},\quad \Re z>0.
\]
A rapid calculation shows that for $\alpha\in\N_{+}$ we have
\begin{equation}
\label{eq: sum}
\phi_{\alpha}(z)
=\sum^{\alpha-1}_{j=0}\frac{\psi_{j}(z)}{j!}
=\left(
\sum^{\alpha-1}_{j=0}\frac{z^{j}}{j!}
\right)
e^{-z}.
\end{equation}
It follows that on $L^{2}(\Omega)$
we have
\begin{equation}
\label{eq: sumbis}
\phi_{\alpha}(tL_{A})
=\sum^{\alpha-1}_{j=0}\frac{1}{j!}\psi_{j}(tL_{A})
=\sum^{\alpha-1}_{j=0}\frac{(tL_{A})^{j}T^{A}_{t}}{j!},
\quad \forall \alpha\in\N_{+}.
\end{equation}
If $A$ is $\chi$-elliptic, then, by \cite{CD-Mixed}, the identity holds true in $L^{\chi}(\Omega)$.

The next result can be referred to as the {\it Calder\'on reproducing formula}, see \cite[Proposition 2.11]{BCF16} for a similar version.

\begin{lemma}
\label{l: L-P}
Let $\chi\in (1,+\infty)$. Suppose that $A$ is $\chi$-elliptic. Then 
we have
\begin{itemize}
\item 
\hskip 4pt $\lim_{t\downarrow 0}\phi_{\alpha}(tL_{A})f=f$ in $L^{\chi}(\Omega)$, 
\item 
$\lim_{t\uparrow \infty}\phi_{\alpha}(tL_{A})f=0$ in $L^{\chi}(\Omega)$, 
\end{itemize}
for every $\alpha>0$ and $f\in L^{\chi}(\Omega)$.
\end{lemma}

\begin{proof}
By \cite{CD-Mixed} the operator $L_{A}$ has bounded $H^{\infty}$-calculus of angle $\vartheta_{\chi}<\pi/2$ in $L^{\chi}(\Omega)$. For every $\vartheta\in [0,\pi/2)$ we have $\phi_{\alpha}\in H^{\infty}(\bS_{\vartheta})$. Therefore, for $0<\e<\pi/2-\vartheta_\chi$,
\begin{equation}
\label{eq: u1}
\sup_{t>0}\norm{\phi_{\alpha}(tL_{A})}{\chi}\leqsim \norm{\phi_{\alpha}}{H^{\infty}(\bS_{\vartheta_{\chi}+\varepsilon})}<\infty.
\end{equation}
Also, for all $z\in
\C_+$ we have
\[
\phi_{\alpha}(tz)=1-\frac{1}{\Gamma(\alpha)}\int^{t}_{0}(sz)^{\alpha}e^{-zs}\frac{\wrt s}{s},
\]
yielding
$$
\aligned
\lim_{t\searrow0}\phi_\alpha(tz) & =1\\
\lim_{t\rightarrow\infty}\phi_\alpha(tz) & =0.
\endaligned
$$

Now the lemma follows from a well-known convergence lemma due to A. McIntosh \cite{Mc}; see \cite[Lemma~2.1]{CDMY}, \cite[Theorem~D]{ADM96} and \cite[Proposition~5.4.1]{Haase}.
\end{proof}

\subsubsection{Mellin transform and Cowling's subordination} 
\label{s: Mellin}
Let $m\in L^{1}(\R_{+},
\wrt\lambda/\lambda)$. The {\it Mellin transform} $\cM m $ of $m$ is the  Fourier transform of $m\circ\exp$:
\[
\cM m (u):=\int^{\infty}_{0}m(\lambda)\lambda^{-iu}\frac{\wrt \lambda}{\lambda},\quad u\in\R.
\]
If $\cM m $ belongs to $L^{1}(\R)$, then we have the {\it Mellin inversion formula}
\begin{equation}
\label{eq: M.info}
m(\lambda)=\frac{1}{\pi}\int^{+\infty}_{-\infty}\cM m (u)\lambda^{iu}\wrt u,\quad \lambda>0.
\end{equation}
Let $A\in\cA(\Omega)$ be $\chi$-elliptic. Then, as we remarked in Section \ref{s: KP},  the operator $L_{A}$ has bounded imaginary powers in $L^{\chi}(\Omega)$ \cite{CD-Mixed} and we have the estimates \eqref{eq: impowers}.

Suppose now that $m\in L^{1}(\R_{+},\wrt\lambda/\lambda)$ is such that $\cM m\in L^{1}(\R)$ and that
the estimate
\[
\mod{\cM m (u)}\leq Ce^{-c|u|},\quad \forall u\in\R
\]
holds for some $C>0$ and $c>\theta_{\chi}$. 
Then we can use \eqref{eq: M.info} to extend $m$ holomorphically to $\bS_{\vartheta_{\chi}+\varepsilon}$ for small $\e>0$. 
Assuming the notation from the proof of Lemma \ref{l: L-P}, we have $m\in H^{\infty}(\bS_{\vartheta_{\chi}+\varepsilon})$. The McIntosh convergence lemma \cite{Mc} that we already used in the proof of Lemma~\ref{l: L-P}, together with 
\eqref{eq: M.info}
shows that
\begin{equation}\label{eq: subord}
m(L_{A})=\frac{1}{\pi}\int^{+\infty}_{-\infty}\cM m (u)L^{iu}_{A}\wrt u,
\end{equation}
where the integral converges in the strong operator topology of $\cB(L^{\chi}(\Omega))$. 

For details and proofs (in the self-adjoint case) see \cite{cowling, Meda, CD-mult}.

\medskip
We shall apply Cowling's subordination to the functions $\psi_{\alpha}$. 
Note that
$\cM\psi_{\alpha}(u)=\Gamma(\alpha-iu)$ for $u\in\R$,
hence, by Stirling's formula,
\begin{equation}\label{eq: Stirling}
\mod{
\cM \psi_{\alpha}
(u)}\leq C_{\alpha} (1+|u|)^{\alpha-1/2}e^{-\pi|u|/2},\quad u\in\R.
\end{equation}
Therefore $\cM \psi_{\alpha}\in L^1(\R)$, which shows that the inversion formula \eqref{eq: M.info} applies with $m=\psi_\alpha$.

Fix $A\in\cA(\Omega)$ and $\alpha>0$. By \eqref{eq: subord}, for every $f\in L^{2}(\Omega)$ and $t>0$ we have
\begin{equation}
\label{eq: subord concrete}
\psi_{\alpha}(tL_{A})f=\frac{1}{\pi}\int^{+\infty}_{-\infty}\left[\cM\psi_{\alpha}\right](u)\,t^{iu}L^{iu}_{A}f\wrt u,
\end{equation}
where the right-hand side should be interpreted as a Bochner integral that converges in $L^{2}(\Omega)$. 
Since both $T^{A}_{t}$ and $\nabla T^{A}_{t}$ are $L^{2}$-bounded, 
we may apply these operators to \eqref{eq: subord concrete}. 
Together with $\psi_{\alpha}(2tL_{A})=2^{\alpha}T^{A}_{t}\psi_{\alpha}(tL_{A})$, we then have
subordination estimates
\begin{align}
\mod{\psi_{\alpha}(2tL_{A})f} & \leq \frac{2^{\alpha}}{\pi}\int^{+\infty}_{-\infty}\mod{\cM \psi_{\alpha}(u)}\mod{T^{A}_{t}L^{iu}_{A}f}\wrt u
\label{eq: u1}\\
\mod{\nabla \psi_{\alpha}(2tL_{A})f} & \leq \frac{2^{\alpha}}{\pi}\int^{+\infty}_{-\infty}\mod{\cM \psi_{\alpha}(u)}\mod{\nabla T^{A}_{t}L^{iu}_{A}f}\wrt u\label{eq: u2}
\end{align}
almost everywhere in $\Omega$.

\subsubsection{Modified trilinear embedding}
The next result follows from Theorem \ref{t: trilinemb} and, at the same time, extends it.

\begin{proposition}
\label{p: newtrilinear}
Let $\Omega\subseteq\R^d$ be an open set and let the spaces $\oU,\oV,\oW$ be as in Section~\ref{s: Heian}.
Take $p,q,r\in(1,\infty)$ such that 
$1/p+1/q+1/r=1$. 
Suppose that the accretive matrices
$A,B,C:\Omega\rightarrow\C^{d\times d}$
are $\max\{p,q,r\}$-elliptic. 
Write $L_A=L_{A,\oU}$,  $L_B=L_{B,\oV}$ and $L_C=L_{C,\oW}$.
Let $\eta_{\alpha_{j}}$, $j=1,2,3$, denote either $\psi_{\alpha_{j}}$ with $\alpha_{j}>0$, or $\phi_{\alpha_{j}}$ with $\alpha_{j}\in\N_+$.
Then for every
$f\in\big(L^p\cap L^2\big)(\Omega)$,
$g\in\big(L^q\cap L^2\big)(\Omega)$
and
$h\in\big(L^r\cap L^2\big)(\Omega)$
we have
\begin{equation}
\label{eq: trilinembnew}
\int_0^\infty\int_{\Omega}
\left| \nabla \eta_{\alpha_{1}}(tL_{A}) f \right|
\left| \nabla \eta_{\alpha_{2}}(tL_{B})g \right|
\left| \eta_{\alpha_{3}}(tL_{C}) h \right|
\wrt x\wrt t
\ \leqsim\
\nor{f}_{p}\nor{g}_{q}\nor{h}_{r}.
\end{equation}
When $\Omega=\R^d$, the same conclusion holds
under milder assumptions, namely, 
%when
%\begin{itemize}
%\item
%$A$ is $p$-elliptic and $(1+p/q)$-elliptic,
%\item
%$B$ is $q$-elliptic and $(1+q/p)$-elliptic,
%\hfill
%{\rm({\Large $\star$})}
%\item
%$C$ is $r$-elliptic.
%\end{itemize}
{\rm({\Large $\star$})} from page \pageref{star}.
The implied embedding constants only depend on 
$\alpha_{1,2,3}$,
$p,q,r$ and $*$-ellipticity constants of $A,B,C$ alluded to in the theorem's assumptions.
\end{proposition}

\begin{proof}
By using \eqref{eq: sum} we reduce to prove \eqref{eq: trilinembnew} in the case when $\eta_{\alpha_{j}}$ is either $\psi_{\alpha}$ with  $\alpha>0$, or $e^{-\lambda}$. For simplicity, we only prove the estimate
\begin{equation}
\label{eq: trilinembbis}
\int_0^\infty\int_{\Omega}
\mod{\nabla \psi_{\alpha}(tL_{A}) f}
\left| \nabla T^{B}_{t}g \right|
\left| T^{C}_{t} h \right|
\wrt x\wrt t
\ \leqsim\
\nor{f}_{p}\nor{g}_{q}\nor{h}_{r};
\end{equation}
the other cases can be proved similarly.

By using, consecutively, \eqref{eq: u2} and Theorem \ref{t: trilinemb} 
we obtain
\[
\aligned
\int_0^\infty\int_{\Omega}
\mod{\nabla \psi_{\alpha}(tL_{A}) f}
&\left| \nabla T^{B}_{t}g \right|
\left| T^{C}_{t} h \right|
\wrt x\wrt t\\
&\leqsim_{\alpha}\int^{+\infty}_{-\infty}\mod{\cM \psi_{\alpha}(u)}\int_0^\infty\int_{\Omega}
\mod{\nabla T^{A/2}_{t}L^{iu}_{A}f}
\left| \nabla T^{B}_{t}g \right|
\left| T^{C}_{t} h \right|
\wrt x\wrt t\wrt u\\
&\leqsim_{\alpha}\norm{g}{q}\norm{h}{r}\int^{+\infty}_{-\infty}\mod{\cM \psi_{\alpha}(u)}\norm{L^{iu}_{A}f}{p}\wrt u
\endaligned
\]
and \eqref{eq: trilinembbis} follows by combining 
\eqref{eq: Stirling} with \eqref{eq: impowers}.
\end{proof}

Recall that the estimate \eqref{eq: impowers} was a consequence of \cite{CD-Mixed} and that, to the best of our knowledge, in the generality considered here, no analogous results are available.

\subsubsection{Integration by parts}
Let $\Omega\subseteq\R^d$ be an open set and let $\oU\subseteq H^{1}(\Omega)$ be one of the subspaces introduced in Section~\ref{s: Heian}. For every $\varepsilon>0$ define $\Phi_{\varepsilon}:\C\rightarrow \C$ by the rule 
$$
\Phi_{\varepsilon}(z):=\frac{z}{1+\varepsilon\sqrt{|z|^{2}+1}}.
$$
\begin{lemma}
\label{l: approx}
Let $f\in\oU$. For every $\varepsilon>0$ we have $\Phi_{\varepsilon}(f)\in \oU\cap L^{\infty}(\Omega)$, $\mod{\Phi_{\varepsilon}(f)}\leq |f|$ and $\mod{\nabla\Phi_{\varepsilon}(f)}\leq 2\mod{\nabla f}$. Moreover, as $\e\rightarrow0$ we get 
$\Phi_{\varepsilon}(f)\rightarrow f$ and $\nabla\Phi_{\varepsilon}(f)\rightarrow \nabla f$, in both cases almost everywhere on $\Omega$ and in $L^2(\Omega)$.
\end{lemma}
\begin{proof}
The function $\Phi_{\varepsilon}$ is of class $C^{\infty}$ and both $\Phi_\e$ and its gradient are bounded on $\C$.

When $\oU=H^{1}(\Omega)$, the lemma easily follows from the characterization of Sobolev spaces in terms of absolute continuity on lines \cite[Theorem 2.1.4]{Ziemer}, 
the chain rule for 
differentiable
functions and Lebesgue dominated convergence theorem. 
One can also apply directly to $\Phi_\e(f)$ a version of the chain rule for weak derivatives \cite[Theorem 2.1.11]{Ziemer}
adapted to complex functions.

When $\oU$ is of the 
types (a) or (c) from Section \ref{s: Heian}, all there is left to prove is that $\Phi_\e\in\oU$.
Consider a sequence $(\phi_{n})_{n\in\N}\in C^{\infty}_{c}(\R^{d}\backslash \Gamma)$ converging to $f$ in $H^{1}(\Omega)$. Then $\Phi_{\varepsilon}(\phi_{n})\in C^{\infty}_{c}(\R^{d}\backslash \Gamma)$ and the chain rule together with the Lebesgue dominated convergence theorem show that $\Phi_{\varepsilon}(\phi_{n})$ converges to $\Phi_{\varepsilon}(f)$ in $H^{1}(\Omega)$, as $n\rightarrow \infty$.
\end{proof}

\begin{lemma}
\label{l: parts} 
Let $p,q,r\in (1,\infty)$ such that
$1/p+1/q+1/r=1$ 
and let $A\in\cA(\Omega)$ be $p$-elliptic. 
Suppose that $u\in \Dom_{2}(L_{A})\cap\Dom_{p}(L_{A})$, $v\in \oU\cap L^{q}(\Omega)$ and $w\in \oU\cap L^{r}(\Omega)$. Then 
\begin{equation}\label{eq: u6pre}
\mod{\int_{\Omega}L_{A}u\cdot vw}\leq 2\Lambda\left(\int_{\Omega}\mod{\nabla u}\mod{\nabla v}\mod{w}+\int_{\Omega}\mod{\nabla u}\mod{\nabla w}\mod{v}\right).
\end{equation}
Suppose furthermore that 
\begin{equation}\label{eq: u5}
\mod{\nabla u}\left(\mod{w}\mod{\nabla v}+\mod{v}\mod{\nabla w}\right)\in L^{1}(\Omega).
\end{equation}
Then we have 
\begin{equation}
\label{eq: u6}
\int_{\Omega}L_{A}u\cdot vw=\int_{\Omega}\sk{A\nabla u}{\nabla\Bar v}w+\int_{\Omega}\sk{A\nabla u}{\nabla \Bar w}v.
\end{equation}
\end{lemma}

\begin{proof}
For $\varepsilon>0$ set $v_{\varepsilon}=\Phi_{\varepsilon}(v)$ and $w_{\varepsilon}=\Phi_{\varepsilon}(w)$. 
Lemma~\ref{l: approx} gives $v_\e,w_\e\in\oU\cap L^\infty(\Omega)$. 
It is well known that the Leibniz rule holds in $H^1\cap L^\infty$, see \cite[Theorem 4.4]{EG15}. This means that $v_\e w_\e\in H^1(\Omega)\cap L^\infty(\Omega)$ and
\begin{equation}
\label{eq: Sampa}
\nabla(v_\e w_\e)=v_\e\nabla w_\e +w_\e\nabla v_\e\,.
\end{equation}
We want to show a little bit more, namely, that 
\begin{equation}
\label{eq: algebra}
v_{\varepsilon}w_{\varepsilon}\in \oU. 
\end{equation}
If $\oU=H^1(\Omega)$, this has just been proved. Assume now that $\oU$ is of type (c) from Section \ref{s: Heian}. We saw in the proof of Lemma~\ref{l: approx} that, for some sequences $(\phi_{n})_{n\in\N},(\psi_{n})_{n\in\N}\in C^{\infty}_{c}(\R^{d}\backslash \Gamma)$,
$$
\left.
\begin{array}{ccc}
\Phi_\e(\phi_n)\rightarrow\Phi_\e(v)\\
\Phi_\e(\psi_n)\rightarrow\Phi_\e(w)
\end{array}
\right\}
\text{ in }H^1(\Omega).
$$
By Lemma~\ref{l: approx} and the Lebesgue dominated convergence theorem, 
it is from here not difficult to see that $\Phi_\e(\phi_n)\Phi_\e(\psi_n)\rightarrow\Phi_\e(v)\Phi_\e(w)=v_\e w_\e$ in $H^1(\Omega)$, as $n\rightarrow\infty$. This proves \eqref{eq: algebra}.

Since $u\in \Dom_{2}(L_{A})$, it follows from \eqref{eq: int by parts} and \eqref{eq: Sampa} that
\begin{equation}
\label{eq: interim}
\int_{\Omega}L_{A}u\cdot v_{\varepsilon}w_{\varepsilon}=\int_{\Omega}\sk{A\nabla u}{\nabla\Bar v_{\varepsilon}}w_{\varepsilon}+\int_{\Omega}\sk{A\nabla u}{\nabla \Bar w_{\varepsilon}}v_{\varepsilon}.
\end{equation}
By \eqref{eq: boundedness} and Lemma \ref{l: approx}, the right-hand side of \eqref{eq: interim} is bounded by the right-hand side of \eqref{eq: u6pre}, uniformly in $\e>0$.  By Lemma~\ref{l: approx} and the Lebesgue dominated convergence theorem, the left-hand side of \eqref{eq: interim} converges to the left-hand side of \eqref{eq: u6} as $\e\rightarrow0$.

Under the assumption \eqref{eq: u5}, the right-hand side of 
\eqref{eq: interim}
converges to the right-hand side of \eqref{eq: u6}, by Lemma~\ref{l: approx} and the Lebesgue dominated convergence theorem.
\end{proof}

\begin{corollary}\label{c: byparts}
Let $\Omega\subseteq\R^d$ be an open set and  
$\oU\subseteq H^{1}(\Omega)$ 
one of the closed subspaces introduced in Section~\ref{s: Heian}. 
Take $p,q,r\in (1,\infty)$ 
such that
$1/p+1/q+1/r=1$. 
Let $A\in\cA(\Omega)$ be $\max\{p,q,r\}$-elliptic and $\alpha_{1},\alpha_{2},\gamma\geq 0$. 
Furthermore, fix
$f\in (L^{p}\cap L^{2})(\Omega)$, $g\in (L^{q}\cap L^{2})(\Omega)$ and $h\in (L^{r}\cap L^{2})(\Omega)$. Then with 
$A_{\sf s}:=(A+A^T)/2$ and for
almost every $t>0$ we have 
\begin{align}
\int_{\Omega}\psi_{\alpha_{1}}(tL_{A})f\, \cdot &\, \psi_{\alpha_{2}}(tL_{A})g\cdot\overline{L_{A^{*}}\psi_{\gamma}(tL_{A^{*}})h}-\int_{\Omega}L_{A}\psi_{\alpha_{1}}(tL_{A})f\cdot\psi_{\alpha_{2}}(tL_{A})g\cdot\overline{\psi_{\gamma}(tL_{A^{*}})h}\nonumber\\
& -\int_{\Omega}\psi_{\alpha_{1}}(tL_{A})f\cdot L_A\psi_{\alpha_{2}}(tL_{A})g\cdot\overline{\psi_{\gamma}(tL_{A^{*}})h}\nonumber\\
=&-2
\int_{\Omega}
\sk{ A_{\sf s} 
     \nabla \psi_{\alpha_{1}} (tL_{A})f}
     {\nabla \overline{ \psi_{\alpha_{2}}  (t L_{A}) g 
                            }
     }  
\overline{ \psi_{\gamma} (tL_{A^{*}} )h}.
\label{eq: u7}
\end{align}
\end{corollary}
\begin{proof}
By Proposition~\ref{p: newtrilinear} and Tonelli's theorem, for almost every $t>0$ the function
\[
\aligned
\mod{\nabla \psi_{\alpha_{1}}(tL_{A})f}\mod{\nabla \psi_{\gamma}(tL_{A^{*}})h}\mod{\psi_{\alpha_{2}}(tL_{A})g} &+\mod{\nabla \psi_{\alpha_{2}}(tL_{A})g}\mod{\nabla \psi_{\gamma}(tL_{A^{*}})h}\mod{\psi_{\alpha_{1}}(tL_{A})f}\\
&+\mod{\nabla \psi_{\alpha_{1}}(tL_{A})f}\mod{\nabla \psi_{\alpha_{2}}(tL_{A})g}\mod{\psi_{\gamma}(tL_{A^{*}})h}
\endaligned
\] 
belongs to $L^{1}(\Omega)$. Now \eqref{eq: u7} follows integrating by parts each term in the left-hand side of \eqref{eq: u7} by means of Lemma~\ref{l: parts}.
\end{proof}

\subsection{Proof of Theorem~\ref{t: princ}}

\subsubsection{Step 1: Duality and density} 
As recalled in Section \ref{s: KP},
for every $\chi\in \{p_{1},p_{2},q_{1},q_{2},r\}$ the semigroup $(T^{A}_{t})_{t>0}$ extends to an analytic and contractive semigroup on $L^{\chi}(\Omega)$ in a cone of positive angle in $\C$. It is well-known \cite{Yosida60} and it can be easily proved by using functional calculus \cite[Example~3.4.6]{Haase} that strong continuity (in particular, analyticity) together with contractivity imply that  the fractional power $-L^{\beta}_{A}$ for $\beta\in (0,1)$ generates an analytic and contractive semigroup on $L^{\chi}(\Omega)$ for every $\chi\in \{p_{1},p_{2},q_{1},q_{2},r\}$. These semigroups are consistent, because they are subordinated to the consistent semigroups $(T^{A}_{t})_{t>0}$, see \cite{KatCar91,Yagi2010,EN}.
An approximation argument based on the above facts shows that $\Dom_{2}(L^{\beta}_{A})\cap \Dom_{\chi_{1}}(L^{\beta}_{A})\cap L^{\chi_{2}}(\Omega)$ is dense in $\Dom_{\chi_{1}}(L^{\beta}_{A})\cap L^{\chi_{2}}(\Omega)$
 and the same for $A$ replaced by $A^{*}$.
Here the domain of fractional power is endowed with the graph norm, 
$\nor{\cdot}_{\chi_1}+\| L_A^\beta\cdot \|_{\chi_1}$.
It also follows that the Hermitian dual $\big(L^{\beta}_{A^{*}}\big)^{*}$ of $L^{\beta}_{A^{*}}$ on $L^{\chi^{\prime}}(\Omega)$ coincides with $L^{\beta}_{A}$ on $L^{\chi}(\Omega)$.

In light of the considerations above, Theorem~\ref{t: princ} is equivalent to the following statement. 
\begin{proposition}
\label{p: dual inequality}
Under the assumptions of Theorem~\ref{t: princ},
for $f\in\Dom_{p_{1}}(L^{\beta}_{A})\cap \Dom_{2}(L^{\beta}_{A})\cap L^{p_{2}}(\Omega)$, $g\in\Dom_{q_{2}}(L^{\beta}_{A})\cap \Dom_{2}(L^{\beta}_{A})\cap L^{q_{1}}(\Omega)$ and $h\in \Dom_{r}(L^{\beta}_{A^{*}})\cap \Dom_{2}(L^{\beta}_{A^{*}})$ we have
\begin{equation*}
\label{eq: KatoPoncedual}
\mod{\int_{\Omega}fg\overline{L_{A^{*}}^\beta h}\wrt x}
\leqsim 
\left(\norm{L^{\beta}_{A}f}{p_{1}}\norm{g}{q_{1}}
+\norm{f}{p_{2}}\norm{L^{\beta}_{A}g}{q_{2}}\right)\norm{h}{r}.
\end{equation*}
\end{proposition}
So in the continuation we focus on proving the above result.

\subsubsection{Step 2: Littlewood-Paley decomposition} 
\label{s: LPdec}
Fix $\alpha\in\N$, $\alpha\geq 2$. By Lemma~\ref{l: L-P} we have
\[
\int_{\Omega}fg\overline{L^{\beta}_{A^{*}}h}\wrt x=-\int^{\infty}_{0}\int_{\Omega}\frac{\wrt}{\wrt t}
\left(
\phi_{\alpha}(tL_{A})f\cdot\phi_{\alpha}(tL_{A})g\cdot\overline{\phi_{\alpha}(tL_{A^{*}})L^{\beta}_{A^{*}}h}
\right)
\wrt x\wrt t.
\]
Observe that the left-hand side above does {\it not} depend on $\alpha$. Thus in principle we could have just worked with $\alpha=2$. Still, we opted for keeping a generic $\alpha$, in order to underline and understand better its role.

\medskip
It follows from the very definition of $\phi_{\alpha}$ or 
\eqref{eq: sumbis} that for $M\in \cA(\Omega)$ and $w\in L^{2}(\Omega)$ we have 
\[
-\frac{\wrt}{\wrt t}\phi_{\alpha}(tL_{M})w
=\frac{1}{(\alpha-1)!}\,L_{M}\psi_{\alpha-1}(tL_{M})w.
\]
Therefore, for the sake of proving Proposition \ref{p: dual inequality} we need to estimate three terms:
\[
\aligned
&I_{1}:=\int^{\infty}_{0}\int_{\Omega}L_{A}\psi_{\alpha-1}(tL_{A})f\cdot\phi_{\alpha}(tL_{A})g\cdot\overline{\phi_{\alpha}(tL_{A^{*}})L^{\beta}_{A^{*}}h}\wrt x\wrt t,\\
&I_{2}:=\int^{\infty}_{0}\int_{\Omega}L_{A}\psi_{\alpha-1}(tL_{A})g\cdot\phi_{\alpha}(tL_{A})f\cdot\overline{\phi_{\alpha}(tL_{A^{*}})L^{\beta}_{A^{*}}h}\wrt x\wrt t,\\
&I_{3}:=\int^{\infty}_{0}\int_{\Omega}\phi_{\alpha}(tL_{A})f\cdot\phi_{\alpha}(tL_{A})g\cdot\overline{L_{A^{*}}\psi_{\alpha-1}(tL_{A^{*}})L^{\beta}_{A^{*}}h}\wrt x\wrt t.
\endaligned
\]
We shall show that the critical term is $I_3$. We label it the {\it resonant term}, owing to its resemblance to the equally named terms $\Pi(f,g)$ from \cite{BF2018}.

\subsubsection{Estimating the terms $I_{1}$ and $I_{2}$}
Below we provide details for the estimate of $I_{1}$ only, because $I_{2}$ can be treated similarly.  

We expand $\phi_{\alpha}(tL_{A})$ and $\phi_{\alpha}(tL_{A^{*}})$ by means of \eqref{eq: sumbis} and reduce the estimate of $I_{1}$ to bounding from above a finite sum of double integrals of the type 
\begin{equation}
\label{eq: I_1 reduced}
\int^{\infty}_{0}\int_{\Omega}
L_{A}\psi_{\alpha-1}(tL_{A})f\cdot\psi_{k}(tL_{A})g\cdot\overline{\psi_{j}(tL_{A^{*}})L^{\beta}_{A^{*}}h}
\wrt x\wrt t,
\quad 0\leq j,k\leq\alpha-1.
\end{equation}
Now we can either 1) subordinate $\psi_{j}(tL_{A})$ and $\psi_{k}(tL_{A^{*}})$ to the imaginary powers and make a reduction to the case $k=j=0$, or 2) control each term separately, as we show below.

Observe that
\[
\psi_{\alpha-1}(tL_{A})f\cdot\overline{\psi_{j}(tL_{A^{*}})L^{\beta}_{A^{*}}h}=\psi_{\alpha-\beta-1}(tL_{A})L^{\beta}_{A}f\cdot\overline{\psi_{j+\beta}(tL_{A^{*}})h}.
\]
Hence in order to estimate \eqref{eq: I_1 reduced} it suffices to estimate terms of the type
\begin{equation}
\label{eq: intermezzo}
\int^{\infty}_{0}\int_{\Omega}L_{A}\psi_{\alpha-\beta-1}(tL_{A})L^{\beta}_{A}f\cdot\psi_{k}(tL_{A})g\cdot\overline{\psi_{j+\beta}(tL_{A^{*}})h}\wrt x\wrt t.
\end{equation}
We now integrate by parts the inner integral by using the first part of Lemma~\ref{l: parts} with $u=\psi_{\alpha-\beta-1}(tL_{A})L^{\beta}_{A}f$, $v=\psi_{k}(tL_{A})g$ and $w=\overline{\psi_{j+\beta}(tL_{A^{*}})h}$, and we apply the trilinear embedding from Proposition~\ref{p: newtrilinear} which gives the estimate 
\[
\mod{I_{1}}\leqsim  \|L^{\beta}_{A}f\|_{p_{1}}\norm{g}{q_{1}}\norm{h}{r}.
\]
By using the very same arguments we also get  
$\mod{I_{2}}\leqsim  \|L^{\beta}_{A}g\|_{q_{2}}\norm{f}{p_{2}}\norm{h}{r}$.

\begin{remark}
At this level of generality, we cannot estimate $I_{1}$ (or $I_{2}$) starting from \eqref{eq: intermezzo} and using two vertical square functions of the type
$$
\left(\int^{+\infty}_{0}\mod{\psi_{\gamma}(tL_{E})u}^{2}\frac{\wrt t}{t}\right)^{1/2},\quad \gamma>0.
$$
Indeed by \cite{CD-Mixed} the square functions are bounded, but for estimating the factor $\psi_{k}(tL_{A})g$ in \eqref{eq: intermezzo} we either need the boundedness of the maximal heat semigroup on $L^{\chi}(\Omega)$ for $\chi$-elliptic matrices, or the uniform boundedness of the semigroups in $L^{\infty}(\Omega)$ which, in our generality, are respectively unknown and false. 
\end{remark}
\subsubsection{Decomposing the resonant term $I_{3}$}
We expand $\phi_{\alpha}(tL_{A})$ by means of \eqref{eq: sumbis} and integrate by parts by means of Corollary~\ref{c: byparts}. In such a manner the estimate of $I_3$ reduces to estimating a finite number of terms of the type  
\[
\aligned
&J_{1}:=\int^{\infty}_{0}\int_{\Omega}L_{A}\psi_{j}(tL_{A})f\cdot\psi_{k}(tL_{A})g\cdot\overline{\psi_{\alpha-1}(tL_{A^{*}})L^{\beta}_{A^{*}}h}\wrt x\wrt t\\
&J_{2}:=\int^{\infty}_{0}\int_{\Omega}\psi_{j}(tL_{A})f\cdot L_{A}\psi_{k}(tL_{A})g\cdot\overline{\psi_{\alpha-1}(tL_{A^{*}})L^{\beta}_{A^{*}}h}\wrt x\wrt t
\endaligned
\]
and the term
\[
J_{3}:=\int^{\infty}_{0}\int_{\Omega}\sk{A_{\sf s}\nabla \phi_{\alpha}(tL_{A})f}{\nabla\overline{\phi_{\alpha}(tL_{A})g}}\overline{\psi_{\alpha-1}(tL_{A^{*}})L^{\beta}_{A^{*}}h}\wrt x\wrt t.
\]
\subsubsection{Estimating the terms $J_{1}$ and $J_{2}$}
We transfer the fractional power ``from $h$ to $f$'':
\[
\aligned
L_{A}\psi_{j}(tL_{A})f\cdot \psi_{k}(tL_{A}&)g\cdot \overline{\psi_{\alpha-1}(tL_{A^{*}})L^{\beta}_{A^{*}}h}\\
&
=t^{-\beta}L^{1-\beta}_{A}\psi_{j}(tL_{A})L^{\beta}_{A}f\cdot \psi_{k}(tL_{A})g\cdot \overline{\psi_{\alpha+\beta-1}(tL_{A^{*}})h}\\
&=\psi_{j+1-\beta}(tL_{A})L^{\beta}_{A}f\cdot \psi_{k}(tL_{A})g\cdot \overline{L_{A^{*}}\psi_{\alpha+\beta-2}(tL_{A^{*}})h}.
\endaligned
\]
We now proceed much as we did for $I_{1}$: we integrate over $\Omega$, we integrate by parts by means of the first part of Lemma~\ref{l: parts} and we apply Proposition~\ref{p: newtrilinear}. This gives the estimate
\[
\mod{J_{1}}\leqsim  \|L^{\beta}_{A}f\|_{p_{1}}\norm{g}{q_{1}}\norm{h}{r}.
\]
A similar argument gives
$\mod{J_{2}}\leqsim  \|L^{\beta}_{A}g\|_{q_{2}}\norm{f}{p_{2}}\norm{h}{r}$.

\subsubsection{The term $J_{3}$.} Write
$\psi_{\alpha-1}(tL_{A^{*}})L^{\beta}_{A^{*}}h=t^{-\beta}\psi_{\alpha+\beta-1}(tL_{A^{*}})h$.
Lemma \ref{l: L-P} (ii) gives
\[
\aligned
\sk{A_{\sf s}\nabla \phi_{\alpha}(tL_{A})f}{\nabla\overline{\phi_{\alpha}(tL_{A})g}}
&=-\int^{\infty}_{t}\frac{\wrt}{\wrt s}\sk{A_{\sf s}\nabla \phi_{\alpha}(sL_{A})f}{\nabla\overline{\phi_{\alpha}(sL_{A})g}}\wrt s\\
&=\int^{\infty}_{t}\sk{A_{\sf s}\nabla L_{A}\psi_{\alpha-1}(sL_{A})f}{\nabla\overline{\phi_{\alpha}(sL_{A})g}}\wrt s\\
&\hskip 10pt
+\int^{\infty}_{t}\sk{A_{\sf s}\nabla \phi_{\alpha}(sL_{A})f}{\nabla\overline{ L_{A}\psi_{\alpha-1}(sL_{A})g}}\wrt s.
\endaligned
\]
In accordance with this decomposition and \eqref{eq: boundedness} we have $\mod{J_{3}}\leq \Lambda(J^{\prime}_{3}+J^{\prime\prime}_{3})$, where
\[
\aligned
&J^{\prime}_{3}=\int^{\infty}_{0}t^{-\beta}\int^{\infty}_{t}\int_{\Omega}\mod{\nabla L_{A}\psi_{\alpha-1}(sL_{A})f}\mod{\nabla\phi_{\alpha}(sL_{A})g}\mod{\psi_{\alpha+\beta-1}(tL_{A^{*}})h}\wrt x\wrt s\wrt t,\\
&J^{\prime\prime}_{3}=\int^{\infty}_{0}t^{-\beta}\int^{\infty}_{t}\int_{\Omega}\mod{\nabla \phi_{\alpha}(sL_{A})f}\mod{\nabla L_{A}\psi_{\alpha-1}(sL_{A})g}\mod{\psi_{\alpha+\beta-1}(tL_{A^{*}})h}\wrt x\wrt s\wrt t.
\endaligned
\]
\subsubsection{Estimating the term $J^{\prime}_{3}$}
Writing
$
t^{-\beta}L_{A}\psi_{\alpha-1}(sL_{A})f
=(s/t)^{\beta}\psi_{\alpha-\beta}(sL_{A})L^{\beta}_{A}f\,s^{-1}
$
we get
\[
\aligned
J^{\prime}_{3}&=\int^{\infty}_{0}\int^{\infty}_{t}\left(\frac{s}{t}\right)^{\beta}\int_{\Omega}\mod{\nabla\psi_{\alpha-\beta}(sL_{A})L^{\beta}_{A}f}\mod{\nabla\phi_{\alpha}(sL_{A})g}\mod{\psi_{\alpha+\beta-1}(tL_{A^{*}})h}\frac{\wrt x\wrt s\wrt t}{s}\\
&=\int^{\infty}_{1}s^{\beta-1}\left(\int^{\infty}_{0}\int_{\Omega}\mod{\nabla\psi_{\alpha-\beta}(stL_{A})L^{\beta}_{A}f}\mod{\nabla\phi_{\alpha}(stL_{A})g}\mod{\psi_{\alpha+\beta-1}(tL_{A^{*}})h}\wrt x\wrt t\right)\wrt s.
\endaligned
\]
Through the subordination to the imaginary powers \eqref{eq: u1}, \eqref{eq: u2},  
we reduce the estimate of $J_3'$ to proving
\begin{equation}
\label{eq: final}
\aligned
\int^{\infty}_{1}s^{\beta-1}&\left(\int^{\infty}_{0}\int_{\Omega}\mod{\nabla T^{A}_{st}L^{iu_{1}}_{A}L^{\beta}_{A}f}\mod{\nabla T^{A}_{st}L^{iu_{2}}_{A}g}\mod{T^{A^{*}}_{t}L^{iu_{3}}_{A^{*}}h}\wrt x\wrt t\right)\wrt s\\
&\leqsim_{\beta} \norm{L^{iu_{1}}_{A}L^{\beta}_{A}f}{p}\norm{L_A^{iu_{2}}g}{q}\norm{L_{A^*}^{iu_{3}}h}{r}
\endaligned
\end{equation}
for all $u_{1},u_{2},u_{3}\in\R$ and $0<\beta<1/r^{\prime}$.
In order to prove \eqref{eq: final}, just apply Lemma \ref{l: final}.

This finishes the proof of Proposition \ref{p: dual inequality} and thus of Theorem \ref{t: princ}.
\qed

\section*{Acknowledgments}
A. Carbonaro was partially supported by the ``National Group for Mathematical Analysis, Probability and their Applications'' (GNAMPA-INdAM).

O. Dragi\v{c}evi\'c was partially supported by the Slovenian Research Agency (research grant J1-1690) and the Ministry of Higher Education, Science and Technology of Slovenia (research program Analysis and Geometry, contract no. P1-0291).

V. Kova\v{c} and K. A. \v{S}kreb were supported in part by the Croatian Science Foundation under the project UIP-2017-05-4129 (MUNHANAP).

This work was initiated during the second author's stay at the University of Zagreb in the Spring of 2019. He would like to thank the said institution for its kind hospitality.
He also expresses his gratitude to Christoph Thiele and Jonathan Bennett for inquiring (in Helsinki in June of 2015, and Matsumoto in February of 2016, respectively) about trilinear embedding theorems for elliptic operators, which stimulated the investigation undertaken in this paper.

\bibliography{Trilinear_ArXiv3-2}{}
 \bibliographystyle{plain}

\end{document}